\documentclass{amsart}
\usepackage{amssymb}
\usepackage{amsbsy}
\usepackage{amscd}
\usepackage{amsmath}
\usepackage{amsthm}
\usepackage{amsxtra}
\usepackage{latexsym}
\usepackage[mathscr]{eucal}
\usepackage{upgreek}
\usepackage{wasysym}
\usepackage[all,cmtip]{xy}
\usepackage[dvipdfmx]{xcolor}
\definecolor{rouge}{rgb}{0.85,0.1,.4}
\definecolor{bleu}{rgb}{0.1,0.2,0.9}
\definecolor{violet}{rgb}{0.7,0,0.8}

\newcommand{\tss}{\hspace{1pt}}
\newcommand{\ts}{\,}
\newcommand{\Caff}{{C}^{\kappa}}
\newcommand{\Fock}{\mc{F}}
\newcommand{\ch}{\on{ch}}
\newcommand{\Mi}{\Upupsilon}
\newcommand{\affQ}{\hat Q_{(0)}}
\newcommand{\affn}{\widehat{\mf{n}}}
\newcommand{\affh}{\widehat{\mf{h}}}

\newcommand{\finn}{\mf{n}}
\newcommand{\finb}{\mf{b}}
\newcommand{\finh}{\mf{h}}
\newcommand{\KL}{\mathbf{KL}}
\newcommand{\bra}{{\langle}}
\newcommand{\ket}{{\rangle}}

\newcommand{\Lam}{\Lambda}

\newcommand{\cprime}{$'$}
\newcommand{\on}{\operatorname}
\newcommand{\+}{\mathop{\oplus}}
\renewcommand{\*}{{\otimes}}
\newcommand{\mc}{\mathcal}
\newcommand{\mf}{\mathfrak}
\newcommand{\mb}{\mathbb}
\newcommand{\fing}{\mf{g}}

\newcommand{\affg}{\widehat{\mf{g}}}

\newcommand{\isomap}{{\;\stackrel{_\sim}{\to}\;}}
\newcommand{\Z}{\mathbb{Z}}
\newcommand{\C}{\mathbb{C}}
\newcommand{\N}{\mathbb{N}}
\newcommand{\Q}{\mathbb{Q}}
\newcommand{\W}{\mathscr{W}}
\newcommand{\ra}{\rightarrow}
\newcommand{\lam}{\lambda}

\newcommand{\Lie}{{\rm Lie}}

\newcommand{\vac}{{|0\rangle}}
\newcommand{\haru}{\operatorname{span}}
\newcommand{\Zhu}{{\rm Zhu}}
\newcommand{\im}{{\rm im}}
\newcommand{\id}{{\rm id}}
\newcommand{\HC}{\mc{HC}}

\def\leq{\leqslant}
\def\geq{\geqslant}

\DeclareMathOperator{\End}{End}
\DeclareMathOperator{\Spec}{Spec}
\DeclareMathOperator{\gr}{gr}
\DeclareMathOperator{\Ad}{Ad}
\DeclareMathOperator{\ad}{ad}
\DeclareMathOperator{\Hom}{Hom}

\numberwithin{equation}{section}

\newcommand{\inv}{^{-1}}

\newcounter{exercise}

\theoremstyle{theorem}
\newtheorem{thm}{Theorem}[section]
\newtheorem{Th}[thm]{Theorem}
\newtheorem{prp}[thm]{Proposition}
\newtheorem{lem}[thm]{Lemma}
\newtheorem{Lem}[thm]{Lemma}

\theoremstyle{definition}
\newtheorem{defn}[thm]{Definition}
\newtheorem{Def}[thm]{Definition}
\newtheorem{rmk}[thm]{Remark}
\newtheorem{Rem}[thm]{Remark}
\newtheorem{cor}[thm]{Corollary}
\newtheorem{Co}[thm]{Corollary}

\newtheorem{exm}[thm]{Example}
\newtheorem{exr}[exercise]{Exercise}

\theoremstyle{remark}

\title{Introduction to W-algebras and their representation theory}
\author{Tomoyuki Arakawa}
\address{Research Institute for Mathematical Sciences, Kyoto University,
 Kyoto 606-8502 JAPAN}
\email{arakawa@kurims.kyoto-u.ac.jp}

\begin{document}
\maketitle

\begin{abstract}
These are lecture notes from  author's mini-course during
 Session 1: ``Vertex algebras,
W-algebras, and application'' of INdAM Intensive research period
``Perspectives in Lie Theory'', at the  Centro di Ricerca Matematica Ennio De Giorgi,  Pisa, Italy. December 9, 2014 -- February 28, 2015.

\end{abstract}

\section{Introduction}
This note is based on lectures given at 
the  Centro di Ricerca Matematica Ennio De Giorgi,  Pisa, in Winter of 2014--2015.
They are aimed as an introduction to $W$-algebras and their representation theory.
Since $W$-algebras appear in many areas  of mathematics and physics
 there are certainly
 many other important topics untouched in the note, partly due to the limitation of the space and partly due to the author's incapability.
 
 The $W$-algebras 
 can be regarded as 
generalizations  of  affine Kac-Moody algebras and the Virasoro algebra.
They appeared  \cite{Zam85,FatLyk88,LukFat89} in the study of
the classification of  two-dimensional rational conformal field theories.
There are several ways to define $W$-algebras, but
it was Feigin and Frenkel \cite{FF90} 
who found
the most conceptual  definition  of principal $W$-algebras  that uses  the {\em quantized  Drinfeld-Sokolov reduction},
which is a version of Hamiltonian reduction.
There are a lot of works on $W$-algebras (see \cite{Bou95} and references therein)
mostly by physicists in 1980's and 1990's,
 but they were mostly on 
 principal $W$-algebras, that is, the $W$-algebras associated with
 principal nilpotent elements.
It was quite recent that Kac, Roan and Wakimoto \cite{KacRoaWak03} 
  defined the $W$-algebra $\W^k(\fing,f)$ associated with 
  a simple Lie algebra
  and its arbitrary nilpotent element  $f$
  by generalizing the method of quantized  Drinfeld-Sokolov reduction.

The advantage of the method of quantized  Drinfeld-Sokolov reduction
is its functoriality, in the sense that 
it gives rise to a functor from the category of representations of affine Kac-Moody algebras
and to the category of representations of $W$-algebras.
Since it is difficult to study $W$-algebras directly (as no presentation 
by generators and relations (OPE's) is known for a  general $W$-algebra),
in this note
we spend the most of our efforts in understanding this functor.

 Although our methods apply to much more general settings (\cite{Ara05,Ara08-a,Ara09b,A2012Dec,ArationalII})
we focus on
 the $W$-algebras associated with 
 Lie algebras $\fing$ of type $A$ and its principal nilpotent element that were originally defined by Fateev and Lykyanov \cite{FatLyk88}.
 They can be
 regarded as affinization of the center of the universal enveloping algebra of $\fing$
 via Konstant's Whittaker model \cite{Kos78} and Kostant-Sternberg's description \cite{KosSte87} of 
 Hamiltonian reduction via BRST cohomolgy, as explained in \cite{FF90}.
For this reason we start with a review of Kostant's results
and proceed to the
construction of BRST complex in the finite-dimensional setting
in \S \ref{sec:Review of Kostant's results}.
$W$-algebras are {\em not} Lie algebras, not even associated algebras in general, but {\em vertex algebras}.
In many cases a vertex algebra can be considered as a quantization of arc spaces of an affine Poisson scheme.
In \S \ref{sec:Arc spaces} we study this view point  that is useful in understanding $W$-algebras
and their representation theory.
In \S \ref{sec:Zhu} we study Zhu's algebras of vertex algebras
that connects 
$W$-algebras with {\em finite $W$-algebras} \cite{BoeTji93,Pre02}.
In \S \ref{sec:W-algbeas} we introduce $W$-algebras 
and study their basic properties.
In \S \ref{sec:rep-theory} we start studying representation theory of $W$-algebras.
In \S \ref{sec:irrep}
 we quickly review some fundamental results on  irreducible representations  of $W$-algebras obtained in \cite{Ara07}.
 One of the fundamental problems (at least mathematically) on $W$-algebras
was  the conjecture of Frenkel, Kac and Wakimoto \cite{FKW92} on the existence 
and construction of so called the {\em minimal models} of $W$-algebras,
which give rive to rational conformal field theories
as in the case of
 the integrable representations of affine Kac-Moody algebras and the minimal models of the Virasoro algebra.
  In \S \ref{sec:KW}
 we give an outline of the proof \cite{A2012Dec} of this conjecture.
 
 \subsection*{Acknowledgments}
 The author is grateful to the organizers 
 of ``Perspectives in Lie Theory''.
 He thanks 
  Naoki Genra and
 Xiao He 
 who wrote the first version of this 
 note. He would also like to thank Anne Moreau and Alberto De Sole for useful comments on the preliminary version of this note.
 His research is supported by JSPS KAKENHI Grant Numbers 25287004 and 26610006.
 



\section{Review of Kostant's results}
\label{sec:Review of Kostant's results}
\subsection{Companion matrices  and invariant polynomials}
\label{sec:1}
Let $G=GL_n(\C)$ be the general linear  group,
and 
let
$\fing=\mf{gl}_n(\C)$ be
the general linear Lie algebra consisting of $n\times n$ matrices.
The group $G$ acts on $\fing$ by the adjoint action:
$x\mapsto \Ad(g)x=gxg^{-1}$, $g\in G$.
 Let
 $\C[\fing]^G$ be the
 subring of the ring $\C[\fing]$ of polynomial functions 
 on $\fing$ consisting of $G$-invariant polynomials.

Recall that 
a matrix
 \begin{align}
A=\begin{pmatrix} 
0 & 0& \cdots  &0 & -a_1\\ 
  1  & 0  & \cdots & 0&-a_2\\
  0  & 1  & \cdot & 0&-a_3\\
 \vdots  & \vdots   & \ddots &\vdots   & \vdots \\
0 &0 &  & 1    & - a_n 
  \end{pmatrix}
  \label{eq:companion-matrix}
 \end{align}is called the
{\em companion matrix} of the polynomial $a_1+a_2 t+a_3 t^2+\dots +a_{n} t^n
\in \C[t]$ since
\begin{align}\label{eq:companion}
 \on{det}(t I-A)=
 a_1+a_2 t+a_3 t^2+\dots +a_{n} t^n
 .
\end{align}
Let
$\mc{S}$ be the affine subspace
of $\fing$ consisting of companion matrices
of the form
\eqref{eq:companion-matrix}.


 \begin{lem}\label{lem;characterization-of-elements-in-S}
For $A\in \fing$
the following conditions are equivalent.
\begin{enumerate}
 \item $A\in G\cdot \mc{S}$.
\item There exists a vector $v\in \C^n$ such that 
$v,Av,A^2v,\dots, A^{n-1}v$ are  linearly independent.
\end{enumerate}
 \end{lem}

  \begin{thm}\label{thm:invariant-polynomials}
The restriction map gives the isomorphism
\begin{align*}
 \C[\fing]^G\isomap \C[\mc{S}].
\end{align*}
 \end{thm}
  \begin{proof}
   Let
   $f\in \C[\fing]$
be  a $G$-invariant polynomial  
   such that
   $f|_{\mc{S}}=0$.
   Then clearly $f|_{G.\mc{S}}=0$.
On the other hand 
it follows from  Lemma \ref{lem;characterization-of-elements-in-S} that
   $G . \mc{S}$ is  a Zariski
   open  subset in $\fing$.
Therefore $f=0$.
   To see the surjectiveness define $p_1,\dots,p_n\in \C[\fing]^G$ by
\begin{align*}
 \det(tI-A)=t^n+p_1(A)t^{n-1}-\dots +p_n(A),
 \quad A\in \fing.
\end{align*}
By \eqref{eq:companion},
we have
 $\C[\mc{S}]=\C[p_1|_{\mc{S}}
,\dots , p_n|_{\mc{S}}]$.
This completes the proof.
  \end{proof}

Put \begin{align}
 f:=
\begin{pmatrix} 
  0 & & & \\ 
  1  & \ddots & &\\
 & \ddots   &  \ddots &  \\
 &  &   1    & 0 
\end{pmatrix} 
\in 
     \mc{S}.
     \label{eq:f}
\end{align}
Note that $f$ is a {\em nilpotent element} of  $\fing$, that is,
$(\ad f)^r=0$ for a sufficiently large $r$.
We have
\begin{align*}
 \mc{S}=f+\mf{a},
\end{align*}
where 
\begin{align*}
 \mf{a}= \Bigg\{ 
\begin{pmatrix} 
  0 &   \cdots &    0& \ast\\ 
\vdots      & & \vdots  &\ast\\
\vdots    &   &\vdots   & \ast \\
0 &    \cdots    &0 &\ast 
\end{pmatrix}  
\Bigg\}.
\end{align*}

Let $\mf{b}$, $\mf{n}$ be the subalgebras of $\fing$ defined by
\begin{align*}
 \mf{b}= \Bigg\{ 
\begin{pmatrix} 
  \ast &  &   & \\ 
    & \ddots & \ast &  \\
  &   0 &  \ddots  &  \\
 &  &   & \ast
\end{pmatrix}  
\Bigg\},\quad
\mf{n}= \Bigg\{ 
\begin{pmatrix} 
  0 &  &   & \\ 
    & \ddots & \ast &  \\
  &   0 &  \ddots  &  \\
 &  &   & 0
\end{pmatrix}  
\Bigg\} \subset \mf{b},
\end{align*}
and let  $N$ be the unipotent subgroup of $G$ corresponding to $\mf{n}$, i.e., 
\begin{equation}
N = \Bigg\{ 
\begin{pmatrix} 
  1 &  &   & \\ 
    & \ddots & \ast &  \\
  &   0 &  \ddots  &  \\
 &  &   & 1
\end{pmatrix}  
\Bigg\}. 
\end{equation}

Let $(~|~)$ be the invariant inner product of $\fing$ defined by
 $(x|y)=\on{tr}(xy)$.
This gives a $G$-equivariant isomorphism
$\fing\isomap \fing^*$.

Define $\chi \in \mf{n}^\ast$ by \begin{align*}
				  \chi(x)=( f|x)\quad\text{for $x\in \mf{n}$}.
				 \end{align*}
Note that $\chi$  is  a character of $\mf{n}$, that is, $\chi([\mf{n},
\mf{n}])=0$.
Hence $\chi$ defines a one-dimensional representation of $N$.

Consider the restriction map
\begin{align*}
\mu:\fing^*\ra \mf{n}^*.
\end{align*}
Then
\begin{align*}
 \mu^{-1}(\chi)=\chi+\finn^\bot\cong f+\mf{b}.
\end{align*}
Here $\fing$ is identified with $\fing^*$ via $(~|~)$.
Since 
$\mu$ is $N$-equivariant and $\chi$ is a one-point orbit of $N$,
it follows that $f+\mf{b}$ is stable under the action of $N$.

 \begin{thm}[Kostant \cite{Kos78}]\label{them:Kostant}
The adjoint action gives the isomorphism
\begin{align*}
 N\times \mathcal{S}\isomap f+\mf{b},\quad (g,x)\mapsto \Ad(g)x
\end{align*}
of affine varieties.
 \end{thm}
 \begin{proof}
It is not
 difficult to see that
 the adjoint action
  gives the bijection
$
 N\times \mathcal{S}\isomap f+\mf{b}
$.
  Since it is  a morphism of
irreducible  varieties and $f+\mf{b}$ is normal,
the assertion follows from Zariski's Main Theorem (see e.g.,
\cite[Corollary 17.4.8]{TauYu05}).
 \end{proof}

  \begin{cor}\label{cor:iso-first}
   The restriction map gives the isomorphisms
 \begin{align*}
\C[\fing]^G
   {\isomap} \C[f+\mf{b}]^N
   {\isomap} \C[\mc{S}].
 \end{align*}
  \end{cor}
 \begin{proof}By Theorem \ref{them:Kostant},
we have
\begin{align*}
\C[f+\mf{b}]^N \cong \C[N]^N\otimes \C[\mc{S}] \cong \C[\mc{S}].
\end{align*}Hence the assertion follows from 
Theorem \ref{thm:invariant-polynomials}.
 \end{proof}

\subsection{Transversality of $\mc{S}$ to $G$-orbits}
\begin{lem}\label{lem:tras-at-f}
The affine spaces $\mc{S}$ and $f+\mf{b}$ intersect transversely at $f$ to $\Ad G\cdot
 f$.
\end{lem}
\begin{proof}
We need to show that 
\begin{equation}
T_f\mf{g} =T_f\mc{S} + T_f (\Ad G\cdot f)
\end{equation}
But $T_f\mf{g} \cong \mf{g}$, $ T_f\mc{S}\cong \mf{a}$,  $T_f (\Ad
 G\cdot f)\cong [\mf{g}, f]$.
The assertion follows since  $\mf{g}=\mf{a}+[\mf{g}, f]$.
\end{proof}

Using the Jacobson-Morozov theorem, we can embed $f$ into an
$\mf{sl}_2$-triple $\{e, f, h\}$ in $\mf{g}$. 
Explicitly, 
we can choose the following elements for $e$  and $h$:
\begin{align}
 e=\sum_{i=1}^{n-1}i(n-i)e_{i,i+1},\quad 
h=\sum_{i=1}^n (n+1-2i)e_{i,i},
\label{eq:sl2-triple}
\end{align}
where
$e_{i,j}$ denotes the standard basis element of $\fing=\on{Mat}_n(\C)$.

The embedding
$\mf{sl}_2=\haru_{\C}\{e,h,f\}\ra \fing$
exponents to a homomorphism $SL_2\ra G$.
Restricting it to the torus $\C^*$ consisting of diagonal matrices 
we obtain a one-parameter subgroup
 $\gamma:\C^*\ra G$.
Set
\begin{align}
 \rho :\C^*\ni t \longmapsto t^2 \Ad \gamma(t)\in GL(\fing).
 \label{eq:C-star-action}
\end{align}
Then
\begin{align*}
 \rho(t)(f+\sum_{i\leq j}c_{ij}e_{i,j})=f+\sum_{i\leq j}t^{2(i-j+1)}c_{ij}e_{i,j}.
\end{align*}
Thus it 
define a $\C^\ast$-action on $\mf{g}$
that preserves $f+\mf{b}$  and
 $\mc{S}$.
This action on $f+\mf{b}$  and
 $\mc{S}$
 contracts to $f$, that is, $\rho(t)x \rightarrow f$ when $t \rightarrow 0$.

\begin{prp}\label{prp:transversality}
The affine space $f+\mf{b}$  (resp.\ $\mc{S}$) intersects $\Ad G\cdot
 x$ transversely at any point $x\in f+\mf{b}$ (resp. $x\in \mc{S}$).
\end{prp}
 \begin{proof}
By Lemma  \ref{lem:tras-at-f}
the intersection of $f+\mf{b}$ with $\Ad G$-orbits is trasversal 
at each point in some open neighborhood of $f$ in $f+\mf{b}$.
By the contracting $\C^*$-action $\rho$,
it follows that the same is true for all points of $f+\mf{b}$.
 \end{proof}

 \subsection{The trasversal slice $\mc{S}$ as a reduced Poisson variety }
 \label{subs:trasversal-slice}
The affine variety
$\fing^*$ 
is equipped with the Kirillov-Kostant Poisson
structure:
the Poisson algebra structure of 
$\C[\fing^*]$ is given by
\begin{align*}
\{x,y\}=[x,y]\quad\text{for } x,y\in \fing\subset \C[\fing^*].
\end{align*}

Consider the restriction map $\mu:\fing^*\ra \mf{n}^*$,
which is a {\em moment map} for the $N$-action on
$\fing^*$.
That is,
$\mu$ is a regular $N$-equivariant morphism
that gives the following commutative diagram of Lie algebras:
\begin{align*}
 \xymatrix{
  & \mf{n}\ar[d]^{} \ar[dl]_-{\mu^*} \\
  \C[\fing^*] \ar[r]^{} &\on{Der}\C[\fing^*]
}&
\end{align*}
Here $\mu^*:\mf{n}\ra \fing\subset \C[\fing^*]$ is the pullback  map,
the map $\C[\fing^*]\ra \on{Der}\C[\fing^*]$ is given by 
$\phi\mapsto \{\phi,?\}$,
and 
$\mf{n}\ra\on{Der}\C[\fing^*]$ is the Lie algebra homomorphism
induced by the coadjoint action of $G$ on $\fing^*$.

The transversality statement  of
Proposition
\ref{prp:transversality} for $f+\mf{b}$
is equivalent to that
$\chi$ is 
 a regular value of $\mu$.
 By Theorem \ref{them:Kostant},
the action of $N$  on $\mu^{-1}(\chi)=\chi+\finn^{\bot}$ is free 
and
\begin{align*}
 \mc{S}\cong \mu^{-1}(\chi)/N.
\end{align*}
Therefore $\mc{S}$
has
the structure of the {\em reduced Poisson variety},
obtained from $\fing^*$ by the Hamiltonian reduction.

The Poisson structure of $\mc{S}$ is described as follows.
Let
\begin{align*}
 I_{\chi}=\C[\fing^*]\sum_{x\in \finn}(x-\chi(x)),
\end{align*}
so that
 \begin{align*}
\C[\mu^{-1}(\chi)]=\C[\fing^*]/I_{\chi}.
 \end{align*}
 Then
 $\C[\mc{S}]$ can be identified as  the subspace of
 $\C[\fing^*]/I_{\chi}$ consisting of all cosets $\phi+\C[\fing^*]I_{\chi}$ such that
 $\{x,\phi\}\in \C[\fing^*]I_{\chi}$ for all $x\in \finn$.
 In this realization,
 the Poisson structure on $\C[\mc{S}]$ is defined by the formula
 \begin{align*}
  \{\phi+\C[\fing^*]I_{\chi}, \phi'+\C[\fing^*]I_{\chi}\}=\{\phi,\phi'\}+\C[\fing^*]I_{\chi}
 \end{align*}
 for $\phi,\phi'$ such that
 $\{x,\phi\}, \{x,\phi'\}\in \C[\fing^*]I_{\chi}$ for all $x\in \finn$.

\begin{prp}
We have  the isomorphism
 $\C[\fing^*]^G\isomap \C[\mc{S}]$ as Poisson algebras.
 In particular the Poisson structure of $\mc{S}$ is trivial.
\end{prp}
\begin{proof}
The restriction map
 $\C[\fing^*]^G\isomap \C[\mc{S}]$
 (see Corollary  \ref{cor:iso-first}) is obviously a homomorphism of  Poisson algebras.
\end{proof}
 
In the next subsection we shall describe
the above Hamiltonian reduction
in more factorial way,
in terms of the {\em BRST cohomology}
(where  BRST refers to the physicists Becchi, Rouet, Stora and Tyutin) 
for later purpose.

\subsection{BRST reduction}
Let
$Cl$ be the 
{\em Clifford algebra} associated with
the vector space $\mf{n}\oplus \mf{n^\ast}$ and its  non-degenerate
bilinear form $(\cdot|\cdot)$ defined by $(f+x|g+y)=f(y)+g(x)$ for $f,
g\in \mf{n^\ast}, x, y \in \mf{n}$. 
Namely,
$Cl$ is the unital $\C$-superalgebra
that is
isomorphic to
$\Lambda(\mf{n})\otimes \Lambda(\mf{n^\ast})$ 
as $\C$-vector spaces,
the natural embeddings
$\Lambda(\mf{n})\hookrightarrow Cl$,
$\Lambda(\mf{n}^*)\hookrightarrow Cl$ are 
homogeneous homomorphism of superalgebras,
and
\begin{align*}
[x,f]
=f(x)\quad x\in\mf{n}\subset \Lam(\finn),\ f\in
 \mf{n}^*\subset \Lam(\finn^*).
\end{align*}
(Note that $[x,f]=xf+fx$ since $x,f$ are odd.)


Let $\{x_\alpha\}_{\alpha\in \Delta_+}$ be  a basis of $\mf{n}$,   $\{x_\alpha^\ast\}_{\alpha\in \Delta_+}$ the dual basis of $\mf{n^\ast}$, and $c_{\alpha, \beta}^{\gamma}$  the structure
 constants of $\mf{n}$,
that is,
$[x_\alpha,x_\beta]=\sum_{\alpha,\beta}^\gamma c_{\alpha\beta}^\gamma x_\gamma$.

\begin{lem}\label{lem:Lie-alg-hom-Clifford}
The following map gives a Lie algebra homomorphism.
 \begin{align*}
\begin{split}
\rho : \mf{n} &\longrightarrow Cl \\
x_\alpha &\longmapsto \sum_{
\beta, \gamma\in \Delta_+}c_{\alpha, \beta}^\gamma x_\gamma x_\beta^\ast.
\end{split}
\end{align*}
We have
\begin{align*}
 [\rho(x),y]=[x,y]\in \mf{n}\subset Cl\quad\quad\text{for }x,y\in \mf{n}.
\end{align*}
\end{lem}

Define an increasing filtration on $Cl$ by setting $Cl_p:= \Lambda^{\leq
p}(\mf{n})\otimes \Lambda(\mf{n}^\ast)$.
We have 
\begin{align*}
0= Cl_{-1} \subset Cl_0 \subset Cl_1 \cdots 
\subset Cl_{N}=Cl,
\end{align*}
where $N=\dim \finn=\frac{n(n-1)}{2}$,
and
\begin{align}
Cl_p\cdot Cl_q \subset Cl_{p+q}, \quad  [Cl_p, Cl_q] \subset Cl_{p+q-1}.
\label{eq:filtration-of-Clifford}
\end{align}
Let $\overline{Cl}$ be
 its associated graded algebra:
\begin{align*}
\overline{Cl}
:= \gr Cl =\bigoplus_{p\geq 0} \dfrac{Cl_p}{Cl_{p-1}}.
\end{align*}
By \eqref{eq:filtration-of-Clifford},
$\overline{Cl}$ is naturally a graded Poisson superalgebra, called the {\em classical Clifford algebra}.

We have
$\overline{Cl}=\Lambda(\mf{n})\* \Lambda(\mf{n}^*)$
as a commutative superalgebra. 
Its  Poisson (super)bracket
is given by
\begin{align*}
& \{x,f\}=f(x),\quad x\in \mf{n}\subset \Lambda(\mf{n}),
\ f\in \mf{n}^*\subset \Lambda(\mf{n}^*),\\
&\{x,y\}=0, \quad x,y\in \mf{n}\subset \Lambda(\mf{n}),
\quad \{f,g\}=0, \ f,g\in \mf{n}^*\subset \Lambda(\mf{n}^*).
\end{align*}


 \begin{lem}\label{lem:invariant-of-classical-Poisson}
We have $\overline{Cl}^{\mf{n}}=\Lambda(\mf{n})$,
where
  $\overline{Cl}^{\mf{n}}
:=\{w\in \overline{Cl}\mid  \{x,w \}=0,
  \forall x\in \mf{n}\}$.
 \end{lem}

The Lie algebra homomorphism $\rho : \mf{n} \longrightarrow Cl_1\subset Cl$
 induces a 
Lie algebra homomorphism
\begin{align}
\bar{\rho}:=\sigma_1\circ \rho : \mf{n} \longrightarrow \overline{Cl},
\label{eq:Lie-alg-hom-Poisson}
\end{align}
where $\sigma_1$ is the projection
$Cl_1\ra Cl_1/Cl_{0}\subset \gr Cl$. 
We have
\begin{align*}
 \{\bar \rho(x),y\}=[x,y]\quad\text{for }x,y\in \mf{n}.
\end{align*}

Set
\begin{align*}
 \bar C(\fing)=\C[\fing^*]\* \overline{Cl}.
\end{align*}
Since it is a tensor product of Poisson superalgebras,
$\bar C(\fing)$ is  naturally a Poisson superalgebra.
\begin{lem}\label{lem:moment}
The following map gives  a Lie algebra homomorphism:
\begin{equation*}
\begin{split}
\bar{\theta}_{\chi}: \mf{n} &\longrightarrow \bar{C}(\mf{g})\\
x &\longmapsto (\mu^*(x)-\chi(x))\otimes 1+ 1\otimes \bar{\rho}(x),
\end{split}
\end{equation*}
that is,
$\{\bar\theta_{\chi}(x),\bar\theta_{\chi}(y)\}=\bar\theta_{\chi}([x,y])$ for $x,y\in \mf{n}$.
\end{lem}

Let $\bar C(\fing)=\bigoplus_{n\in \Z}\bar C^n(\fing)$ be the
$\Z$-grading defined by
 $\deg \phi \*1=0$ ($\phi\in \C[  \mf{g}^*]$), $\deg
 1\* f=1$ ($f\in \mf{n}^\ast$), $\deg 1\* x=-1$ ($x\in \mf{n}$).
We have
\begin{align*}
 \bar C^n(\fing)=
\C[\fing^*]\otimes (\bigoplus_{j-i=n}\Lambda^i (\mf{n})\otimes \Lambda^j(\mf{n}^\ast)).
\end{align*}

 \begin{lem}[{\cite[Lemma 7.13.3]{BeiDri96}}]
  \label{lem:barQ}
There exists a unique element
$ \bar{Q} \in \bar{C}^1(\mf{g})$ such that 
\begin{align*}
\{\bar{Q}, 1\otimes
 x\}=\bar{\theta}_{\chi}(x)\quad\text{for } x \in \mf{n}.
\end{align*}
 We have $\{\bar{Q}, \bar{Q}\}=0$.
\end{lem}
\begin{proof}
Existence. 
It is straightforward to see that the element
\begin{align*}
\bar{Q}=\sum_\alpha(x_\alpha-\chi(x_\alpha))\otimes x_\alpha^\ast-1\otimes
 \dfrac{1}{2}\sum_{\alpha,\beta,\gamma}c_{\alpha, \beta}^\gamma x_\alpha^\ast x_\beta^\ast x_\gamma
\end{align*}
satisfies the condition.

Uniqueness. 
Suppose that  $\bar Q_1, \bar Q_2\in \bar{C}^1(\fing)$ satisfy the
 condition. 
Set $R=Q_1-Q_2\in \bar C^1(\fing)$.
Then $\{R, 1\* x\}=0$,
and so,
$R\in \C[\fing^*]\* \overline{Cl}^{\mf{n}}$.
But by Lemma \ref{lem:invariant-of-classical-Poisson},
$\overline{Cl}^{\mf{n}}\cap \overline{Cl}^1=0$.
Thus  $R=0$ as required.

To show that 
$\{\bar Q,\bar Q\}=0$, observe that
\begin{align*}
\{1\otimes x, \{1\otimes y, \{\bar{Q}, \bar{Q}\}\}
\}=0, \quad \forall x, y \in \mf{n}
\end{align*}
(note that $\bar Q$ is odd).
Applying  Lemma \ref{lem:invariant-of-classical-Poisson} twice,
we get that 
$\{\bar Q,\bar Q\}=0$.
\end{proof}

 Since $\bar{Q}$ is odd, 
Lemma \ref{lem:barQ} implies that
\begin{align*}
 \{\bar Q,\{\bar Q,a\}\}=\frac{1}{2}\{\{\bar Q,\bar Q\},a\}=0
\end{align*}
for any $a\in \bar C(\fing)$.
That is,
 $\ad \bar{Q}:=\{\bar{Q}, \cdot\}$ satisfies that 
\begin{align*}
(\ad \bar{Q})^2=0.
\end{align*}
Thus, $(\bar{C}(\mf{g}), \ad \bar{Q})$ is a {\em differential graded Poisson superalgebra}.
Its cohomology $H^\bullet(\bar{C}(\mf{g}), \ad \bar{Q})
=\bigoplus\limits_{i\in \Z}H^i(\bar C(\mf{g}),\ad \bar Q)$
inherits
a graded Poisson superalgebra structure from $\bar C(\fing)$.

According to Kostant and Sternberg \cite{KosSte87}
the Poisson structure of $\C[\mc{S}]$ may be described through the following isomorphism:
 \begin{thm}[\cite{KosSte87}]\label{thm:KS-classical}
We have $H^i(\bar{C}(\mf{g}), \ad \bar{Q})=0$
 for $i\ne 0$
 and 
  \begin{align*}
  H^0(\bar{C}(\mf{g}), \ad \bar{Q})\cong \C[\mc{S}]
 \end{align*}
as Poisson algebras.
\end{thm}

\begin{proof}
Give a bigrading on $\bar C:=\bar{C}(\mf{g})$ by setting 
$$\bar{C}^{i, j}=\C[\mf{g}^\ast]\otimes \Lambda^i(\mf{n}^\ast)\otimes
 \Lambda^{-j}(\mf{n}),$$
so that $\bar C=\bigoplus\limits_{i\geq 0,j\leq 0}\bar C^{i,j}$.

Observe that $\ad \bar Q$ decomposes as  $\ad \bar{Q} = d_++d_-$ 
 such that
 \begin{align}
d_-(\bar{C}^{i,j})\subset \bar{C}^{i, j+1},\quad
  d_+(\bar{C}^{i,j})\subset \bar{C}^{i+1, j}.
  \label{eq:dpm}
 \end{align}
Explicitly,
 we have
 \begin{align*}
  &  d_-=\sum_i(x_i-\chi(x_i))\otimes \ad x_i^\ast,\\
  & d_+=\sum_i \ad
 x_i\otimes x_i^\ast -1\otimes \dfrac{1}{2}\sum_{i, j, k}c_{i,
 j}^kx_i^\ast x_j^\ast \ad x_k +\sum_i 1\otimes \bar{\rho}(x_i)\ad
 x_i^\ast.
 \end{align*}
 Since $\ad \bar{Q}^2=0$, \eqref{eq:dpm}
 implies that
$$d_-^2=d_+^2=[d_-, d_+]=0.$$ 
It follows that there exists a spectral sequence 
$$E_r \Longrightarrow H^\bullet(\bar{C}(\mf{g}), \ad \bar{Q})$$ 
such that 
 \begin{align*}
  & E_1^{\bullet,q}=H^q(\bar{C}(\mf{g}), d_-)
 =H^q(\C[\fing^*]
 \* \Lam (\finn),d_- )\* \Lam^\bullet(\finn^*),\\
 & E_2^{p,q}=H^{p}(H^q(\bar C(\mf{g}), d_-), d_+).
 \end{align*}
Observe that  
 $(\bar{C}(\mf{g}), d_-)$ is identical to the Koszul complex $\C[\fing^*]$ associated with the sequence
$x_1-\chi(x_1),x_2-\chi(x_2.)\dots,
 x_{N}-\chi(x_{N})$ tensorized  with $\Lam(\finn^*)$.
 Since
 $\C[\mu^{-1}(\chi)]={\C[\mf{g}^\ast]}/{\sum_i\C[\mf{g}^\ast](x_i-\chi(x_i)})$,
 we get that
\begin{align*}
						       H^i(\bar{C}(\mf{g}), d_-)=\begin{cases}
										  \C[\mu^{-1}(\chi)]
						       \otimes \Lambda(\mf{n}^\ast),&\text{if}\quad i=0
										 \\0,&\text{if}\quad i\neq 0.\end{cases}
\end{align*}
Next, notice that
 $(H^{0}(C(\mf{g}), d_-), d_+)$ is identical to the Chevalley complex for the Lie algebra cohomology $H^\bullet(\mf{n}, \C[\mu^{-1}(\chi)])$.
Therefore Theorem \ref{them:Kostant} gives that
 $$H^i(H^{\bullet}(C(\mf{g}), d_-), d_+)= 
 \begin{cases}
\C[\mc{S}], & i=0\\
0, & i\neq 0.
 \end{cases}$$
Hence  the spectral sequence collapses at $E_2=E_{\infty}$ and
 we get that
 $H^i(\bar C(\fing), \ad \bar Q)
 =0$ for $i\ne 0$.
Moreover, there is an isomorphism
 \begin{align*}
  H^0(\bar C(\fing), \ad \bar Q)\isomap  H^0(H^0(\bar C(\fing),d_-),d_+)=\C[\mc{S}],\quad [c]\mapsto [c].
 \end{align*}
 This completes the proof.
\end{proof}

\begin{thm}
The natural map $\C[\mf{g}^*]^G \longrightarrow H^0(\bar{C}(\mf{g}), \ad \bar{Q})$ defined by sending $p$ to $p\otimes 1$ is an isomorphism of Poisson algebras.\end{thm}

\begin{proof}
It is clear that
the map is a well-defined homomorphism of Poisson algebras
since $\C[\fing^*]^G$ is the Poisson center of $\C[\fing^*]$.
The assertion follows
from the commutativity of the following diagram.
\begin{align*}
 \xymatrix{
 & \C[\fing^*]^G\ar[d]^{} \ar[dl]_-{\cong } \\
 \C[\mc{S}]
 &  
H^0(\bar C(\fing),\ad \bar Q).
 \ar[l]_-{\cong}^{}
}
\end{align*}

\end{proof}

\subsection{Quantized Hamiltonian reduction}
\label{subsection:Quantized Hamiltonian reduction}
We shall now quantize the above construction
following \cite{KosSte87}.

Let $\{U_i(\fing)\}$ be the
PBW filtration of 
the universal enveloping algebra $U(\fing)$ of $\fing$,
that is,
$U_i(\fing)$ is the subspace of $U(\fing)$ spanned by the products
of at most $i$ elements of $\fing$.
Then
\begin{align*}
 0=U_{-1}(\fing)\subset U_0(\fing)\subset U_1(\fing)\subset  \dots,\quad
 U(\fing)=\bigcup_i U_i(\fing),\\
 U_i(\fing)\cdot U_j(\fing)\subset U_{i+j}(\fing),\quad
 [U_i(\fing),U_j(\fing)]\subset U_{i+j-1}(\fing).
\end{align*}
The  associated graded space
$\gr U(\fing)=\bigoplus_{i\geq 0}U_i(\fing)/U_{i-1}(\fing)$
is naturally a Poisson algebra, and
the PBW Theorem  states that
\begin{align*}
 \gr U(\fing)\cong \C[\fing^*]
\end{align*}
as Poisson algebras.
Thus,
$U(\fing)$ is a quantization of
$\C[\fing^*]$.

Define
\begin{align*}
 C(\fing)=U(\fing)\* Cl.
\end{align*}
It is naturally a $\C$-superalgebra,
where $U(\fing)$ is considered as a purely even subsuperalgebra.
The filtration of $U(\fing)$ and $Cl$ induces the filtration
of $C(\fing)$:
$C_p(\fing)=\sum_{i+j\leq p}U_i(\fing)\* Cl_j$,
and we have
\begin{align*}
 \gr C(\fing)\cong \bar C(\fing)
\end{align*}
as Poisson superalgebras.
Therefore, $C(\fing)$ is a quantization of $\bar C(\fing)$.

Define the $\Z$-grading $C(\fing)=\bigoplus\limits_{n\in \Z}C^n(\fing)$
by setting
$\deg (u\*1)=0$ ($u\in U(\fing)$),
$\deg (1\*f) =1$ ($f\in \finn^*$),
$\deg (1\*x) =-1$ ($x\in \finn$).
Then
\begin{align*}
 C^n(\fing)=U(\mf{g})\otimes (\bigoplus_{j-i=n}\Lambda^i (\mf{n})\otimes \Lambda^j(\mf{n}^\ast)).
\end{align*}
\begin{lem}\label{lem:quantized-moment-map}
The following map defines a Lie algebra homomorphism.
\begin{equation*}
\begin{split}
\theta_{\chi}: \mf{n} &\longrightarrow C(\mf{g})\\
x &\longmapsto (x-\chi(x))\otimes 1+ 1\otimes \rho(x)
\end{split}
\end{equation*}
 \end{lem}

\begin{lem}[{\cite[Lemma 7.13.7]{BeiDri96}}]\label{lem:Q}
There exists a unique element
 $Q \in C^1(\mf{g})$
such that
 \begin{align*}
[Q, 1\otimes x]=\theta_{\chi}(x), \quad \forall x \in \mf{n}.
 \end{align*}
We have $Q^2=0$.
 \end{lem}
  \begin{proof}
   The proof 
is similar to that of Lemma \ref{lem:barQ}.
In fact   the element $Q$ is 
explicitly given by the same formula as $\bar Q$:
\begin{align*}
Q=\sum_\alpha(x_\alpha-\chi(x_\alpha))\otimes x_\alpha^\ast-1\otimes
 \dfrac{1}{2}\sum_{\alpha,\beta,\gamma}c_{\alpha, \beta}^\gamma x_\alpha^\ast x_\beta^\ast x_\gamma.
\end{align*}
  \end{proof}
 Since $Q$ is odd, Lemma \ref{lem:Q} implies that 
 \begin{align*}
(\ad Q)^2=0.
 \end{align*}
 Thus,
 $(C(\mf{g}), \ad Q)$
 is a {\em differential graded algebra}, and
 its cohomology
 $H^\bullet(C(\mf{g}), \ad Q)$ is a graded superalgebra.

 However
the 
operator on $\gr  C(\fing)=\bar C(\fing)$ induced by $\ad Q$ does not coincide
with $\ad \bar Q$.
To remedy this, we introduce the {\em Kazhdan filtration}
$K_\bullet C(\fing)$
of $C(\fing)$ as follows:
Defined a $\Z$-grading on $\fing$ by
\begin{align*}
\fing=\bigoplus_{j\in \Z}\fing_j,\quad
\fing_j=\{x\in \fing: [h,x]=2 jx\}
\end{align*}
where $h$ is defined in 
\eqref{eq:sl2-triple}.
Then $\finn=\bigoplus_{j>0}\fing_j\subset \finb=\bigoplus_{j\geq
0}\fing_j$,
and
\begin{align*}
\finh:=\fing_0
\end{align*}
 is the Cartan subalgebra of $\fing$ consisting of diagonal matrices.
Extend the basis $\{x_\alpha\}_{\alpha\in \Delta_+}$
of $\finn$ to the basis $\{x_a\}_{a\in \Delta_+\sqcup I}$ of $\finb$
by adding a basis $\{x_i\}_{i\in I}$ of $\finh$.
Let $c_{a, b}^d$ denote the structure constant of $\finb$ with respect to this basis.
 \begin{lem}
The map  $\rho:\finn\ra Cl$ extends to the Lie algebra homomorphism
 \begin{align*}
  \rho: \finb\ra Cl,\quad x_a\mapsto \sum_{\beta,\gamma\in
  \Delta_+}c_{a,\beta}^\gamma x_\gamma x_\beta^*.
 \end{align*}
   \end{lem}
Define the Lie algebra homomorphism \begin{align*}
  \theta_0: \finb\ra C(\fing),\quad x_i\mapsto x_i\* 1+1\*\rho(x_i),
 \end{align*}
 and define a $\Z$-grading on $C(\fing)$ by
\begin{align*}
 C(\fing)=\bigoplus_{j\in \Z}C(\fing)[j],\quad
C(\fing)[j]=\{c\in C(\fing)\mid  [\theta_0(h),c]=2j x\}.
\end{align*}
Set
\begin{align*}
K_p C(\fing)=\sum_{i-j\leq p}C_i(\fing)[j],\quad\text{where }C_i(\fing)[j]=C_i(\fing)\cap C(\fing)[j].
\end{align*}
Then
$K_{\bullet}C(\fing)$ defines an increasing, exhaustive, separated filtration of $C(\fing)$
such that
$K_p C(\fing)\cdot K_q C(\fing)\subset K_{p+q}C(\fing)$,
$[K_p C(\fing), K_qC(\fing)]\subset K_{p+q-1}C(\fing)$,
and
$\gr_K C(\fing)=\bigoplus_{p}K_pC(\fing)/K_{p-1}C(\fing)$ is isomorphic to
$\bar C(\fing)$ as Poisson superalgebras.
Moreover,
the complex $(\gr_K C(\fing), \ad Q)$ is identical to $(\bar C(\fing), \ad \bar Q)$.

Let $\mc{Z}(\fing)$ be the center of $U(\fing)$.

 \begin{thm}[\cite{Kos78}]\label{Thm:whittaker-model}
  We have
  $H^{i}(C(\mf{g}), \ad Q)=0$ for $i\ne 0$
  and the map $\mc{Z}(\mf{g}) \longrightarrow H^0(C(\mf{g}), \ad Q)$
  defined by sending $z$ to $[z\otimes 1]$ is an isomorphism of algebras.
  Here $[z\* 1]$ denotes the cohomology class of $z\* 1$.
 \end{thm}

 \begin{proof}
  We have
 the spectral sequence 
$$E_r \Longrightarrow H^\bullet(C(\mf{g}), \ad Q)$$
such that 
$$E_1^{\bullet,i}=H^i(\gr_K {C}(\mf{g}), \ad \bar{Q})\cong \begin{cases}\C[\mf{g}^\ast]^G,&\text{if}\quad i=0
\\0,&\text{if}\quad i\neq 0.\end{cases}$$
Therefore the spectral sequence collapses  at $E_1=E_\infty$, so we get 
  $$\gr H^0(C(\mf{g}), \ad Q)\cong \C[\mf{g}^\ast]^G.$$
  Since the homomorphism  $\mc{Z}(\fing)\ra H^0(C(\fing), \ad Q)$,
  $z\mapsto [z\*1]$, respects the filtration
  $\mc{Z}_{\bullet}(\fing)$ and $K_{\bullet}H^0(C(\fing),\ad Q)$,
  where $\mc{Z}_p(\fing)=\mc{Z}(\fing)\cap U_p(\fing)$, $K_{p}H^{\bullet}(C(\fing),\ad Q)=
  \im (H^0(K_p C(\fing),\ad Q)\ra H^0(C(\fing),\ad Q))$,
  we get the desired isomorphism.
 \end{proof}
 \begin{rmk}[{see \cite[{\S 2}]{Ara07} for the details}]
  As in the case of  $\bar C(\fing)$,
$C(\mf{g})$ is also bigraded, we can also write $\ad Q = d_++d_-$ such that $d_+(C^{i,j})\subset C^{i+1, j}, d_-(C^{i,j})\subset C^{i, j+1}$ and get a spectral sequence 
$$E_r \Longrightarrow H^\bullet(C(\mf{g}), \ad Q)$$ 
  such that
  \begin{align*}
   E_2^{p,q}=H^p(H^q(C(\mf{g}), d_-), d_+) \cong 
  \delta_{q,0} H^p(\mf{n}, U(\mf{g})\otimes_{U(\mf{n})}\C_{\chi})\\
  \cong \delta_{p,0}\delta_{q,0}H^0(\finn,U(\fing)\*_{U(\finn)}\C_{\chi})\cong \End_{U(\mf{g})}(U(\mf{g})\otimes_{U(\mf{n})} \C_{\chi})^{op}.
  \end{align*}
Where $\C_\chi$ is the one-dimensional representation of $\mf{n}$ defined
  by the character $\chi$. 
  Thus we get the Whittaker model isomorphism \cite{Kos78}
  $$\mc{Z}(\mf{g})\cong  H^0(C(\mf{g}), \ad Q ) \cong
  \End_{U(\mf{g})}(U(\mf{g})\otimes_{U(\mf{n})} \C_{\chi})^{op}.$$
 \end{rmk}

 \subsection{Classical Miura map}\label{subsection:classical Miura}
 Let $\finn_-=\bigoplus\limits_{j<0}\fing_j$ be the subalgebra of $\fing$ consisting of lower triangular matrices,
 and set
 $\finb_-=\bigoplus\limits_{j\leq 0}\fing_j=\finn_-\+ \finh$.
 We have
 \begin{align}
  \fing=\finb_-\+\finn_+.
  \label{eq:dec}
 \end{align}

 Extend the basis
 $\{x_a\}_{a\in \Delta_+\sqcup I}$ to the basis
 $\{x_a\}_{a\in \Delta_+\sqcup I\sqcup \Delta_-}$ by adding a basis
 $\{x_{\alpha}\}_{\alpha\in \Delta_-}$ of $\finn_-$.
 Let $c_{a,b}^d$ be the structure constant with respect to this basis.
Extend $\theta_0:\finb\ra C(\fing)$ to the linear map
 $\theta_0:\fing\ra C(\fing)$ by setting
 \begin{align*}
  \theta_0(x_a)=x_a\* 1+1\*\sum_{\beta,\gamma\in \Delta_+}c_{a,\beta}^{\gamma}x_{\gamma}x_{\beta}^*.
 \end{align*}
 We already know that
 the restriction of $\theta_0$ to $\finn$ is a Lie algebra homomorphism and
 \begin{align*}
  [\theta_0(x),1\* y]=1\* [x,y]\quad \text{for }x,y\in \finn.
 \end{align*}
 Although 
$\theta_0$ is not a Lie algebra homomorphism,
  we have the following.
\begin{lem}
The restriction of $\theta_0$  to $\finb_-$ is a Lie algebra homomorphism.
	 We have
 $[\theta_0(x),1\* y]=1\* \ad^*(x)(y)$ for $x\in \finb_-$, $y\in \finn^*$,
 where $\ad^*$ denote the coadjoint action and
$\finn^*$ is identified with  $(\fing/\finb_-)^*$.
\end{lem}

Let
$C(\fing)_+$ denote the subalgebra of
$C(\fing)$ generated by
$\theta_0(\finn)$ and $\Lam(\finn)\subset Cl$,
and let $C(\fing)_-$ denote the subalgebra generated by
$\theta_0(\finb_-)$ and $\Lam(\finn^*)\subset Cl$.
\begin{lem}\label{lem:linear-dec-finite}
 The multiplication map gives a linear isomorphism
 \begin{align*}
  C(\fing)_-\*C(\fing)_+\isomap C(\fing).
 \end{align*}
\end{lem}
\begin{lem}\label{lem-decom-finite}
 The subspaces $C(\fing)_-$
 and $C(\fing)_+$ are subcomplexes of
 $(C(\fing),\ad Q)$.
 Hence $C(\fing)\cong C(\fing)_-\* C(\fing)_+$ as complexes.
\end{lem}
\begin{proof}
 The fact that $C(\fing)_+$ is subcomplex is obvious (see Lemma \ref{lem:Q}).
 The fact that
 $C(\fing)_-$ is a  subcomplex follows from the following formula.
 \begin{align*}
  &[Q,\theta_0(x_a)]=\sum_{b\in \Delta_-\sqcup I,\alpha\in \Delta_+}c_{\alpha,a}^b \theta_0(x_b)(1\*x_\alpha^*)-1\*\sum_{\beta,\gamma\in
  \Delta_+}c_{a,\beta}^{\gamma}\chi(x_{\gamma})x_{\beta}^* \\
  &[Q,1\* x_{\alpha}^*]=-1\*\frac{1}{2}\sum_{\beta,\gamma\in\Delta_+}c_{\beta,\gamma}^{\alpha}x_{\beta}^*x_{\gamma}^*
 \end{align*}
 ($a\in \Delta_-\sqcup I$, $\alpha\in \Delta_+$).
\end{proof}
\begin{prp}\label{prp:subcomp}
 $H^\bullet(C(\fing)_-, \ad Q)\cong H^\bullet(C(\fing),\ad Q)$.
\end{prp}
\begin{proof}
 By Lemma \ref{lem-decom-finite} and  Kunneth's Theorem,
 \begin{align*}
  H^p(C(\fing),\ad Q)\cong \bigoplus_{i+j=p}H^i(C(\fing)_-,\ad Q)\* H^j(C(\fing)_+,\ad Q).
 \end{align*}
 On the other hand
 we have
 $\ad(Q)(1\*x_\alpha)=\theta_\chi(x_{\alpha})=\theta_0(x_{\alpha})-\chi(x_{\alpha})$
 for $\alpha\in \Delta_-$.
Hence $C(\fing)_-$ is isomorphic to the tensor product of complexes of
 the form $\C[\theta_{\chi}(x_{\alpha})]\* \Lam (x_{\alpha})$
with the differential $\theta_{\chi}(x_{\alpha})\* x_{\alpha}^*$, where
 $x_{\alpha}^*$ denotes the odd derivation of the exterior algebra  $\Lam (x_{\alpha})$ with one variable $x_{\alpha}$  such that
 $x_{\alpha}^*(x_{\alpha})=1$.
Each of these complexes has one-dimensional zeroth cohomology and zero
 first cohomology.
Therefore
 $H^i(C(\fing)_+,\ad Q)=\delta_{i,0}\C$.
This completes the proof.
\end{proof}

Note that
the cohomological gradation takes only non-negative values on $C(\fing)_-$.
Hence by Proposition \ref{prp:subcomp}
we may identify
$\mc{Z}(\fing)=H^0(C(\fing),\ad Q)$  with  the subalgebra
$H^0(C(\fing)_-,\ad Q)=\{c\in C(\fing)_-^0\mid \ad Q(c)=0\}$ of $C(\fing)_-$.

  Consider the decomposition
  \begin{align*}
   C(\fing)_-^0=\bigoplus_{j\leq 0}C(\fing)_{-,j}^0,\quad C(\fing)^0_{-,j}
   =\{c\in C(\fing)_-^0\mid
   [\theta_0(h),c]=2jc\}.
  \end{align*}
  Note that
  $C(\fing)_{-,0}^0$   is generated by $\theta_0(\finh)$ and
  is isomorphic to $U(\finh)$.
The projection 
  \begin{align*}
   C(\fing)_-^0\rightarrow C(\fing)_{-,0}^0\cong  U(\finh)
  \end{align*}
  is an algebra homomorphism,
  and hence,
  its restriction 
\begin{align*}
\Mi: \mc{Z}(\fing)=H^0(C(\fing)_-,\ad Q)\ra U(\finh)
\end{align*}
is also an algebra homomorphism.
\begin{prp}\label{prp:miura-classical}
 The map $\Mi$ is an embedding.
\end{prp}

Let $K_{\bullet}C(\fing)_{\pm}$
 be the filtration of $C(\fing)_{\pm}$
induced by  the Kazhdan filtration of
$C(\fing)$.
We have the isomorphism
\begin{align*}
\bar C(\fing)= \gr_K C(\fing)\cong \gr_K C(\fing)_-\otimes \gr_K C(\fing)_+
\end{align*}
as complexes.
Similarly as above,
we have $H^i(\gr_KC(\fing)_+,\ad \bar Q)=\delta_{i,0}\C$,
and
\begin{align}
 H^0(\bar C(\fing),\ad \bar Q)\cong H^0(\gr_KC(\fing)_-,\ad \bar Q).
\label{eq;iso-11-3}
\end{align}
 \begin{proof}[Proof of Proposition \ref{prp:miura-classical} ]
The filtration
$K_{\bullet}U(\finh)$
of $U(\finh)\cong C(\fing)_{-,0}^0$ induced by the Kazhdan filtration
coincides with the usual PBW filtration.
By \eqref{eq;iso-11-3}
and Theorem \ref{thm:KS-classical},
 the induced map
\begin{align*}
H^0(\gr_K C(\fing)_-, \ad Q)\ra \gr_K U(\finh)
\end{align*}
can be identified with the 
restriction map
\begin{align}
 \bar \Mi: \C[\mc{S}]= \C[f+\finb]^N\ra \C[f+\finh].
 \label{eq:gr-of-classical-Miura}
\end{align}
 It is sufficient to show that 
$\bar \Mi$ is injective.

If $\varphi\in \C[f+\finb]^N$ is in the kernel,
$\varphi(g.x)=0$ for all $g\in N$ and $x\in f+\finh$.
Hence it is enough to show that  the image of the 
 the action map
 \begin{align}
  N\times (f+\finh)\ra f+\finb,\quad (g,x)\mapsto \Ad(g)x,
\label{eq:the-action-map-for-HC}
 \end{align}
is Zariski dense in $f+\finb$.

The differential of this morphism at
$(1,x)\in N\times (f+\finh)$ is given by
\begin{align*}
 \finn\times \finh\ra \finb,\quad (y,z)\mapsto [y,x]+z.
\end{align*}
This is an isomorphism if 
$x\in f+\finh_{\on{reg}}$,
where 
$\finh_{\on{reg}}=\{x\in \finh\mid \finn^x=0\}$.
Hence \eqref{eq:the-action-map-for-HC} is a dominant morphism as required, see e.g.
\cite[Theorem 16.5.7]{TauYu05}.
\end{proof}

\begin{rmk}\label{rem:classical-Miura}
 The fact that $\bar \Mi$ is injective is in fact well-known.
 Indeed, under the identifications $\C[\mc{S}]\cong \C[\fing]^G$, $\C[f+\finh]\cong \C[\finh]$,
 $\bar \Mi$ is identified with the Chevalley restriction map
 $\C[\fing]^G\isomap \C[\finh]^W$,
 where $W=\mf{S}_n$.

 The advantage of the above  proof is that it applies to a general finite $W$-algebra (\cite{Lyn79}),
 and also, it generalizes to the affine setting, see \S \ref{subsection:Miura}.
 \end{rmk}

The map $\Mi$ is called the classical {\em Miura map}.

 \subsection{Generalization to an arbitrary simple Lie algebra}
 It is clear that the above argument works if we replace $\mf{gl}_n$ by $\mf{sl}_n$,
 and $\mf{a}$ by $\mf{a}\cap \mf{sl}_n$.

 More generally, let $\fing$ be an arbitrary simple Lie algebra.
 Let $f$  be a {\em principal} (regular) nilpotent element of $\fing$,
 $\{e,f,h\}$ an associated $\mf{sl}_2$-triple. One may assume that
  \begin{align*}
   f=\sum_{i\in I}f_i,
  \end{align*}
  where $f_i$ is a  root vector of roots
  $\alpha_i$ and
  $\{\alpha_i\}_{i\in I}$ is the set of simple roots of $\fing$.
  Define
  the {\em Kostant slice} $\mc{S}$
  by
  \begin{align*}
   \mc{S}:=f+\fing^e\subset \fing=\fing^*,
  \end{align*}
  where $\fing^e$ is the centralizer of $e$ in $\fing$.

  Then all the statements in previous subsections that make sense 
  hold 
by replacing the set of companion matrices by the Kostant slice (\cite{Kos78}).

  \subsection{Generalization to finite $W$-algebras}
  \label{subsection:Generalization to finite $W$-algebras}
 In fact,  the above  argument works in more general setting of
 Hamiltonian reduction. In particular for {\em Slodowy slices}. 
Namely, for a non-zero nilpotent element $f$ of a finite-dimensional
 semisimple Lie algebra $\mf{g}$, we can use Jacobson-Morozov's theorem
 to embed $f$ into an $sl_2$-triple $\{e, f, h\}$.
 The Slodowy slice at $f$ is defined to be the affine subslace
 \begin{align*}
\mc{S}_f=f+ \mf{g}^e
 \end{align*}
 of $\fing$.

The Slodowy slice $\mc{S}_f$ has the following properties.
\begin{itemize}
\item $\mc{S}_f$ intersects the $G$-orbits at any point of $\mc{S}_f$, where $G$ is the adjoint group of $\mf{g}$.
\item $\mc{S}_f$ admits a $\C^\ast$-action which is contracting at $f$.
\end{itemize}
As in the case of the set of companion matrices
$\mc{S}_f$ can be realized by Hamiltonian reduction. 
Let
$\fing_j=\{x\in \fing\mid [h,x]=2j x\}$,
so that
\begin{align*}
 \fing=\bigoplus_{j\in \frac{1}{2}\Z}\fing_j.
\end{align*}
Then the subspace $\fing_{1/2}$ admits
a symplectic form defined by $\bra x|y\ket=(f|[x,y])$.
Choose a Lagrangian subspace $l$ of $\fing_{1/2}$ with respect to this
form,
and set
$\mf{m}=l+\sum_{j\geq 1}\fing_j$.
Then $\mf{m}$ is a nilpotent subalgebra
of $\fing$ and 
$\chi:\mf{m}\ra \C$, $x\mapsto (f|x)$,
defines a character.
Let $M$ be the unipotent subgroup of $G$ corresponding to $\mf{m}$,
that is, 
$\Lie M=\mf{m}$.
The adjoint action of $M$ on 
$\mf{g}$ is Hamiltonian,
so we can 
consider the moment map of this action
\begin{align*}
 \mu: \mf{g}^\ast \longrightarrow \mf{m}^\ast,
\end{align*}
which is just a restriction map.
Then 
we have 
the following realization of the Slodowy slice.
$$\mc{S}_f \cong \dfrac{\mu^{-1}(\chi)}{M}$$

To obtain the BRST realization of this Hamiltonian reduction
we simply
 replace the Clifford algebra $Cl$ by $Cl_{\mf{m}}$, i.e., 
the Clifford algebra associated to $\mf{m}\oplus \mf{m}^\ast$.
Then
 we can define the operator $\ad \bar{Q}$ similarly and get a
 differential cochain complex $(\C[\mf{g}^\ast]\otimes
 \overline{Cl}_{\mf{m}}, \ad \bar{Q})$. We have 
$$\C[\mc{S}_f]\cong H^0(\C[\mf{g}^\ast]\otimes \overline{Cl}_{\mf{m}},
\ad \bar{Q} )$$
as Poisson algebras.

As above, 
this construction has a natural quantization 
and 
the quantization $U(\fing,f)$ of $\mc{S}_f$
thus defined
is called  the {\em finite $W$-algebra associated to the pair $(\mf{g},
f)$} \cite{Pre02}:
$$U(\mf{g}, f):=  H^0(U(\mf{g})\otimes Cl_{\mf{m}}, \ad \bar{Q}_+ ) \cong \End_{U(\mf{g})}(U(\mf{g})\otimes_{U(\mf{m})} \C_\chi)^{op},$$
where $\C_\chi$ is the one-dimensional representation of $\mf{m}$
defined by $\chi$ (cf. \cite{De-Kac06,Ara07}).

\section{Arc spaces, Poisson vertex algebras,  and  associated varieties
 of vertex algebras.}
 \label{sec:Arc spaces}
 \subsection{Vertex algebras}
 A {\em vertex algebra} is a vector space $V$ equipped with 
 $|0\ket \in V$ (the vacuum vector),
 $T\in \End V$ (the translation operator),
 and
 a bilinear product
 \begin{align*}
V\times V\ra V((z)),\quad (a,b)\mapsto  a(z)b,
\end{align*}
where $a(z)=\sum_{n\in \Z}a_{(n)}z^{-n-1}$, $a_{(n)}\in \End V$, such that
\begin{enumerate}
\item $(|0\ket)(z)=\id_V$,
\item $a(z)|0\ket \in V[[z]]$ and $\lim\limits_{z\ra 0}a(z)|0\ket =a$ for all $a\in V$,
\item $(Ta)(z)=\partial_z a(z)$ for all $a\in V$,
where $\partial_z=d/dz$,
\item for any $a,b\in V$,
$(z-w)^{N_{a,b}}[a(z),b(w)]=0$ for some $N_{a,b}\in \Z_+=\{0,1,2,\ldots\}$. \end{enumerate}
The last condition is called the {\em locality},
which is equivalent to the fact that
\begin{align}
[a(z),b(w)]=\sum_{n= 0}^{N_{a,b}-1}(a_{(n)}b)(w)\frac{1}{n!}\partial_w^n\delta(z-w),
\label{eq:OPE}
\end{align}
where $\delta(z-w)=\sum_{n\in \Z}w^n z^{-n-1}\in \C[[z,w,z^{-1},w^{-1}]]$.

A consequence of the definition is
the  following {\em Borcherds identities}:
\begin{align}
 &[a_{(m)}, b_{(n)}]
 =\sum_{i\geq 0}\begin{pmatrix}
				 m\\i
		\end{pmatrix}(a_{(i)}b)_{(m+n-i)},
\label{eq:com-formula}\\
&(a_{(m)}b)_{(n)}=\sum_{j\geq 0}(-1)^j
\begin{pmatrix}
m\\j
\end{pmatrix}(a_{(m-j)}b_{(n+j)}-(-1)^mb_{(m+n-j)}a_{(j)}).
\end{align}

We write \eqref{eq:OPE}
 as 
 \begin{align*}
 [a_{\lam}b]=\sum_{n\geq 0}\frac{\lam^n}{n!}a_{(n)}b\in V[\lam],
\end{align*}
and call it 
the {\em $\lam$-bracket }of $a$ and $b$.
(We have $a_{(n)}b=0$ if $(z-w)^n[a(z),b(w)]=0$.)

Here are some properties of $\lambda$-brackets.
\begin{align}
 & [(Ta) _{\lam}b]=-\lam[a_{\lam}b],\quad [a_{\lam}(T b)]=(\lam+T)[a_{\lam}b],
 \label{eq:sesquilinearity}
 \\
 &[b_{\lam}a]=-[a_{-\lam-T}b],\label{eq:skew}\\
 &[a_{\lam}[b_{\mu}c]]-[b_{\mu}[a_{\lam}c]]=[[a_{\lam}b]_{\lam+\mu}c].
 \label{eq:Jacobi}
\end{align}

The normally ordered product
on $V$ is defied as $:ab:=a_{(-1)}b$.
We also write $:ab:(z)=:a(z)b(z):$.
We have
\begin{align*}
:a(z)b(z):=a(z)_+b(z)+b(w)a(z)_-,
\end{align*}
where
$a(z)_+=\sum_{n<0}a_{(n)}z^{-n-1}$,
$a(z)_-=\sum_{n\geq 0}a_{(n)}z^{-n-1}$.
We have the following {\em non-commutative Wick formula}.
\begin{align}
& [a_{\lam}:bc:]=:[a_{\lam}b]c:+:[a_{\lam}c]b:+\int_0^{\lam}[[a_{\lam}b]_\mu c]d\mu,
\label{eq:non-com-wick1}
\\&
[:ab:_{\lam}c]=:(e^{T\partial_{\lam}}a)[b_{\lam}c]:
 +:(e^{T\partial_{\lam}}b)[a_{\lam}c]:+\int_{0}^\lam [b_\mu[a_{\lam-\mu}c]]d\mu.
 \label{eq:non-com-wick2}
\end{align}

\subsection{Commutative vertex algebras and differential algebras}
A vertex algebra $V$ is called {\em commutative}
if 
\begin{align*}
[a_{\lam}b]=0,\quad \forall a,b\in V,
\end{align*}
or equivalently,
$a_{(n)}=0$ for $n\geq 0$ in $\End V$ for all $a\in V$.
This condition is equivalent to that
\begin{align*}
 [a_{(m)}, b_{(n)}]=0\quad\forall a,b\in \Z,\ m,n\in \Z
\end{align*}
by 
\eqref{eq:com-formula}.

A commutative vertex algebra has the  structure of a unital commutative 
algebra 
by the product
\begin{align*}
 a\cdot b=:ab:=a_{(-1)}b,
\end{align*}
where the unite is given by the vacuum vector $|0\ket$.
The translation operator $T$ of $V$
acts on $V$ as a derivation with respect to this product:
\begin{align*}
 T(a\cdot  b)=(Ta)\cdot b+a\cdot (Tb).
\end{align*}
Therefore 
a commutative vertex algebra 
has the structure of a
 {\em differential algebra},
that is,
 a unital commutative algebra
equipped with a derivation.
Conversely,
there is a unique vertex algebra structure on
a differential algebra $R$  with a derivation $T$ 
such that
\begin{align*}
Y(a,z)=e^{zT}a
\end{align*}
for  $a\in R$.
This correspondence gives the following.
\begin{thm}[\cite{Bor86}]
 The category of commutative vertex algebras
is the same as that of differential algebras.
\end{thm}

\subsection{Arc spaces}
Define the  (formal) disc
as
\begin{align*}
 D=\on{Spec}(\C[[t]]).
\end{align*}
For a scheme $X$, a homomorphism $\alpha:D\ra X$ is called an {\em arc} of $X$.
  \begin{Th} [\cite{Bosch:1990fj,EinMus,Ish07}]
   Let $X$ be a scheme of finite type over $\C$,
   $Sch$ the category of schemes of finite type over $\C$,
   $Set$ the category of sets.
     The contravariant  functor
  \begin{align*}
 Sch\ra Set,\quad Y\mapsto \Hom_{Sch} (Y\widehat{\times} D,X),
  \end{align*}
   is represented by a scheme $JX$,
   that is,
   \begin{align*}
\Hom_{Sch}(Y,JX)\cong \Hom_{Sch} (Y\widehat{\times} D,X).
\end{align*}
for any $Y\in Sch$.   
Here $Y\widehat{\times} D$
is the completion of $Y{\times} D$ with respect to the subscheme
$Y\widehat{\times}\{0\}$.
\end{Th}
By definition,
the $\C$-points of  $JX$
are
\begin{align*}
\Hom_{Sch} (\Spec\C,JX)=\Hom_{Sch}(D,X),
\end{align*}
that is, the set of arcs of $X$.
The reason we need the completion $Y\widehat{\times} D$
in the definition is that 
$A\* \C[[t]]\subsetneqq A[[t]]=A\widehat{\otimes} \C[[t]]$ in general.

The scheme $JX$ is called the {\em arc space}, or the {\em infinite jet scheme}, of $X$.

It is easy to describe $JX$ when
$X$ is affine:

First, consider the case
$X= \C^N= \Spec \C[x_1, x_2, \cdots, x_N]$.
  The $\C$-points of $JX$ are 
  the arcs $\Hom_{Sch}(D, JX) $,
  that is,
the ring homomorphisms
\begin{equation*}
\gamma: \C[x_1, x_2, \cdots, x_N]\ra \C[[t]].
\end{equation*}
Such a map is determined by the image 
\begin{align}
 \gamma(x_i)=\sum_{n\geq 0}\gamma_{i,(-n-1)}
t^{n}
\label{eq:gamma(xi)}
\end{align}
of each $x_i$, and 
conversely,
the coefficients 
$\{\gamma_{i,(-n-1)}\}$ determines a $\C$-point of $JX$.
If we choose coordinates $x_{i,(-n-1)}$ of $JX$ as
$x_{i,(-n-1)}(\gamma)=\gamma_{i,(-n-1)}$,
we have
\begin{align*}
J\C^N = \Spec \C[x_{i,(n)}| i=1, 2, \cdots, N,  n=-1, -2, \cdots ].
\end{align*}

Next,  let
$X=\Spec R$, 
with $R={\C[x_1, x_2, \cdots, x_N]}/{\langle f_1, f_2, \cdots, f_r\rangle}$.
The arcs of $X$ are
	   \begin{equation*}
\Hom_{\on{ring}}
\big( \frac{\C[x_1, x_2, \cdots, x_n]}{<f_1, f_2, \cdots, f_r>}, \C[[t]]\big) \subset \Hom_{\on{ring}}( \C[x_1, x_2, \cdots, x_n], \C[[t]]).
	   \end{equation*}
An element $\gamma \in \Hom_{\on{ring}}( \C[x_1, x_2, \cdots, x_n], \C[[t]])$ is an element of this subset if and only if  $\gamma(f_i)=0$ for $i=1, 2, \cdots, r$.
By writing
\begin{align*}
f_i(x_1(t),x_2(t),\dots,x_N(t))=\sum_{m\geq 0}\frac{f_{i,m}}{m!}t^m
\end{align*}
with $f_{i,m}\in \C[x_{i,(-n-1)}]$,
where $x_i(t):=\sum_{m\geq 0}x_{i,(-m-1)}t^{m}$,
we get that
$$JX = \Spec \frac{\C[x_{i(n)}| i=1, 2, \cdots, N; n=-1, -2, \cdots]}{\langle f_{i,m}(x_{i(n)}), i=1, 2, \cdots, r; m\geq 0 \rangle}.$$
\begin{lem}
Define the derivation $T$
of 
$\C[x_{i(n)}| i=1, 2, \cdots, N; n=-1, -2, \cdots ]$ by 
\begin{align*}
T x_{i(n)}=-nx_{i(n-1)}.
\end{align*}
 Then $f_{i,m}= T^nf_i$ for $n\geq 0$. Here we identify $x_i$ with $x_{i(-1)}$.
\end{lem}
With the above lemma, we 
conclude  that for the affine scheme $X=\Spec R$,
$R= {\C[x_1, x_2, \cdots, x_n]}/{\langle f_1, f_2, \cdots, f_r\rangle}$, its arc space
 $JX$ is
the affine scheme
$\Spec (JR)$,
where
 $$JR:= \frac{\C[x_{i(n)}| i=1, 2, \cdots, N; n=-1, -2, \cdots]}{\langle T^nf_i, i=1, 2, \cdots, r; n\geq 0 \rangle}$$
and $T$ is as defined in the lemma. 

The derivation $T$ acts on the above quotient ring $JR$.
Hence
for an affine scheme $X=\Spec R$,
the coordinate ring  $JR=\C[JX]$
of its arc space $JX$ is a differential algebra,  hence is a commutative vertex algebra.
\begin{Rem}\label{Rem:universal property of JR}
The differential algebra $JR$ has the universal property that
$$\Hom_{\rm{dif.alg.}}(JR,A)\cong \Hom_{\rm{ring}}(R,A)$$
for any differential algebra $A$,
where $ \Hom_{\rm{diff. alg.}}(JR, A)$ is the set of homomorphisms $JR\ra A$ of differential algebras.
\end{Rem}

For a general scheme $Y$ of finite type with an affine open covering $\{U_i\}_{i\in I}$, its arc space $JY$ is obtained by glueing
 $JU_i$ (see \cite{EinMus,Ish07}). In particular, the structure sheaf $\mc{O}_{JY}$ is a sheaf of commutative vertex algebras.
 
 There is a natural 
 projection $\pi_{\infty}: JX\ra X$  that corresponds to the embedding $R\hookrightarrow JR$, $x_i\ra x_{i,(-1)}$,
 in the case $X$ is affine.
 In terms of arcs, $\pi_{\infty}(\alpha)=\alpha(0)$ for $\alpha\in \Hom_{Sch}(D,X)$,
 where $0$ is the unique closed point of the disc $D$.

 The map from a scheme to its arc space is functorial. i.e., a scheme homomorphism $f : X \rightarrow Y$ induces a scheme homomorphism $Jf : JX \rightarrow JY$ that makes the following diagram commutative:
 \begin{align*}
			      \begin{CD}
JX @>Jf>   > J Y\\
@VV \pi_{\infty}V @VV \pi_{\infty}  V\\
X@> f >>Y. \end{CD}
			    \end{align*}
 In terms of arcs,
$Jf(\alpha)=f\circ \alpha$
for $\alpha\in \Hom_{Sch}(D,X)$.

We also have
\begin{align}
J(X \times Y)\cong  JX\times JY.
\label{eq:product-of-jets}
\end{align}
Indeed,
for any scheme $Z$,
\begin{equation*}
\begin{split}
\Hom (Z, J(X\times Y) )
&= \Hom (Z\widehat\times D, X\times Y) \\
&\cong \Hom (Z\widehat\times D, X) \times \Hom(Y\widehat \times D, Y) \\
&= \Hom(Z, JX) \times \Hom(Y, JY) \\
&\cong \Hom(Z, JX\times JY).
\end{split}
\end{equation*}

\begin{Lem}\label{prp:homeo}
The natural morphism $X_{\on{red}}\ra X$ induces 
an isomorphism $J X_{\on{red}}\ra JX$ of topological spaces,
where $X_{\on{red}}$ denotes the reduced scheme of $X$.
\end{Lem}
\begin{proof}
We may assume that $X=\Spec R$.
An arc $\alpha$  of $X$ corresponds to a ring homomorphism
$\alpha^*:R\ra \C[[t]]$. Since $\C[[t]]$ is an integral domain
it decomposes as $\alpha^*: R\ra R/\sqrt{0}\ra \C[[t]]$.
Thus, $\alpha $ is an arc of $X_{\on{red}}$.
\end{proof}

If $X$ is  a point, 
then $JX$ is also a point,
since
$\Hom(D,X)=\Hom(\C,\C[[t]])$ consists of only one element.
Thus, Lemma \ref{prp:homeo} implies the following.
\begin{cor}\label{Co:jets-0-dim}
If $X$ is zero-dimensional then $JX$ is also zero-dimensional.
\end{cor}

\begin{Th}[\cite{Kolchin:1973kq}]
$JX$ is irreducible if $X$ is irreducible.
\end{Th}

\begin{Lem}\label{Lem:dominant-jet}
Let $Y$ be irreducible, and
let $f:X\ra Y$ be a morphism
that restricts to a bijection between some open subsets $U\subset X$ and $V\subset Y$.
Then $Jf: JX\ra JY$ is  dominant.
\end{Lem}
\begin{proof}
$Jf$ restricts to the isomorphism $JU\isomap JV$,
and the open subset $JV$ is dense in $JY$ since $JY$ is irreducible.
\end{proof}

\subsection{Arc space of Poisson varieties and Poisson vertex algebras}
Let $V$ be a commutative vertex algebra,
or equivalently, a differential algebra.
 $V$ is called a {\em Poisson vertex algebras}  if
it is equipped with a
bilinear map
\begin{align*}
 V\times V\ra V[\lam], \quad (a,b)\mapsto \{a_{\lam}b\}=\sum_{n\geq 0}\frac{\lam^n}{n!} a_{(n)}b,\quad
 a_{(n)}\in \End V,
\end{align*}
also called the {\em $\lam$-bracket},
satisfying the following axioms:
\begin{align}
 & \{(Ta) _{\lam}b\}=-\lam\{a_{\lam}b\},\quad \{a_{\lam}(T b)\}=(\lam+T)\{a_{\lam}b\},
 \label{eq:sesquilinearity-P} \\
 &\{b_{\lam}a\}=-\{a_{-\lam-T}b\},\label{eq:skew-P}\\
 &\{a_{\lam}\{b_{\mu}c\}\}-\{b_{\mu}\{a_{\lam}c\}\}=\{\{a_{\lam}b\}_{\lam+\mu}c\}, \label{eq:Jacobi-P}\\
 & \{a_{\lam}(bc)\}=\{a_{\lam}b\}c+\{a_{\lam}c\}b,\quad \{(ab)_{\lam}c\}=\{a_{\lam+T}c\}_{\rightarrow }b
 +\{b_{\lam+T}c\}_{\rightarrow}a, \label{eq:lipniz-P}
\end{align}
where
the arrow means that
$\lam+T $ should be moved to the right,
that is,
$\{a_{\lam+T}c\}_{\rightarrow }b=\sum_{n\geq 0}(a_{(n)}c)\frac{(\lam+T)^n}{n!}b$.

The first equation in \eqref{eq:lipniz-P}
says that $a_{(n)}$, $n\geq 0$, is a derivation of the ring $V$.
(Do not confuse $a_{(n)}\in \on{Der}(V)$, $n\geq 0$, with the multiplication $a_{(n)}$ as a vertex algebra,
which should be zero for a commutative vertex algebra.)

Note that \eqref{eq:sesquilinearity-P}, \eqref{eq:skew-P},  \eqref{eq:Jacobi-P}
are the same as  \eqref{eq:sesquilinearity}, \eqref{eq:skew},  \eqref{eq:Jacobi},
and 
\eqref{eq:lipniz-P} is the same with
\eqref{eq:non-com-wick1} and \eqref{eq:non-com-wick2} without the third terms.
In particular, by \eqref{eq:Jacobi-P},
we have
\begin{align}
 [a_{(m)}, b_{(n)}]=\sum_{i\geq 0}\begin{pmatrix}
				   m\\i
 \end{pmatrix}(a_{(i)}b)_{(m+n-i)},\quad m,n\in \Z_+.
 \label{eq:commutator-formula-Poisson}
\end{align}

\begin{thm}[{\cite[Proposition 2.3.1]{Ara12}}]
Let $X$ be an affine Poisson scheme, that is,  $X=\Spec R$ for some Poisson
 algebra $R$. Then there is a unique Poisson vertex algebra structure on
 $JR=\C[JX]$ such that 
\begin{align*}
\{a_\lambda b\}= \{a, b\} \quad\text{for }a, b \in R\subset JR,
\end{align*}
where $\{a, b\}$ is the Poisson bracket in $R$. 
\end{thm}
\begin{proof}
 The uniqueness is clear by \eqref{eq:sesquilinearity} since $JR$ is generated by $R$ as a differential algebra.
 We leave it to the reader to check the well-definedness.
\end{proof}

 \begin{rmk}
More generally, let $X$  be a Poisson scheme which is not necessarily
  affine.
Then the structure sheaf $\mathcal{O}_{JX}$ carries a unique vertex
  Poisson algebra
structure such that $\{f_{\lam}g\}=\{f,g\}$ for $f,g\in \mathcal{O}_X\subset
  \mathcal{O}_{JX}$, see \cite[Lemma 2.1.3.1]{AKM}.
 \end{rmk}

\begin{exm}
Let $G$ be an affine algebraic group, $\fing=\on{Lie}G$. 
The arc space $JG$ is naturally a proalgebraic group.
Regarding $JG$ as the $\C[[t]]$-points of $G$,
we have $JG=G[[t]]$.
Similarly, $J\fing=\fing[[t]]=\on{Lie}(JG)$.

 The affine space $\fing^*$ is a Poisson variety
 by the  Kirillov-Kostant Poisson structure, see \S \ref{subs:trasversal-slice}.
 If 
 $\{x_i\}$is  a basis of $\fing$,
 then
 $$ \C[\mf{g}^*]=\C[x_1, x_2, \cdots, x_n].$$
Thus
\begin{equation}
\begin{split}
J\mf{g}^*&= \Spec \C[x_{i(-n)}| i=1, 2, \cdots, l; n\geq 1].
\end{split}
\end{equation}
So we may identify $\C[J\fing^*]$ with the symmetric algebra $S(\mf{g}[t^{-1}]t^{-1})$.
 
 Let $x=x_{(-1)}\vac=(xt^{-1})\vac$, where we denote by $\vac$ the 
unite element in $S(\mf{g}[t^{-1}]t^{-1})$.  
Then  \eqref{eq:commutator-formula-Poisson} gives that
 \begin{equation}
 [x_{(m)},y_{(n)}]=[x,
  y]_{(m+n)},\quad
x,y\in \fing,\
 m,n\in\Z_{\geq 0}.
 \end{equation}
 So the Lie algebra $J\mf{g}=\fing[[t]]$ acts on $\C[J\mf{g}^*]$.
This action coincides with that obtained by differentiating the action of $JG =G[[t]]$ on $J\mf{g}^*$
induced by the coadjoint action of $G$.
In other words, the vertex Poisson algebra structure of 
$\C[J\fing^*]$ comes from the $JG$-action on $J\fing^*$.
\end{exm}

\subsection{Canonical filtration of vertex algebras.}
Haisheng Li
\cite{Li05} has shown that {\em every} vertex algebra is canonically filtered:
For a vertex algebra $V$, 
let $F^pV$ be the subspace of
$V$ spanned by the elements
\begin{align*}
a^1_{(-n_1-1)}a^2_{(-n_2-1)}\cdots a^r_{(-n_r-1)}|0\ket
\end{align*}
with 
$a^1, a^2, \cdots, a^r\in V$, $n_i \geq 0$, $n_1+n_2+\cdots +n_r
 \geq p$.
Then
\begin{align*}
 V=F^0V\supset F^1 V\supset \dots .
\end{align*}
It is clear that 
 $TF^pV \subset F^{p+1}V$.
 
 Set $(F^pV)_{(n)}F^qV:= \haru_{\C}\{a_{(n)}b| a\in F^p V ,b \in
	    F^qV\} $.

\begin{Lem}
We have
\begin{align*}
F^pV=\sum_{j\geq 0}(F^0V)_{(-j-1)} F^{p-j}V.
\end{align*}
\end{Lem}

\begin{prp}\label{Prp:filtration}
 \begin{itemize}
\item[(1)]  $(F^pV)_{(n)}(F^qV)\subset F^{p+q-n-1}V$.
Moreover, if $n\geq 0$, we have $(F^pV)_{(n)}(F^qV) \subset F^{p+q-n}V$.
\item[(2)] 
The filtration $F^{\bullet}V$
is separated, that is,
$\bigcap_{p \geq 0} F^pV = \{0\}$, if $V$ is a positive energy representation over itself.
\end{itemize} 
\end{prp}
\begin{proof}
It is straightforward to check.
((2) also follows from Lemma \ref{lem:G=F} below.)
\end{proof}
In this note  we assume that 
the filtration $F^{\bullet}V$ is separated.

Set
\begin{align*}
 \gr V=\bigoplus_{p\geq 0}F^pV/F^{p+1}V.
\end{align*}
We denote by $\sigma_p: F^pV \mapsto F^pV/F^{p+1}V$ for $p \geq 0$, the
canonical quotient map.

Proposition \ref{Prp:filtration} gives the following.
 \begin{prp}[\cite{Li05}]\label{p:2}
 The space $\gr V$ is a Poisson vertex algebra by
 $$\sigma_p(a)\cdot\sigma_q(b):=\sigma_{p+q}(a_{(-1)}b), \quad
 \sigma_p(a)_{(n)}\sigma_q(b):=\sigma_{p+q-n}(a_{(n)}b)$$ for 
$a \in F^pV$, $b\in F^qV$, $n\geq 0$.
\end{prp}
Set
\begin{align*}
R_V:= F^0V/F^1V \subset \gr V.
\end{align*}
Note that $F^1 V=\haru_{\C}\{a_{(-2)}b\mid a,b\in V\}$.
 \begin{prp}[\cite{Zhu96,Li05}]
  The restriction of the Poisson structure  gives $R_V$ a Poisson algebra structure,
that is,
$R_V$
 is a Poisson algebra by 
\begin{align*}
\bar{a}\cdot \bar{b}:=\overline{a_{(-1)}b},
\quad\{\bar{a}, \bar{b}\}=\overline{a_{(0)}b},
\end{align*}
where $\bar a=\sigma_0(a)$.
 \end{prp}
\begin{proof}
 It is straightforward from Proposition \ref{p:2}.
\end{proof}
In the literature $F^1 V$ is often denoted by $C_2(V)$
and
the Poisson algebra
$R_V$ is called \emph{Zhu's $C_2$-algebra}.

A vertex algebra $V$ is called {\em finitely strongly  generated} if $R_V$ is 
finitely generated as a ring.
If the images of vectors $a_1,\dots,a_N\in V$ generate $R_V$,
we say that $V$ is strongly generated by
$a_1,\dots,a_N$.

Below we always assume that a vertex algebra
$V$ is finitely strongly generated.

Note that if $\phi:V\ra W$ is a homomorphism of vertex algebras,
$\phi$ respects the canonical filtration, that is,
$\phi(F^pV)\subset F^pW$.
Hence it induces the homomorphism $\gr V\ra \gr W$
of Poisson vertex algebra homomorphism
which we denote by $\gr \phi$.

\subsection{Associated variety and singular support of vertex algebras
}
\begin{defn}
Define the {\em associated scheme} $\tilde{X}_V$ and the {\em associated variety} $X_V$ of a vertex algebra $V$ as 
\begin{align*}
\tilde{X}_V:= \Spec R_V,\quad X_V:=\on{Specm} R_V=(\tilde X_V)_{\on{red}}.
\end{align*}
\end{defn}

It was shown in  {\cite[Lemma 4.2]{Li05}}
that 
 $\gr V$ is generated by the subring $R_V$ as a differential algebra.
Thus, we have a surjection 
$JR_V\ra \gr V$ of differential algebras by Remark \ref{Rem:universal property of JR}.
This is in fact a homomorphism of Poisson vertex algebras:
\begin{thm}[{\cite[Lemma 4.2]{Li05}}, {\cite[Proposition 2.5.1]{Ara12}}]
\label{thm:surj-poisson}
The identity map $ R_V \rightarrow R_V$ induces a surjective Poisson
 vertex algebra homomorphism
\begin{align*}
JR_V=\C[J\tilde X_V]  \twoheadrightarrow \gr V.
\end{align*}
\end{thm}

Let $a^1,\dots, a^n$
be a set of strong generators  of $V$.
Since
 $\gr V\cong V$ as $\C$-vector spaces by the assumption that
$F^{\bullet}V$
is separated,
it follows from Theorem \ref{thm:surj-poisson}
that
 $V$
is spanned by elements
\begin{align*}
 a^{i_1}_{(-n_1)}\dots  a^{i_r}_{(-n_r)}|0\ket \quad \text{with }r\geq
 0,\ n_i\geq 1.
\end{align*}

\begin{defn}
Define the {\em singular support} of a vertex algebra $V$ as 
\begin{align*}
SS(V):= \Spec (\gr V) \subset J\tilde X_V.
\end{align*}
\end{defn}
\begin{thm}
We have
$\dim SS(V)=0$ if and only if $\dim X_V=0$. 
\end{thm}
\begin{proof}
 The ``only if'' part is 
obvious sine $\pi_{\infty}(SS(V))=\tilde{X}_V$,
where 
$\pi_{\infty}:J\tilde X_V\ra \tilde X_V$ is the projection.
The  ``if'' part
follows from Corollary \ref{Co:jets-0-dim}. 
\end{proof}

\begin{defn}
We call $V$ {\em lisse} (or {\em $C_2$-cofinite}) if $\dim X_V=0$.
\end{defn}
\begin{rmk}
Suppose that
 $V$ is $\Z_+$-graded, so that $V=\bigoplus_{i\geq 0}V_i$, 
and   that  $V_0=\C\vac$.
Then  $\gr V$ and $R_V$ are  equipped with the induced grading:
\begin{align*}
&\gr V= \bigoplus_{i\geq 0}(\gr V)_i,\quad (\gr V)_0=\C,\\
&R_V=\bigoplus_{i\geq 0}(R_V)_i,\quad (R_V)_0=\C.
\end{align*}
So
the following conditions are equivalent:
\begin{enumerate}
 \item $V$ is lisse.
       \item $X_V=\{0\}$.
\item The image of any vector $a\in V_i$ for $i\geq 1$ in $\gr V$ is
      nilpotent.
\item The image of any vector $a\in V_i$ for $i\geq 1$ in $R_V$ is nilpotent.
\end{enumerate}
Thus,
lisse vertex algebras  can be regarded as a generalization of
 finite-dimensional algebras.
 \end{rmk}

\begin{Rem}\label{Rem:one-step-further}
 Suppose that
 the Poisson structure of $R_V$ is trivial.
 Then the 
 Poisson vertex algebra structure of $JR_V$ is trivial, and so is that of $\gr V$
 by Theorem \ref{thm:surj-poisson}.
This happens if and only if
\begin{align*}
(F^pV)_{(n)}(F^qV)\subset F^{p+q-n+1}V\quad\text{ for all $n\geq 0$.}
\end{align*}
 If this is the case,
 one can give $\gr V$ yet another Poisson vertex algebra structure
 by setting  \begin{align}
  \sigma_p(a)_{(n)}\sigma_q(b):=\sigma_{p+q-n+1}(a_{(n)}b)\quad \text{for }n\geq 0.
  \label{eq:further}
 \end{align}
 (We can repeat this procedure if this Poisson vertex algebra structure is again trivial).
\end{Rem}

\subsection{Comparison with weight-depending filtration}
Let
$V$ be a vertex algebra that is  $\Z$-graded 
by some Hamiltonian $H$:
$$V= \bigoplus_{\Delta\in \Z}V_\Delta  \quad \mbox{where} \quad V_\Delta:=\{v\in V| Hv=\Delta v \}.$$
Then there is \cite{Li04} another natural filtration of $V$ 
defined as follows.

For a homogeneous vector
$a\in V_{\Delta}$,
$\Delta$ is called the {\em conformal weight} of $a$ and is denote by $\Delta_a$.
Let 
$G_pV$  be the subspace 
of $V$
spanned by the vectors
\begin{align*}
a^1_{(-n_1-1)}a^2_{(-n_2-1)}\cdots a^r_{(-n_r-1)}|0\ket
\end{align*}
with $\Delta_{a^1}+\dots +\Delta_{a^r}\leq p$.
Then $G_{\bullet}V$ defines an increasing filtration of $V$:
\begin{align*}
 0=G_{-1}V\subset G_0 V\subset \dots G_1 V\subset \dots,
\quad V=\bigcup_p G_p V.
 \end{align*}
Moreover  we have
\begin{align*}
& T G_p V\subset G_p V,\\
&(G_p)_{(n)}G_q V\subset G_{p+q}V\quad \text{for }n\in \Z,\\
&(G_p)_{(n)}G_q V\subset G_{p+q-1}V\quad \text{for }n\in \Z_+,
\end{align*}
It follows that $\gr_G V=\bigoplus G_pV/G_{p-1}V$ 
is naturally a Poisson vertex algebra.

It is not too difficult to see the following.
\begin{lem}[{\cite[Proposition 2.6.1]{Ara12}}]\label{lem:G=F}
 We have 
\begin{align*}
 F^p V_{\Delta}=G_{\Delta-p}V_{\Delta},
\end{align*}
where $F^p V_{\Delta}=V_{\Delta}\cap F^p V$,
$G_p V_{\Delta}=V_{\Delta}\cap G_p V$.
Therefore
\begin{align*}
 \gr V\cong \gr_G V
\end{align*}
as Poisson vertex algebras.
\end{lem}

\subsection{Example: universal affine vertex algebras}
\label{subs:affineVOA}
Let $\mf{a}$ be a Lie algebra
with a symmetric invariant bilinear form 
$\kappa$.
Let
\begin{align*}
\widehat{\mf{a}}=\mf{a}[t,t^{-1}]\+ \C  \mathbf{1}
\end{align*} 
be the Kac-Moody affinization of $\mf{a}$.
It is a Lie algebra with commutation relations
\begin{align*}
 [xt^m,yt^n]=[x,y]t^{m+n}+m\delta_{m+n,0}\kappa(x,y)\mathbf{1},
\quad x,y\in \mf{a},\ m,n\in \Z,\quad
\quad [\mathbf{1},\widehat{\mf{a}}]=0.
\end{align*}
Let
\begin{align*}
V^\kappa(\mf{a}) =
 U(\hat{\mf{a}})\otimes_{U(\mf{a}[t]\oplus \C \mathbf{1})}\C,
\end{align*}
where $\C$ is one-dimensional representation of 
$\mf{a}[t]\oplus \C \mathbf{1}$ on which
$\mf{a}[t]$ acts trivially and $\mathbf{1}$ acts as the identity.
The space $V^\kappa(\mf{a})$ is naturally graded:
$V^\kappa(\mf{a}) =\bigoplus_{\Delta\in \Z_{\geq 0}}V^\kappa(\mf{a}) _{\Delta}$,
where the grading is defined by setting 
$\deg xt^n=-n$,
$\deg |0\ket=0$.
Here $|0\ket=1\*1$.
We have $V^\kappa(\mf{a})_0=\C|0\ket$.
We identify $\mf{a}$ with $V^\kappa(\mf{a})_1$ via the linear isomorphism
defined by $x\mapsto xt^{-1}|0\ket$.

There is a unique vertex algebra structure on 
$V^\kappa(\mf{a})$ such that
$|0\ket$ is the vacuum vector and 
$$Y(x,z)=x(z):=\sum_{n\in \Z}(xt^{n})z^{-n-1},\quad x\in \mf{a}.$$
(So $x_{(n)}=xt^n$ for $x\in \mf{a}=V^\kappa(\mf{a})_1$, $n\in \Z$).

The vertex algebra $V^\kappa(\mf{a})$ is called the {\em universal affine
 vertex algebra associated with} ($\mf{a}$, $\kappa$).

We have
$F^1 V^\kappa(\mf{a})=\mf{a}[t^{-1}]t^{-2}V^\kappa(\mf{a})$,
and the Poisson algebra isomorphism
\begin{equation}
\begin{split}
\C [\mf{a}^*] &\isomap
R_{V^\kappa(\mf{a})}= V^k(\mf{a})/\mf{g}[t^{-1}]t^{-2}V^\kappa(\mf{a}) 
\\
x_1\dots x_r& \longmapsto \overline{(x_1t^{-1})\dots (x_r t^{-1})\vac}
 \quad (x_i\in \mf{a}).
\end{split}
\end{equation}
Thus
\begin{align*}
X_{V^\kappa(\mf{a})}=\mf{a}^*.
\end{align*}
 We have the isomorphism 
\begin{align}
 \C [J\mf{a}^*] \simeq \gr V^\kappa(\mf{a})
\label{eq:iso-gr-affineVA}
\end{align}
because the
 graded dimensions of  both sides coincide.
Therefore
\begin{align*}
SS(V^\kappa(\mf{a}))=J\mf{a}^*.
\end{align*}

The isomorphism \eqref{eq:iso-gr-affineVA}  follows also from the fact
that
\begin{align*}
 G_p V^\kappa(\mf{a})=U_p(\mf{a}[t^{-1}]t^{-1})
|0\ket,
\end{align*}
where $\{U_p(\mf{a}[t^{-1}]t^{-1})\}$ is the PBW filtration of $U(\mf{a}[t^{-1}]t^{-1})$.

\subsection{Example: simple affine vertex algebras}
\label{subs:simple-affine}
For a finite-dimensional simple Lie algebra $\mf{g}$ and $k\in \C$,
we denote by
$V^k(\mf{g})$ the universal affine vertex algebra 
$V^{k\kappa_0}(\mf{g})$,
where $\kappa_0$ is the normalized invariant inner product of $\mf{g}$,
that is,
 $$\kappa_0(\theta,\theta)=2,$$ where $\theta$ is the highest root of $\mf{g}$.
Denote by $V_k(\mf{g})$ the unique simple graded quotient of
$V^k(\mf{g})$.
As a
 $\hat{\mf{g}}$-module, 
$V_k(\mf{g})$
is isomorphic to the irreducible highest weight representation
$L(k\Lambda_0)$
of $\affg$ with highest weight $k\Lam_0$,
where $\Lam_0$ is the weight of the basic representation of $\affg$.

\begin{thm}
\label{t:6}
The vertex algebra $V_k(\mf{g})$ is lisse
if and only if $V_k(\mf{g})$ is integrable as
a  $\hat{\mf{g}}$-module, which is true if and only if $k \in \Z_+$.
\end{thm}

\begin{lem}\label{lem:radical}
 Let $(R, \partial)$ be a differential algebra over $\mathbb{Q}$, $I$ a differential   ideal of $R$, i.e.,
$I$ is an ideal of $R$ such that
 $\partial I \subset I$. Then $\partial \sqrt{I} \subset \sqrt{I}$.
\end{lem}

\begin{proof}
Let $a\in \sqrt{I}$, so that $a^m \in I$ for some $m\in \N= \{1,2,\ldots\}$. Since $I$ is  $\partial$-invariant, we have $\partial^m a^m \in I$. But 
$$\partial^m a^m= \sum_{0 \leq i \leq m}{m \choose i}a^{m-i} (\partial
 a)^i
\equiv m!(\partial a)^m\pmod{\sqrt{I}}.$$
Hence
$(\partial a)^m\in \sqrt{I}$,
and therefore,
$\partial a\in \sqrt{I}$.
\end{proof}
\begin{proof}[Proof of the ``if'' part of Theorem \ref{t:6}]
 Suppose that
 $V_k(\fing)$ is integrable.
 This condition
 is equivalent to that
 $k\in \Z_+$  and
the maximal submodule 
$N_k$ of $V^k(\fing)$ is generated by the singular vector
$(e_{\theta}t^{-1})^{k+1}\vac$ (\cite{Kac90}).
The exact sequence
$0\ra N_k\ra V^k(\fing)\ra V_k(\fing)\ra 0$
induces the exact sequence
\begin{align*}
0\ra  I_k\ra R_{V^k(\fing)}
\ra R_{V_k(\fing)}\ra 0,
\end{align*}
where 
$I_k$ is the image of $N_k$ in 
$R_{V^k(\fing)}=\C[\fing^*]$,
and so,
$R_{V_k(\fing)}=\C[\fing^*]/I_k$.
The image of the singular  vector in $I_k$ is given by
$e_{\theta}^{k+1}$.
Therefore,
 $e_{\theta}\in \sqrt{I}$.
 On the other hand,
 by  Lemma \ref{lem:radical},
 $\sqrt{I_k}$ is preserved by the adjoint action of $\fing$.
  Since  $\fing$ is simple,
 $\fing\subset \sqrt{I}$.
This proves  that $X_{V_k(\fing)}=\{0\}$ as required.
\end{proof}
The proof of ``only if'' part
follows from
\cite{DonMas06}.
We will give a different proof
 using $W$-algebras in Remark \ref{Rem:minimal}.
 
 In view of Theorem \ref{t:6},
 one may regard
the lisse condition as  a generalization of the integrability condition
to an arbitrary vertex algebra.
\section{Zhu's algebra}
\label{sec:Zhu}
In this section we will introduce and study the Zhu's algebra of a vertex algebra,
which plays an important role in the representation theory.

See \cite{KacPisa} in this volume for the definition of modules over vertex algebras.

\subsection{Zhu's $C_2$-algebra and Zhu's algebra of a vertex algebra.}
\label{subsection:Zhu}
Let
$V$ be a  $\Z$-graded 
vertex algebra. 
  \emph{Zhu's algebra} $\Zhu V$ \cite{FreZhu92,Zhu96} is defined as 
\begin{align*}
\Zhu (V):= V/V\circ
 V
\end{align*}
 where $V \circ V :=\haru \{a\circ b | a, b \in V\}$ and 
\begin{align*}
a\circ b:=
 \sum_{i\geq 0}{\Delta_a \choose i}a_{(i-2)}b
\end{align*}for homogeneous elements
 $a, b$ and extended linearly. It is an associative algebra with
 multiplication defined as 
\begin{align*}
a\ast b:= \sum_{i\geq 0}{\Delta_a \choose
 i}a_{(i-1)}b
\end{align*}
for homogeneous elements $a,b\in V$.

For a simple positive energy representation
$M=\bigoplus_{n\in
 \Z_+}M_{\lam+n}$,
 $M_{\lam}\ne 0$, of $V$,
let $M_{top}$ be the top degree component $M_{\lam}$ of $M$.
Also,
for 
a homogeneous  vector $a\in V$,
let $o(a)=a_{(\Delta_a-1)}$,
so that $o(a)$ preserves the homogeneous  components of any graded
 representation of $V$.

The importance of Zhu's algebra in vertex algebra theory is the
following fact that was established by Yonchang Zhu.
\begin{thm}[\cite{Zhu96}]
 For any positive energy representation $M$ of $V$,
$\overline{a}\mapsto o(a)$ defines a well-defined representation of
 $\Zhu (V)$ on $M_{top}$.
Moreover, the correspondence $M\mapsto M_{top}$ gives a bijection
 between the set of isomorphism classes of irreducible positive energy
representations of $V$ and that of simple $\Zhu (V)$-modules.
\end{thm}
A vertex algebra $V$ is called a {\em chiralization} of an algebra $A$
if $\Zhu (V)\cong A$.

Now we define an increasing filtration of Zhu's algebra. 
For this, we assume that $V$ is $\Z_+$-graded:
$V= \bigoplus_{\Delta\geq 0}V_\Delta$.
Then $V_{\leq p}=\bigoplus _{\Delta= 0}^p V_\Delta$
gives an increasing filtration of $V$.
Define $$\Zhu _p(V):= \im(V_{\leq p}\ra \Zhu( V)).$$
Obviously, we have  
$$0=\Zhu_{-1}(V) \subset  \Zhu_0(V) \subset \Zhu_1(V) \subset \cdots, \quad
\mbox{and} \quad  \Zhu (V)= \bigcup_{p\geq -1} \Zhu_p(V).$$
Also, since
$a_{(n)}b\in V_{\Delta_a+\Delta_b-n-1}$ for $a\in V_{\Delta_a}$, $b\in
V_{\Delta_b}$,
we have
\begin{align}\label{eq:Zhu-mul}
\Zhu_p(V) \ast \Zhu_q(V) \subset \Zhu_{p+q}(V).
\end{align}
The following assertion follows from the
skew symmetry.
\begin{lem}\label{lem:commu-ZHu}
We have
\begin{equation*}
b \ast a\equiv \sum_{i\geq 0}{\Delta_a -1 \choose i}a_{(i-1)}b  \quad \pmod{ V\circ V},
\end{equation*}
and hence,
\begin{equation*}
\begin{split}
a \ast b -b\ast a &\equiv \sum_{i\geq 0}{\Delta_a-1 \choose i}a_{(i)}b  \quad
\pmod{ V\circ V}.
\end{split}
\end{equation*}
 \end{lem}
By Lemma \ref{lem:commu-ZHu},
we have
\begin{align}\label{Zhu:comm}
[\Zhu_p(V), \Zhu_q(V)] \subset \Zhu_{p+q-1}(V).
\end{align} 
By \eqref{eq:Zhu-mul}
and \eqref{Zhu:comm},
the associated graded
$\gr \Zhu (V)=\bigoplus_p \Zhu_p (V)/\Zhu_{p-2}(V)$
is naturally a graded Poisson algebra.

Note that $a\circ b\equiv a_{(-2)}b\pmod{\bigoplus_{\Delta\leq \Delta_a+\Delta_b}V_{\Delta}}$ for 
homogeneous elements $a,b\in V$.

 \begin{lem}[Zhu, see  {\cite[Proposition 2.17(c)]{De-Kac06}, \cite[Proposition 3.3]{ALY}}]
  \label{lem:Zhu-vs-ZhusC2}
The following map defines a well-defined surjective homomorphism of Poisson algebras. 
\begin{equation*}
\begin{split}
\eta_V:R_V &\longrightarrow \gr \Zhu (V) \\
\bar{a} &\longmapsto a \pmod{V\circ V + \bigoplus_{\Delta < \Delta_a}V_\Delta}.
\end{split}
\end{equation*} 
\end{lem}

\begin{rmk}
The  map $\eta_V$ is not an isomorphism
 in general. For an example,
let
 $\mf{g}$ be the simple Lie algebra of type $E_8$ and $V=V_1(\mf{g})$. Then $\dim R_V > \dim  \Zhu V =1$.
\end{rmk}
 

\begin{cor}
 If $V$ is lisse then $\Zhu V$ is finite dimensional.
 Hence the number of isomorphic classes of simple positive energy representations of $V$ is finite.
\end{cor}


In fact the following stronger facts are known.
 \begin{thm}[\cite{AbeBuhDon04}]
  Let $V$ be lisse.
  Then any simple $V$-module is a positive energy representation.
  Therefore the number of isomorphic classes of simple $V$-modules  is finite.
 \end{thm}

 \begin{thm}[\cite{Dong:1998kq,MatNagTsu05}]
  Le
  $V$ be lisse.
  Then the abelian category of $V$-modules is equivalent to the module category of a finite-dimensional associative algebra.
 
\end{thm}
\subsection{Computation of Zhu's algebras}
We say that a vertex algebra $V$ {\em admits a PBW basis} if $R_V$ is a polynomial algebra and 
the map $\C[JX_V]\twoheadrightarrow  \gr V$ is an isomorphism.
\begin{thm}
\label{t:10} 
If $V$ admits a $PBW$ basis, then 
$\eta_V: R_V \twoheadrightarrow \gr \Zhu V$ is an isomorphism.
\end{thm}
\begin{proof}
We have $\gr \Zhu(V)=V/\gr (V\circ V)$,
where $\gr (V\circ V)$ is the associated graded space of $V\circ V$ with respect to the filtration
induced  by the filtration $V_{\leq p}$.
We wish to show that
$\gr (V\circ V)=F^1 V$.
Since $a\circ b\equiv a_{(-2)}b\pmod {F_{\leq \Delta_a+\Delta_b}V}$,
it is sufficient to show that $a\circ b\ne 0$ implies that $a_{(-2)}b\ne 0$.

Suppose that 
$a_{(-2)}b=(Ta)_{(-1)}b=0$.
Since
$V$ admits a PBW basis,
$\gr V$ has no zero divisors.
That fact that $V$ admits a PBW basis
also 
shows that $Ta=0$ implies that $a=c|0\ket$ for some constant $c\in \C$.
Thus, 
 $a$  is a constant multiple of $|0\ket $, in which case $a\circ b=0$.
\end{proof}

\begin{exm}[Universal affine vertex algebras]
The universal affine vertex algebra
$V^\kappa(\mf{a})$  (see  \S \ref{subs:affineVOA})
admits a PBW basis.
Therefore
\begin{align*}
\eta_{V^\kappa(\mf{a})}
:R_{V^\kappa(\mf{a})}= \C[\mf{a}^*]\isomap \gr \Zhu V^\kappa(\mf{a}).
\end{align*}
On the other hand, 
from Lemma \ref{lem:commu-ZHu}
one finds that
\begin{equation}
 \label{eq:Zhu-universal-affine}
\begin{split}
U(\mf{a}) &\longrightarrow \Zhu (V^\kappa(\mf{a}))\\
\mf{a} \ni x & \longmapsto \bar x=\overline{x_{(-1)}\vac}
\end{split}
\end{equation}
gives  a well-defined algebra homomorphism.
This map  respects the filtration on both sides,
where
the filtration in the left-hand-side is the PBW filtration.
Hence it induces a map between their associated graded algebras,
which  is identical to $\eta_{V^\kappa(\mf{a})}$.
Therefore \eqref{eq:Zhu-universal-affine} is  an isomorphism,
that is to say,
 $V^{\kappa}(\mf{a})$ is a chiralization of $U(\mf{a})$.
\end{exm}

\begin{exr}\label{exr:superPBW}
 Extend Theorem \ref{t:10}
to the case that $\mf{a}$ is a Lie superalgebra.
\end{exr}

\begin{exm}[Free fermions]
\label{eq:free-fermions}
Let 
$\mf{n}$
be a finite-dimensional vector space.
The {\em Clifford affinization} $\hat{Cl}$ of $\mf{n}$
is the 
Clifford algebra
associated with
$\mf{n}[t, t^{-1}]\oplus \mf{n}^\ast[t, t^{-1}]$
and its symmetric bilinear form
defined by
\begin{align*}
 (xt^m|f t^n)=\delta_{m+n,0}f(x),
\ (xt^m|yt^n)=0=(ft^m|gt^n)
\end{align*}
for $x,y\in \finn$,
$f,g\in \finn^*$,
$m,n\in \Z$.

Let $\{x_\alpha\}_{\alpha\in \Delta_+}$ be a basis of $\mf{n}$,
$\{x_\alpha^*\}$ its dual basis.
We write
$\psi_{\alpha,m}$ for 
$x_\alpha t^m \in \hat{Cl}$
and $\psi_{\alpha,m}^*$ for $x_\alpha^*t^m \in \hat{Cl}$,
so that
 $\hat{Cl}$ 
  is the associative superalgebra with 
\begin{itemize}
\item  odd generators: $\psi_{\alpha,m}, \psi_{\alpha,m}^*, m\in \Z,  \alpha \in \Delta_+$.
 \item   relations: $[\psi_{\alpha,m}, \psi_{\beta,n}]= [\psi_{\alpha,m}^*, \psi_{\beta,n}^*]=0, [\psi_{\alpha,m}, \psi_{\beta,n}^*]=
	 \delta_{\alpha, \beta}\delta_{m+n, 0}.$
\end{itemize}
Define the {\em charged fermion Fock space} associated with $\finn$ as
\begin{equation*}
\mc{F}_{\finn}:= \hat{Cl}/(\sum_{\substack{m\geq 0\\\alpha \in \Delta_+}}
\hat{Cl}\psi_{\alpha,m}+\sum_{\substack{k\geq 1\\\beta \in \Delta_+}}\hat{Cl}\psi_{\beta,k}^*).
\end{equation*}
It is an irreducible $\hat{Cl}$-module,
and as $\C$-vector  spaces we have
\begin{equation*}
 \mc{F}_{\finn} \cong  \Lambda(\mf{n}^\ast[t^{-1}])
  \otimes \Lambda(\mf{n}[t^{-1}]t^{-1}).
\end{equation*}

There is a unique vertex (super)algebra structure on
$\mc{F}_{\finn}$ such that
the image of $1$  is the vacuum $|0\ket$  and 
\begin{align*}
& Y(\psi_{\alpha,-1}|0\ket,z)=\psi_\alpha(z):=\sum_{n\in \Z}\psi_{\alpha,n}z^{-n-1},\\
&Y(\psi_{\alpha,0}^*|0\ket,z)=\psi_\alpha^*(z):=\sum_{n\in \Z}\psi_{\alpha,n}^*z^{-n}.
\end{align*}
We have
$F^1 \mc{F}_{\finn}=\finn^\ast[t^{-1}]t^{-1}\mc{F}_{\finn}+\finn[t^{-1}]t^{-2}\mc{F}_{\finn}$,
and  it follows that there is an isomorphism
\begin{align*}
 \begin{array}{ccc}
 \overline{Cl}& \isomap& R_{\mc{F}_{\finn}}, \\
 x_\alpha& \mapsto &\overline{\psi_{\alpha,-1}|0\ket},\\
 x_\alpha^*& \mapsto &\overline{\psi_{\alpha,0}^*|0\ket}
 \end{array}
\end{align*}
as Poisson superalgebras.
Thus,
\begin{align*}
 X_{\mc{F}_{\finn}}=T^* \Pi \finn,
\end{align*}
where $\Pi \finn$ is the space $\finn$ considered as a purely odd
 affine space.
The arc space
 $J T^*\Pi \finn$ 
 is also
 regarded as a purely odd affine space,
 such that
 $\C[JT^*\Pi \finn]=  \Lambda(\mf{n}^\ast[t^{-1}])
  \otimes \Lambda(\mf{n}[t^{-1}]t^{-1})$. 
The map
$\C[JX_{\mc{F}_{\finn}}]\ra \gr \mc{F}_{\finn}$ is an isomorphism and
 $\mc{F}_{\finn}$ admits a PBW basis.
Therefore we have the isomorphism
 \begin{align*}
  \eta_{\mc{F}_{\finn}}:R_{\mc{F}_{\finn}}=\overline{Cl}\isomap \Zhu (\mc{F}_{\finn})
 \end{align*}
by Exercise \ref{exr:superPBW}.
On the other hand the map
\begin{align*}
\begin{array}{ccc}
 Cl&\ra &\Zhu (\mc{F}_{\finn})\\
x_\alpha &\mapsto &\overline{\psi_{\alpha,-1}|0\ket}, \\
x_\alpha^* &\mapsto &\overline{\psi_{\alpha,0}^*|0\ket} \\
\end{array}
\end{align*}
gives an algebra homomorphism that respects the filtration.
Hence we have
\begin{align*}
 \Zhu (\mc{F}_{\finn})\cong Cl.
\end{align*}
That is,
 $\mc{F}_{\finn}$ is a chiralization of $Cl$.
\end{exm}

%
%
%
%
%

\section{$W$-algebras}\label{sec:W-algbeas}
We are now in a position to define $W$-algebras. We will construct a differential graded vertex algebra, so that  its cohomology algebra is
 a vertex algebra and that will be our main object to study.

For simplicity, we let $\mf{g}=\mf{gl}_n$ and we only consider the
principal nilpotent case.
However the definition works for any simple Lie algebra.
 The general definition for an arbitrary nilpotent element  will be similar but
one does need a new idea 
(see \cite{KacRoaWak03} for the most general definition). 

\subsection{The BRST complex}
Let $\fing, \finn$ be as in \S \ref{sec:1}.
Denote by $\kappa_\fing$ the
Killing form on $\fing$
and $\kappa_0=\frac{1}{2n}\kappa_{\fing}$,
so that $\kappa_0(\theta,\theta)=2$.

Choose  any symmetric invariant bilinear form $\kappa$ on $\fing$
and let $V^\kappa(\fing)$ be the universal affine vertex algebra 
associated with $(\fing, \kappa)$
(see \S \ref{subs:affineVOA})
and let $\Fock=\Fock_{\finn}$  be the fermion Fock space as in Example
\ref{eq:free-fermions}.

We have the following commutative diagrams: 
\begin{align*}
 \xymatrix{
\C[J\fing^*]\ar[d]_{\on{Zhu}(?)} & V^{\kappa}(\fing)\ar[d]^{\on{Zhu}(?)}
 \ar[dl]_{R_{?}} \ar[l]_{\on{gr}(?)} \\
 \C[\fing^*]
 &  
U(\fing),
 \ar[l]^{\gr(?)}
}
\quad
 \xymatrix{
\C[JT^* \Pi \finn]\ar[d]_{\on{Zhu}(?)} & \Fock\ar[d]^{\on{Zhu}(?)}
 \ar[dl]_{R_{?}} \ar[l]_{\on{gr}(?)} \\
 \overline{Cl}
 &  
Cl
 \ar[l]^{\gr(?)}
}
\end{align*}

Define
\begin{align*}
 \Caff(\mf{g}):= V^{\kappa}(\mf{g}) \otimes \mc{F}.
\end{align*}
Since it is a tensor product of two vertex algebras,
$\Caff(\mf{g})$ is a vertex algebra.
We have
\begin{equation*}
R_{\Caff(\mf{g})}=R_{V^{\kappa}(\mf{g})} \otimes R_{\mc{F}}=\C[\mf{g}^\ast] \otimes \overline{Cl}=\bar{C}(\mf{g}),
\end{equation*}
and 
\begin{equation*}
\Zhu \Caff(\mf{g})=\Zhu V^{\kappa}(\mf{g}) \otimes \Zhu \mc{F}=U(\mf{g}) \otimes Cl=C(\mf{g}).
\end{equation*}
Thus,
$\Caff(\fing)$ is a chiralization of $C(\fing)$
considered in \S \ref{subsection:Quantized Hamiltonian reduction}.
Further we have
\begin{align*}
 \gr \Caff (\fing)=\gr V^{\kappa}(\fing)\* \gr \mc{F}=\C[J\fing^*]\otimes \C[JT^* \Pi \finn].
\end{align*}
So we have the following commutative diagram:
\begin{align*}
 \xymatrix{
& \Caff(\fing)\ar[d]^{\on{Zhu}(?)}
 \ar[dl]_{R_{?}}  \\
\bar C(\fing)
 &  
 C(\fing)
 \ar[l]^{\gr(?)}
}
\end{align*}


Define a gradation 
\begin{equation}
\mc{F}=\bigoplus_{p \in \Z}\mc{F}^p
\end{equation}
by setting
$\deg \psi_{\alpha,m}=-1, \deg
 \psi_{\alpha,k}^*=1, \forall i, j \in I, m, k\in \Z$,
$\deg |0\ket =0$.
This induces a
 $\Z$-grading (that is different from the conformal grading)
on $\Caff(\fing)$:
\begin{equation}
\Caff(\mf{g})= V^{\kappa}(\mf{g}) \otimes \mc{F}
 =\bigoplus_{p \in \Z}C^{\kappa,
 p}(\mf{g}),
 \quad \mbox{where} \quad C^{\kappa, p}(\mf{g}):= V^{\kappa}(\mf{g}) \otimes \mc{F}^p.
\end{equation}

Let $V(\finn)$ be the 
the universal affine vertex algebra associated with
$\finn$ and the zero bilinear form,
which  is  identified with the vertex subalgebra 
of $V^{\kappa}(\fing)$ generated by 
$x_{\alpha}(z)$ with $\alpha\in \Delta_+$.

\begin{lem}\label{lem:chiralization-of-va-from-clifford}
The following defines a vertex algebra homomorphism.
\begin{equation*}
\begin{split}
\hat{\rho} : V(\mf{n}) &\longrightarrow \mc{F}\\
       x_\alpha(z) &\longmapsto \sum_{\beta, \gamma\in
 \Delta_+}c_{\alpha \beta}^\gamma \psi_\beta^\ast(z) \psi_\gamma(z).
\end{split}
\end{equation*}
\end{lem}
 \begin{rmk}
  In the above formula
 the normally ordered product is not needed because $\finn$ is nilpotent.
 \end{rmk}

The map $\hat\rho$ induces an algebra homomorphism
\begin{align*}
\Zhu V(\finn)=U(\finn)\ra \Zhu \mc{F}=Cl
\end{align*}
and a Poisson algebra homomorphism 
\begin{align*}
R_{V(\finn)}=\C[\finn^*]\ra
R_{\mc{F}}=\overline{Cl}
\end{align*}that are identical to $\rho$ and $\bar \rho$
(see Lemma \ref{lem:Lie-alg-hom-Clifford} and \ref{eq:Lie-alg-hom-Poisson}), respectively.

Recall  the character $\chi:\finn\ra \C$,
$x\mapsto (f|x)$.
\begin{lem}\label{lem:theta-hat}
The following defines a vertex algebra homomorphism.
 \begin{equation*}
\begin{split}
\hat{\theta}_{\chi} : V(\mf{n}) &\longrightarrow \Caff(\fing)\\
       x_\alpha(z) &\longmapsto (x_\alpha(z) + \chi(x_\alpha))\otimes \id +\id\otimes \hat{\rho}(x_\alpha(z)).
\end{split}
\end{equation*}
\end{lem}

The map $\hat\theta_{\chi}$ induces an algebra homomorphism
\begin{align*}
\Zhu V(\finn)=U(\finn)\ra \Zhu \Caff(\fing)=C(\fing)
\end{align*}
and a Poisson algebra homomorphism 
\begin{align*}
R_{V(\finn)}=\C[\finn^*]\ra
R_{\mc{F}}=\overline{C(\fing)}
\end{align*}that are identical to $\theta_{\chi}$ and $\bar \theta$,
respectively
(see Lemmas \ref{lem:moment} and \ref{lem:quantized-moment-map}).

The proof of the following assertion is similar to that of 
 Lemma \ref{lem:barQ}.

\begin{prp}\label{Pro:Q-theta-hat}
There exists a  unique element $\hat{Q} \in C^{k, 1}(g)$ such that
\begin{align*}
[\hat{Q}_{\lam}  (1\*\psi_\alpha)]=\hat{\theta}_{\chi} (x_\alpha),\quad
 \forall \alpha\in \Delta_+.
\end{align*}
We have $[\hat{Q}_{\lam} \hat{Q}]=0$.
\end{prp}
The field $\hat{Q}(z)$ is given explicitly as
\begin{equation*}
\hat{Q}(z)
= \sum_{\alpha \in \Delta_+}(x_\alpha+ \chi (x_\alpha))\otimes
\psi_\alpha^\ast (z) - 
\id \otimes \frac{1}{2}\sum_{\alpha, \beta, \gamma \in \Delta_+}
c_{\alpha, \beta}^\gamma \psi_\alpha^\ast (z)\psi_\beta^\ast (z)\psi_\gamma (z).
\end{equation*}
 Since
$\hat{Q}$ is odd and $[\hat{Q}_\lambda \hat{Q}]=0$, we have 
\begin{align*}
\hat{Q}_{(0)}^2=0.
\end{align*}
(Recall  that we write  $\hat{Q}(z)=\sum_{n\in \Z}\hat{Q}_{(n)}z^{-n-1}
$.)
So $(\Caff(\mf{g}), \hat{Q}_{(0)})$ is a cochain complex.

\begin{lem}
If it is nonzero,
the cohomology $H^\bullet(\Caff(\mf{g}), \hat{Q}_{(0)})$ 
inherits the vertex algebra structure from $\Caff(\mf{g})$.
\end{lem}
\begin{proof}
Set
 $Z:=\{v\in \Caff(\fing)\mid \hat{Q}_{(0)}v=0\}$,
$B=\hat{Q}_{(0)}\Caff(\fing)\subset Z$,
so that
 $H^\bullet(\Caff(\mf{g}), \hat{Q}_{(0)})=Z/B$.
From the commutator   formula \eqref{eq:com-formula},
we know that
\begin{equation*}
[\hat{Q}_{(0)},a_{(m)}]=(\hat{Q}_{(0)}a)_{(m)}\quad \forall a\in
 \Caff(\fing),
m\in \Z.
\end{equation*}
Thus,
if $a,b\in Z$,
then $ \hat Q_{(0)}(a_{(m)}b)=0$,
that is,
$a_{(m)}b\in Z$.
It follows that $Z$
 a vertex subalgebra of $\Caff(\fing)$.
Further, if $a\in Z$
 and $b=\hat Q_{(0)}b'\in B$,
then $ a_{(m)}b=a_{(m)}\hat Q_{(0)}b'=\hat{Q}_{(0)}(a_{(m)}b)\in B$.
Hence $B$
is an ideal of  $Z$.
This completes the proof.
\end{proof}

\begin{defn}
The {\em $W$-algebra} $\W^\kappa(\fing)=\W^\kappa(\mf{g},f)$ associated to $(\mf{g}, f, \kappa)$ 
 is defined to be the zero-th cohomology of the cochain complex
 $(\Caff(\mf{g}), \hat{Q}_{(0)})$,
that is,
\begin{equation*}
\W^\kappa(\mf{g}): =H^0(\Caff(\mf{g}), \hat{Q}_{(0)}).
\end{equation*}
\end{defn}

This definition of $\W^\kappa(\fing)$
is due to Feigin and Frenkel \cite{FF90}.
In \S \ref{subsection:Miura}
we show that the above   $\W^\kappa(\fing)$
 is identical to the 
original
$W$-algebra defined by Fateev and Lukyanov \cite{FatLyk88}.

\subsection{Cohomology of associated graded}

We have $\hat{Q}_{(0)} F^p\Caff(\mf{g}) \subset F^p\Caff(\mf{g})$, so
$(\gr^F\Caff(\mf{g}), \hat{Q}_{(0)})$ is also a cochain complex.
The cohomology $H^\bullet(\gr^F\Caff(\mf{g}), \hat{Q}_{(0)})$ inherits a Poisson vertex algebra structure from $\gr^F\Caff(\mf{g})$.

\begin{thm}\label{them:vanishing-associated-graded-BRST}
We have
$H^{i}(\gr^F\Caff(\mf{g}), \hat{Q}_{(0)})=0$ 
 for $i\ne 0$ and
 \begin{align*}
H^0(\gr^F\Caff(\mf{g}), \hat{Q}_{(0)}) \cong
 \C[J\mc{S}]
 \end{align*}as Poisson vertex algebras,
 where $\mc{S}$ is the  slice defined in \S \ref{sec:Review of Kostant's results}.
\end{thm}

 \begin{proof}
  The proof is an arc space analogue of that
  of
Theorem \ref{thm:KS-classical}.
  
The moment map $\mu: \mf{g}^\ast \longrightarrow \mf{n}^\ast$ for
 the $N$-action on $\mf{g}$ induces a $JN$-equivariant morphism
\begin{align*}
J\mu: J\mf{g}^\ast  \longrightarrow J\mf{n}^\ast.
\end{align*}
The pullback
 $(J{\mu})^*: \C[J\finn^*]\ra \C[ J\fing^*]$ is an embedding of vertex
 Poisson
algebras.

The point 
$\chi=J\chi$ of $J\finn^*$ corresponds to
the arc
$\alpha\in \Hom (D,\mf{n}^\ast) =\Hom (\C[\mf{n}^\ast], \C[[t]])
$ 
such that $\alpha(f)=\chi(x)$ for $x\in \finn \subset \C[\mf{n}^\ast]$.

We have
\begin{align*}
 (J\mu)^{-1}(\chi)=J (\mu^{-1}(\chi))=\chi+ J\finb\subset J \fing^*,
\end{align*}
and the adjoint action gives the isomorphism
\begin{align}
JN\times J \mc{S}\isomap  J\mu^{-1}(\chi)
\label{eq:jets-of-product}
\end{align}
by Theorem \ref{them:Kostant}
and \eqref{eq:product-of-jets}.

Now put
\begin{equation}
 C:=\gr \Caff(\mf{g})= \C[J\mf{g}^\ast]\otimes \Lambda (\mf{n}[t^{-1}]t^{-1})
  \otimes \Lambda
  (\mf{n}^\ast[t^{-1}])
\end{equation}
and define a bigrading on $C$ 
by
\begin{equation}
 C=\bigoplus_{i\leq 0, j\geq 0}C^{i, j}, \quad \mbox{where}\quad C^{i, j} =\C[J\mf{g}^\ast]\otimes
  \Lambda^{-i} (\mf{n}[t^{-1}]t^{-1})\otimes \Lambda^j(\mf{n}^\ast[t^{-1}]).
 \end{equation}

  As before,
  we can decompose the operator $\hat{Q}_{(0)}$ as the sum of two suboperators such that each of them preserves one grading but increase the other grading by $1$. Namely, we have
\begin{equation*}
\begin{split}
&\hat{Q}_{(0)}=\hat{d}_++\hat{d}_-,\\
&\hat{d}_-: C^{i, j} \longrightarrow C^{i, j+1},
\quad \hat{d}_+: C^{i, j} \longrightarrow C^{i+1, j}.
\end{split}
\end{equation*}
This shows that
\begin{equation*}
(\hat{d}_+)^2=(\hat{d}_-)^2=[\hat{d}_+, \hat{d}_-]=0.
\end{equation*}
Thus we can get a spectral sequence
$ E_r \Longrightarrow H^\bullet (C, \hat{Q}_{(0)})$
such that
\begin{equation*}
E_1= H^\bullet (C, \hat{d}_-),\quad
E_2= H^\bullet ( H^\bullet (C, \hat{d}_-), \hat{d}_+).
\end{equation*}
  This is a converging spectral sequence since
  $C$ is a direct sum of subcomplexes $F^p \Caff(\fing)/F^{p+1}\Caff(\fing)$,
  and the associated filtration is regular on each subcomplex.

The complex
$(C,\hat d_-)$ is the Koszul complex
with respect to the sequence
\begin{align*}
x_1t^{-1}-\chi(x_1),\dots,
x_{N}t^{-1}-\chi(x_{N}),
x_1t^{-2},x_2t^{-2},\dots, x_{N}t^{-2},
 x_1t^{-3}, x_2t^{-3},\dots
\end{align*}
where $N=\dim \finn$.
Hence we have 
\begin{equation}
H^i (C, \hat{d}_-)
=\delta_{i, 0}\C[J\mu^{-1}(\chi)]\otimes \Lambda(\mf{n}^\ast[t^{-1}]).
\label{eq:E1}
\end{equation}
Next, by \eqref{eq:E1},
the complex 
$(H^0 (C, \hat{d}_-),\hat d_+)$
is identical to the Chevalley complex for  
the Lie algebra cohomology 
$H^{\bullet}(J\finn,\C[J\mu^{-1}(\chi)])=H^{\bullet}(\finn[[t]],\C[J\mu^{-1}(\chi)])$.
By \eqref{eq:jets-of-product},
\begin{align*}
H^{i}(J\finn,\C[J\mu^{-1}(\chi)])=H^i(J\finn,\C[JN]\* \C[J \mc{S}])\\
=H^i(J\finn,\C[JN])\* \C[J \mc{S}]=\delta_{i,0}\C[J\mc{S}].
\end{align*}
We conclude that 
\begin{align*}
H^{i}(H^j (C, \hat{d}_-),\hat d_+)=\delta_{i,0}\delta_{j,0}\C[\mc{S}].
\end{align*}
Thus, the spectral sequence $E_r$ collapses at $E_2=E_{\infty}$,
 and we get the desired isomorphisms. 
 \end{proof}

\begin{thm}[\cite{FF90,FreBen04}]\label{thm:vanishing-W}
We have 
$H^{0}(\Caff(\mf{g}), \hat{Q}_{(0)})=0$ for $i\ne 0$
and 
 \begin{align*}
\gr \W^k(\mf{g})=\gr H^0(\Caff(\mf{g}), \hat{Q}_{(0)}) \cong H^0(\gr \Caff(\mf{g}), \hat{Q}_{(0)})=\C[J\mc{S}].
\end{align*}
In particular, $R_{W^k(\mf{g})}\cong \C[\mc{S}] \cong \C[\mf{g}]^G$,
so $\tilde{X}_{\W^k(\fing)}=\mc{S}$,
$SS(\W^k(\fing))=J\mc{S}$.
\end{thm}

The proof of Theorem \ref{thm:vanishing-W} will be given in \S \ref{Proof of Theorem {thm:vanishing-W}}.

Note that
there is a spectral sequence for $H^\bullet(\Caff(\mf{g}), \hat{Q}_{(0)})$
such that
$E_1^{\bullet,q}=H^q(\gr \Caff(\mf{g}),\hat{Q}_{(0)})$.
Hence 
Theorem \ref{thm:vanishing-W} would immediately follow from Theorem \ref{them:vanishing-associated-graded-BRST}
if this spectral sequence converges.
However, this is not clear at this point because our algebra is not Noetherian.

\begin{Rem}
The complex
$(\Caff(\fing),\hat Q_{(0)})$
is identical to Feigin's standard complex for the semi-infinite $\finn[t,t^{-1}]$-cohomology
$H^{\frac{\infty}{2}+\bullet}(\finn[t,t^{-1}],V^{\kappa}(\fing)\* \C_{\hat\chi})$
with coefficient in the $\fing[t,t^{-1}]$-module  $V^{\kappa}(\fing)\* \C_{\hat\chi}$ (\cite{Feu84}),
where $\C_{\hat \chi}$ is the one-dimensional representation of $\finn[t,t^{-1}]$ defined by the character
$\hat \chi:\finn[t,t^{-1}]\ra \C$, $xt^n\mapsto \delta_{n,-1}\chi(x)$:
 \begin{align}
H^{\bullet}(\Caff(\fing),\hat Q_{(0)})\cong H^{\frac{\infty}{2}+\bullet}(\finn[t,t^{-1}],V^{\kappa}(\fing)\* \C_{\hat \chi}).
 \label{eq:deltaef-of-J}
 \end{align}
\end{Rem}

\subsection{$W$-algebra associated with $\mf{sl}_n$}

It is straightforward to generalize the above definition to
an arbitrary simple Lie algebra $\fing$.
In particular, by replacing  $V^{\kappa}(\mf{gl}_n)$ with  $V^{k}(\mf{sl}_n)$, $k\in \C$,
we  define
the $W$-algebra
\begin{align}
 \W^k(\mf{sl}_n):=H^0(C^{k}(\mf{sl}_n),\hat {Q}_{(0)})
 \label{eq:sln-W}
\end{align}
associated with $(\mf{sl}_n,f)$ at level $k$.

We have 
$V^{\kappa}(\mf{gl}_n)= \pi_{\kappa}\otimes V^{\kappa}(\mf{sl}_n)$,
where $\kappa|_{\mf{sl}_n\times \mf{sl}_n}=k\kappa_0$
and $\pi_{\kappa}$ is the rank $1$ {\em Heisenberg vertex algebra} generated by
$I(z)=\sum_{i=1}^n e_{ii}(z)$
with $\lambda$-bracket $[I_{\lam}I]=\kappa(I,I)\lam$.
It follows that
$C^k(\mf{gl}_n)=\pi_{\kappa}\otimes C^k(\mf{sl}_n)$.
As easily seen,
$\hat{Q}_{(0)}I=0$.
Hence
$H^{\bullet}(\Caff(\mf{gl}_n))=\pi_{\kappa}\otimes H^{\bullet}(C^k(\mf{sl}_n))$,
so that
\begin{align*}
\W^\kappa(\mf{gl}_n)=\W^k(\mf{sl}_n)\otimes \pi_{\kappa}.
\end{align*}
In particular if we choose the
form $\kappa$ to be $k\kappa_0$,
we find that
$\pi_{\kappa}$ belongs to the center of $\W^\kappa(\mf{gl}_n)$
as $\pi_{\kappa}$ belongs to the center of $\Caff(\mf{gl}_n)$.
Thus,
$\W^k(\mf{sl}_n)$ is isomorphic to the
quotient of $\W^{k\kappa_0}(\mf{gl}_n)$ by the ideal generated by $I_{(-1)}|0\ket$.

\subsection{The grading of  $\W^\kappa(\fing)$}
The standard conformal grading of $\Caff(\fing)$ is given by the
Hamiltonian
$H$ defined by
\begin{align*}
H|0\ket =0,\quad[H,x_{(n)}]=-n x_{(n)}\quad(x\in \fing),\\
\quad [H,\psi_{\alpha,n}]=-n\psi_{\alpha,n},\quad [H,\psi_{\alpha,n}^*]=-n\psi_{\alpha,n}^*.
\end{align*}
However
$H$ is not well-defined in $\W^\kappa(\fing)$ since $H$ does not commute
with the action of 
$$\hat{Q}_{(0)}=\sum_{\alpha\in \Delta_+}\sum_{k\in \Z}(x_{\alpha})_{(-k)}\psi_{\alpha,k}
+\sum_{\alpha\in \Delta_+}\chi(x_{\alpha})\psi_{\alpha,1}
-\sum_{\alpha,\beta,\gamma\in \Delta_+}\sum_{k+l+m=0}c_{\alpha,\beta}^{\gamma}\psi_{\alpha,k}^*\psi_{\beta,l}^*
\psi_{\gamma,m}.$$
Here and below we omit the tensor product sign.

To remedy this,
define the linear operator $H_\W$
by
\begin{align*}
H_\W|0\ket =0,
\quad 
[H_\W,(x_{i})_{(n)}]=-n (x_{i})_{(n)}\quad(i\in I),\\
\quad[H_\W,(x_{\alpha})_{(n)}]=(\alpha(\rho^{\vee})-n) (x_{\alpha})_{(n)}\quad(\alpha\in \Delta),\\
\quad [H_\W,\psi_{\alpha,n}]=(\alpha(\rho^{\vee})-n)\psi_{\alpha,n},\quad [H_\W,\psi_{\alpha,n}^*]=(-\alpha(\rho^{\vee})-n)\psi_{\alpha,n}^*,
\quad (\alpha\in \Delta_+).
\end{align*}
Here $\rho^{\vee}=1/2h$, where $h$ is defined in \eqref{eq:sl2-triple}.
Set $\Caff(\fing)_{\Delta,new}=\{v\in \Caff(\fing)\mid H_\W c=\Delta c\}$.
Then
\begin{align}
\Caff(\fing)=\bigoplus_{\Delta\in \Z}\Caff(\fing)_{\Delta,new}.
\label{eq:newgrading-of-C}
\end{align}
Since 
$[\hat{Q},H_\W]=0$,
$\Caff(\fing)_{\Delta,new}$ is a subcomplex of $\Caff(\fing)$.
We have
\begin{align*}
H^{\bullet}(\Caff(\fing),\hat{Q}_{(0)})=\bigoplus_{\Delta\in \Z}H^{\bullet}(\Caff(\fing),\hat{Q}_{(0)})_{\Delta},
\quad H^{\bullet}(\Caff(\fing),\hat{Q}_{(0)})_{\Delta}=H^{\bullet}(\Caff(\fing)_{\Delta,new},\hat{Q}_{(0)}).
\end{align*}
In particular
$\W^\kappa(\fing)=\bigoplus_{\Delta\in \Z}\W^\kappa(\fing)_{\Delta}$.
Note that the grading \eqref{eq:newgrading-of-C}
is not bounded from below.

If $k\ne -n$ then the action of $H_\W$  on the vertex subalgebra  $\W^k(\mf{sl}_n)$ of $ \W^k(\fing)$ is inner:
Set
\begin{align*}
L(z)=L_{sug}(z)+\rho^{\vee}(z)+L_{\mc{F}}(z)=\sum_{n\in \Z}L_nz^{-n-1},
\end{align*}
where $L_{sug}(z)$ is the Sugawara field of $V^{k}(\mf{sl}_n)$:
$$L_{sug}(z)=\frac{1}{2(k+n)}\sum_{a}:x_a(z)x^a(z):,$$
and
$$L_{\mc{F}}(z)=\sum_{\alpha\in \Delta_+}(\on{ht}(\alpha):\partial_z \psi_{\alpha}(z)\psi_{\alpha}^*(z):
+(1-\on{ht}(\alpha)):\partial_z \psi_{\alpha}^*(z)\psi_{\alpha}(z):).$$
Here $\{x_a\}$ is a basis of $\mf{sl}_n$ and $\{x^a\}$ is the dual basis of $\{x_a\}$ with respect to $(~|~)$.
Then $\hat Q_{(0)}L=0$,
and so
$L$ defines an element of $\W^k(\mf{sl}_n)$.
It is a conformal vector of $\W^k(\mf{sl}_n)$,
that is to say,
$L_0=H_\W$ and $L_{-1}=T$ 
and
\begin{align*}
[L_m,L_n]=(m-n)L_{m+n}+\frac{m^3-m}{12}\delta_{m,n}c,
\end{align*}
where $c\in \C$ is the central charge of $L$,
which is in this case  given by
$$(n-1)(1-n(n+1)(n+k-1)^2/(n+k).$$
\subsection{Decomposition of BRST complex}
We extend the map in \S \ref{subsection:classical Miura}
to the linear map
 $\widehat{\theta}_0:\fing[t,t^{-1}]\ra \Caff(\fing)$ by setting
 \begin{align*}
  \widehat{\theta}_0(x_a(z))=x_a(z)+\sum_{\beta,\gamma\in \Delta_+}c_{a,\beta}^{\gamma}
 : \psi_{\gamma}(z)\psi_{\beta}^*(z):.
 \end{align*}

 \begin{prp}\label{Pro:decomposition-def}
  \begin{enumerate}
   \item
	 The correspondence
   \begin{align*}
    x_a(z)\mapsto J_a(z):=\widehat\theta_0(x_a(z))
    \quad (x_a\in \mf{b}_-)   \end{align*}
 defines
 a vertex algebra embedding
 $V^{\kappa_\mf{b}}(\mf{b})\hookrightarrow \Caff(\fing)$,
 where $\kappa_\mf{b}$ is the bilinear form on $\finb$ defined by
 $\kappa_\mf{b}(x,y)=\kappa(x,y)+\frac{1}{2}\kappa_\fing(x,y)$.
We have
$$[{J_a}_\lam \psi_\alpha^*]=\sum_{\beta\in \Delta_+}c_{a,\beta}^{\alpha}\psi_{\beta}^*.$$
 \item
	 The correspondence
   \begin{align*}
    x_\alpha(z)\mapsto J_\alpha(z):=\widehat\theta_0(x_\alpha)
    \quad (x_\alpha\in \mf{n})   \end{align*}
 defines
 a vertex algebra embedding
 $V(\mf{n})\hookrightarrow \Caff(\fing)$.
We have
$$[{J_\alpha}_\lam \psi_\beta]=\sum_{\beta\in \Delta_+}c_{\alpha,\beta}^{\gamma}\psi_{\gamma}^*.$$
  \end{enumerate}
 \end{prp}

Let
$\Caff(\fing)_+$ denote the subalgebra of
$\Caff(\fing)$ generated by
$J_{\alpha}(z)$ and $\psi_{\alpha}(z)$ with $\alpha\in \Delta_+$,
and let $\Caff(\fing)_-$ denote the subalgebra generated by
$J_{a}(z)$ and $\psi_{\alpha}^*(z)$ with $a\in \Delta_-\sqcup I$, $\alpha\in \Delta_+$.

The proof of the following assertions are parallel to that of 
Lemma \ref{lem:linear-dec-finite},
Lemma  \ref{lem-decom-finite} and Proposition \ref{prp:subcomp}.

\begin{lem}
 The multiplication map gives a linear isomorphism
 \begin{align*}
  \Caff(\fing)_-\*\Caff(\fing)_+\isomap \Caff(\fing).
 \end{align*}
\end{lem}
\begin{lem}\label{lem-decom}
 The subspaces $\Caff(\fing)_-$
 and $\Caff(\fing)_+$ are subcomplexes of
 $(\Caff(\fing), \hat Q_{(0)})$.
 Hence $\Caff(\fing)\cong \Caff(\fing)_-\* \Caff(\fing)_+$ as complexes.
\end{lem}
  \begin{thm}[\cite{BoeTji94,FreBen04}]
  We have $H^{i}(\Caff(\fing)_+,\hat{Q}_{(0)})=\delta_{i,0}\C$.
  Hence
   $H^{\bullet}(\Caff(\fing),\hat{Q}_{(0)})=H^{\bullet}(\Caff(\fing)_-,\hat Q_{(0)})$.
   In particular
   $\W^\kappa(\fing)=H^0({\Caff(\fing)_-},\hat Q_{(0)})$.
 \end{thm}

Since the complex $\Caff(\fing)_-$
 has no positive cohomological degree,
its zeroth cohomology
   $\W^k(\fing)=H^0(\Caff(\fing)_-,\hat Q_{(0)})$
   is a vertex {\em subalgebra}
   of $\Caff(\fing)_-$.
   Observe also that
   $\Caff(\fing)_-$    has no negative degree with respect
 to the Hamiltonian $H_{\W}$,
and each homogeneous space is finite-dimensional: 
 \begin{align}
  \Caff(\fing)_-=\bigoplus_{\Delta\in \Z_-}\Caff(\fing)_{-,\Delta,new},
  \quad \dim \Caff(\fing)_{-,\Delta,new}<\infty.
  \label{eq:grading-Cnew}
 \end{align}
 Here $\Caff(\fing)_{-,\Delta,new}=\Caff(\fing)_{-}\cap \Caff(\fing)_{\Delta,new}$.
 \subsection{Proof of Theorem \ref{thm:vanishing-W}}
 \label{Proof of Theorem {thm:vanishing-W}}
    As $\affQ F^{p}{\Caff(\fing)_-}\subset F^{p}{\Caff(\fing)_-}$,
    one can consider a spectral sequence for 
 $H^{\bullet}({\Caff(\fing)_-},\affQ)$
    such that
    the $E_1$-term is $H^{\bullet}(\gr  {\Caff(\fing)_-},\affQ)$.
    This spectral sequence clearly
    converges,
    since ${\Caff(\fing)_-}$ is a direct sum of finite-dimensional subcomplexes
    $\Caff(\fing)_{-,\Delta,new}$.

    We have
    $\gr {\Caff(\fing)_-}\cong S(\finb_-[t^{-1}]t^{-1})\* \Lam
    (\finn[t^{-1}]t^{-1})
\cong \C[J\mu^{-1}(\chi)]\* \Lam(\finn[t^{-1}]t^{-1})$,
and the complex
$(\gr  {\Caff(\fing)_-},\affQ)$ is identical to the Chevalley  complex
    for the Lie algebra cohomology $H^{\bullet}(\finn[t],
    \C[J\mu^{-1}(\chi)])$.
Therefore 
\begin{align}
 H^{i}(\gr  {\Caff(\fing)_-},\affQ)
 \cong \delta_{i,0}\C[J\mc{S}].
\label{eq:vanishing-intermidiate}
\end{align}
    Thus
     the spectral sequence collapses at $E_1=E_{\infty}$,
and we get 
\begin{align*}
 \gr^G H^{i}(  {\Caff(\fing)_-},\affQ)\cong  H^{i}(\gr  {\Caff(\fing)_-},\affQ)
 \cong \delta_{i,0}\C[J\mc{S}].
\end{align*}
    Here $\gr^G H^{i}( {\Caff(\fing)_-},\affQ)$ is the associated graded space with respect to the filtration
    $G^{\bullet }H^{i}( {\Caff(\fing)_-},\affQ)
    $ induced by  the filtration
    $F^\bullet {\Caff(\fing)_-}$, that is,
\begin{align*}
G^pH^{i}( {\Caff(\fing)_-},\affQ)=\im(H^i(F^p {\Caff(\fing)_-},\affQ)\ra H^i({\Caff(\fing)_-},\affQ)).
\end{align*}
    We claim that the filtration $G^{\bullet}H^{0}( {\Caff(\fing)_-},\affQ)$ coincides with the canonical filtration
    of $H^{0}( {\Caff(\fing)_-},\affQ)=\W^\kappa(\fing)$.
    Indeed, from the definition of
    the canonical filtration we have
    $F^pW^k(\fing)\subset G^p \W^\kappa(\fing)$ for all $p$,
    and hence,
    there is a Poisson vertex algebra homomorphism
    \begin{align}
     \gr \W^\kappa(\fing,f)\ra \gr^G \W^\kappa(\fing,f)\cong \C[J\mc{S}]
     \label{eq:From-F-to-G}
    \end{align}
    that restricts to a surjective homomorphism
   \begin{align*}
    \W^\kappa(\fing)/F^1 \W^\kappa(\fing)\twoheadrightarrow \W^\kappa(\fing)/G^1\W^\kappa(\fing)\cong \C[\mc{S}].
   \end{align*}
    Since $\C[J\mc{S}]$  is generated by $\C[\mc{S}]$ as differential algebras
    it follows that
    \eqref{eq:From-F-to-G} is surjective.
    On the other hand
    the cohomology vanishing and the Euler-Poincar\'{e} principle imply that
    the graded character of $\W^\kappa(\fing)$ and $\C[J\mc{S}]$ are the same.
    Therefore  \eqref{eq:From-F-to-G} is an isomorphism, and thus,
    $G^p \W^\kappa(\fing)=F^p \W^\kappa(\fing)$ for all $p$.
    
    Finally
    the embedding
    $\gr {\Caff(\fing)_-}\ra \gr \Caff(\fing)$
    induces an isomorphism
    \begin{align*}
     H^{0}(\gr  {\Caff(\fing)_-},\affQ)\cong H^0(\gr \Caff(\fing),\affQ)
    \end{align*}
    by  Theorem \ref{them:vanishing-associated-graded-BRST}
    and   \eqref{eq:vanishing-intermidiate}.
    This completes the proof.
\qed

\subsection{Zhu's algebra of $W$-algebra}
Let 
$\Zhu_{new} (\Caff(\mf{g}))$
be Zhu's algebra of
$\Caff(\mf{g})$ 
with respect to the Hamiltonian $H_{W}$,
$\Zhu_{old} (\Caff(\mf{g}))$
Zhu's algebra of
$\Caff(\mf{g})$ 
with respect to the standard Hamiltonian
$H$.
We have
\begin{align*}
\Zhu_{new} (\Caff(\mf{g}))\cong \Zhu_{old} (\Caff(\mf{g}))\cong C(\fing),
\end{align*}
see \cite[Proposition 5.1]{A2012Dec} for the details.
Then it is legitimate to write $\Zhu(\Caff(\mf{g}))$ for $\Zhu_{new} (\Caff(\mf{g}))$ or $\Zhu_{old} (\Caff(\mf{g}))$.

By the commutation formula, we have 
\begin{align*}
\hat{Q}_{(0)} (\Caff(\mf{g})
 \circ \Caff(\mf{g}) ) \subset \Caff(\mf{g}) \circ \Caff(\mf{g}).
\end{align*}
 Here the circle $\circ$ is defined as in the definition of the Zhu
 algebra (with respect to the grading $H_\W$). So $(\Zhu_{new} \Caff(\mf{g}), \hat{Q}_{(0)})$ is a differential,
 graded algebra,
 which is identical to $(C(\fing),\ad Q)$.

\begin{thm}[\cite{Ara07}]
We have
\begin{align*}
\Zhu \W^\kappa(\mf{g}) \cong H^0(\Zhu_{new} \Caff(\mf{g}), \hat{Q}_{(0)}) \cong  \mc{Z}(\mf{g}).
\end{align*}
\end{thm}

\begin{proof}
By 
Theorem \ref{thm:vanishing-W},
it follows that
$\W^\kappa(\fing)$ admits a PBW basis.
Hence
$\eta_{\W^\kappa(\fing)}: \gr \Zhu\W^\kappa(\fing)\ra R_{\W_k(\fing)}$ is an
 isomorphism by Theorem \ref{t:10}.
On the other hand
we have a natural algebra homomorphism $\Zhu \W^\kappa(\mf{g})
 \longrightarrow H^0(\Zhu \Caff(\mf{g}), \hat{Q}_{(0)})$ which makes the
 following diagram commute.\begin{align*}
			      \begin{CD}
\gr \Zhu \W^\kappa(\mf{g}) @>\eta_{\W^\kappa(\fing,f)}>\cong  > R_{\W^\kappa(\mf{g})}\\
@VV V @V\cong V \text{Theorem \ref{thm:vanishing-W}}  V\\
\gr \mc{Z}(\mf{g})@> \cong >>\C[\mc{S}].
 \end{CD}
			    \end{align*}
%
Note that we have the isomorphisms $H^0(R_{\Caff(\mf{g})}, \hat{Q}_{(0)})\cong H^0(\bar{C}^k(\mf{g}), \ad \bar{Q}_{(0)}) \cong \C[\mc{S}]$ and $\gr H^0(\Zhu_{new} \Caff(\mf{g}), \hat{Q}_{(0)}) \cong \gr \mc{Z}(\mf{g})$ in the diagram. Now the other three isomorphisms will give the desired isomorphism.
\end{proof}

We conclude that we have the following commutative diagram: 
\begin{align*}
 \xymatrix{
\C[J\mc{S}]\ar[d]_{\on{Zhu}(?)} & \W^\kappa(\fing)\ar[d]^{\on{Zhu}(?)}
 \ar[dl]_{R_{?}} \ar[l]_{\on{gr}(?)} \\
 \C[\mc{S}]
 &  
\mc{Z}(\fing).
 \ar[l]^{\gr(?)}
}
\end{align*}
\begin{Rem}
The same  proof applies for an arbitrary simple Lie algebra $\fing$.
In particular, we have
$\Zhu(\W^k(\mf{sl}_n))\cong \mc{Z}(\mf{sl}_n)$.
In fact the same  proof applies for the $W$-algebra associated with a simple Lie algebra $\fing$ and an arbitrary 
nilpotent element $f$ of $\fing$ to show its Zhu's algebra is isomorphic to the finite $W$-algebra $U(\fing,f)$ (\cite{De-Kac06}).
\end{Rem}
\subsection{Explicit generators}
It is possible to write down the explicit generators of $\W^\kappa(\fing)\subset \Caff (\fing)_-$.

Recall that the
{\em column-determinant} of a matrix $A=(a_{ij})$ over
an associative algebra is defined by
\begin{align*}
\on{cdet} A=\sum_{\sigma\in\mf{S}_n}\on{sgn}\sigma\cdot a^{}_{\sigma(1)\tss 1}
a^{}_{\sigma(2)\tss 2}\dots\ts a^{}_{\sigma (n)\ts n}.
\end{align*}

Introduce an extended Lie algebra $\finb[t^{-1}]t^{-1}\oplus\C\tau$, where the element $\tau$ commutes
with $\mathbf{1}$, and
\begin{align*}
\big[\tau, x_{(-n)}\big]=n\tss x_{(-n)}\qquad\text{for}\quad
x\in \finb,
n\in \finn,
\end{align*}
where $x_{(-n)}=x\tss t^{-n}$. This induces
an associative algebra structure on
the tensor product
space $U\big(\finb[t^{-1}]t^{-1}\big)\+ \C[\tau]$.

Consider the matrix
\begin{align*}
B=\begin{bmatrix}
\alpha\tss\tau+(e_{11})_{(-1)} &-1\phantom{-}&0&\dots & 0\\[0.4em]
(e_{2\tss 1})_{(-1)} &\alpha\tss\tau+(e_{2\tss 2})_{(-1)} &-1\phantom{-}&\dots & 0\\[0.4em]
\vdots &\vdots &\ddots & &\vdots\\[0.4em]
(e_{n-1\tss 1})_{(-1)} &(e_{n-1\tss 2})_{(-1)} &\dots
&\alpha\tss\tau+(e_{n-1\ts n-1})_{(-1)} &-1\phantom{-}\\[0.4em]
(e_{n\tss 1})_{(-1)} &(e_{n\tss 2})_{(-1)} &\dots &\dots  &\alpha\tss\tau+(e_{n\tss n})_{(-1)}.
   \end{bmatrix}
\end{align*}
with entries
in $U(\finb[t^{-1}]t^{-1}])\* \C[\tau]\* \C[\alpha]$,
where $\alpha$ is a parameter.

For its column-determinant\footnote{It is easy to verify that
$\on{cdet}B$ coincides with the {\it row-determinant\/}
of $B$ defined in a similar way.}
we can write
\begin{align*}
\on{cdet}B=\tau^{\tss n}+W^{(1)}_{\alpha}\tau^{\tss n-1}+\dots+
W^{(n)}_{\alpha}
\end{align*}
for certain coefficients
$W^{(r)}_{\alpha}$ which are elements of 
$U(\finb[t^{-1}]t^{-1}])\* \C[\alpha]$.
Set
\begin{align*}
W^{(i)}=W^{(i)}_{\alpha}|_{\alpha=k+n-1}.
\end{align*}
This is an element of $U(\finb[t^{-1}]t^{-1})$,
which we identify with $V^{\kappa_\mf{b}}(\finb)\subset \Caff (\fing)_-$.
\begin{Th}[\cite{AM}]\label{Th:AM}
$\W^k(\fing)$ is strongly generated by
$W^{(1)},\dots, W^{(n)}$.
\end{Th}

\subsection{Miura Map}\label{subsection:Miura}
   The Cartan subalgebra $\mf{h}$ of $\fing$
   acts on ${\Caff(\fing)_+}$ by $x_i\mapsto
(J_i)_{(0)}$, $i\in I$,
see Proposition \ref{Pro:decomposition-def}.
Let $\Caff(\fing)_+^{\lam}$ be the weight space of weight $\lam\in \mf{h}^*$ with respect to this action.
Then
   \begin{align*}
    {\Caff(\fing)_+}=\bigoplus_{\lam\leq 0}{\Caff(\fing)_+^{\lam}},
\quad {\Caff(\fing)_+^0}=V^{\kappa_\mf{b}}(\mf{h})\subset V^{\kappa_k}(\finb).
   \end{align*}
     The vertex algebra
     $V^{\kappa_\finh}(\mf{h})$ is the 
   {\em Heisenberg vertex algebra}
   associated with $\finh$ and the bilinear form $\kappa_{\finh}:=\kappa_\mf{b}|_{\finh\times \finh}$.

The projection
   ${\Caff(\fing)_+}\ra {\Caff(\fing)_+^0}=V^{\kappa_\finh}(\mf{h})$ with respect to this decomposition is a vertex algebra homomorphism.
 Therefore it restriction
   \begin{align}
    \hat \Mi
    :   \W^\kappa(\fing)\ra V^{\kappa_\finh}(\mf{h})
    \label{eq:miura}
   \end{align}
   is also a vertex algebra homomorphism that is called the {\em Miura map}.
 \begin{thm}\label{thm:Miura}
  The Miura map is injective for all $k\in \C$.
 \end{thm}
 \begin{proof}
 The induced  Poisson vertex algebra homomorphism
   \begin{align}
\gr    \hat \Mi
    : \gr  \W^\kappa(\fing)=\C[J\mathcal{S}]\ra \gr V^{\kappa_\finh}(\mf{h})=\C[J\finh^*]\cong \C[J(f+\finh)]
    \label{eq:miura}
   \end{align}
   is just a restriction map
   and 
coincides with $J\bar \Mi$,
where $\bar \Mi$ is defined in  \eqref{eq:gr-of-classical-Miura}.
Clearly,   it is sufficient to show that $J\bar \Mi$ is injective.

Recall that the action map gives an isomorphism
\begin{align*}
N\times (f+\finh_{\on{reg}})\isomap U\subset f+\finb,
\end{align*}
where $U$ is some open subset of $f+\finb$,
see the proof of Proposition \ref{prp:miura-classical}.
Therefore, by Lemma \ref{Lem:dominant-jet},
the action map $JN\times J(f+\finh)\ra J(f+\finb)$ is dominant.
Thus, the induced map
$\C[J(f+\finb)]\ra \C[JN\times J(f+\finh)]$ is injective,
and so is $J\bar \Mi:C[J(f+\finb)]^{JN}\ra \C[JN\times J(f+\finh)]^{JN}=\C[J(f+\finh)]$.
\end{proof}
 \begin{Rem}
  It is straightforward to generalize 
  Theorem \ref{thm:Miura}
  for the $W$-algebra $\W^k(\fing)$ associated with a general simple Lie algebra $\fing$.
  \end{Rem}

 \begin{thm}\label{thm:image-of-Miura}
 Let $x_i=E_{i i}\in \finh\subset \fing=\mf{gl}_n$,
 and $J_i(z)$ the corresponding field of $V^{\kappa_k}(\finh)$.
The image 
$\Mi(W^{(i)}(z))$ of $W^{(i)}(z)$
by the Miura map
is described by 
\begin{align*}
 \sum_{i=0}^n \Mi(W^{(i)})(z)(\alpha\partial_z)^{n-i}=
: (\alpha\partial_z+J_1(z))(\alpha\partial_z+J_2(z))\dots (\alpha\partial_z+J_N(z)):,
\end{align*}
where $\alpha=k+n-1$,
$ W^{(0)}(z)=1$, 
$[\partial_z,J_i(z)]=\frac{d}{dz}J_i(z)$.
 \end{thm}
\begin{proof}
It is straightforward from Theorem \ref{Th:AM}.
 \end{proof}

Note that
if we choose
$\kappa$ to be $k\kappa_0$
and
set 
$\sum_{i=1}^NJ_i(z)=0$,
we obtain 
the image of the generators of $\W^k(\mf{sl}_n)$ 
by  the Miura map $\hat \Mi$.
For $k+n\ne 0$.
this expression
 can be written in more symmetric manner:
 Set $   b_i(z)=\frac{1}{\sqrt{k+n}}J_i(z)$,
 so that
 $\sum_{i=1}^nb_i(z)=0$,
 and
  \begin{align*}
[(b_i)_{\lam}b_j]=\begin{cases}
(1-\frac{1}{n})\lam &\text{if }i=j,\\
-\frac{1}{n}\lam&\text{if }i\ne j.
\end{cases}
  \end{align*}
  Then we obtain the following original
  description of the $\W^k(\mf{sl}_n)$ due to Fateev and Lukyanov \cite{FatLyk88}.
  \begin{cor}\label{Co:Fateev-Lukyanov}
  Suppose that $k+n\ne 0$.
   Then 
the image of $\W^k(\mf{sl}_n)$ by the Miura map
is the vertex subalgebra  generated
  by fields
$\tilde W_2(z)\dots,\tilde W_n(z)$ defined by
\begin{align*}
 \sum_{i=0}^n\tilde W_i(z)(\alpha_0\partial_z)^{n-i}=
: (\alpha_0\partial_z+b_1(z))(\alpha_0\partial_z+b_2(z))\dots (\alpha_0\partial_z+b_n(z)):,
\end{align*}
   where $\alpha_0=\alpha_++\alpha_-$,
   $\alpha_+=\sqrt{k+n}$,
   $\alpha_-=-1/\sqrt{k+n}$,
  $\tilde W_0(z)=1$,
$  \tilde W_1(z) =0$.
\end{cor}
  \begin{cor}\label{Co:duality}
   Suppose that $k+n\ne 0$.
We have
   \begin{align*}
    \W^{k}(\mf{sl}_n)\cong \W^{{}^Lk}(\mf{sl}_n),
   \end{align*}
   where ${}^Lk$ is defined by $(k+n)({}^Lk+n)=1$.
 \end{cor}

 \begin{exm}
Let $\fing=\mf{sl}_2$,
$k\ne -2$.
Set $b(z)=\sqrt{2}b_1(z)=-\sqrt{2}b_2(z)$,
so that
 $[b_{\lam}b]=\lam$.
Then the right-hand-side of the formula in Corollary \ref{Co:duality}
becomes
\begin{align*}
& :(\alpha_0\partial_z+\frac{1}{\sqrt{2}}b(z))(\alpha_0\partial_z-\frac{1}{\sqrt{2}}b(z)):
\\&=\alpha_0^2\partial_z^2-L(z),
\end{align*}
where 
$$L(z)=\frac{1}{2}:b(z)^2:+\frac{\alpha_0 }{\sqrt{2}}\partial_z b(z).
$$
It is well-known and is straightforward to check that 
the field  generates the Virasoro algebra of central charge $1-6(k+1)^2/(k+2)$.
Thus
$\W^k(\mf{sl}_2)$,
$k\ne -2$, is isomorphic to the universal Virasoro vertex algebra of central charge $1-6(k+1)^2/(k+2)$.
 \end{exm}

In the case that $\kappa=\kappa_c:=-\frac{1}{2}\kappa_{\fing}$,
then it follows from Theorem \ref{thm:Miura} that 
 $\W^{\kappa_c}(\mf{gl}_n)$ is commutative
since $V^{(\kappa_c)_{\finb}}(\finh^*)$ is commutative.
In fact the following fact is known:
Let $Z(V^{\kappa}(\fing))=\{z\in V^{\kappa}(\fing)\mid 
[z_{(n)},a_{(n)}]=0\}$,
the center of $V^{\kappa}(\fing)$.
 \begin{Th}[\cite{FeiFre92}]\label{Th:FFcenter}
We have the isomorphism
$$Z(V^{\kappa_c}(\fing))\isomap \W^{\kappa_c}(\fing),
\quad z\mapsto [z\*1].$$
 \end{Th}
This is a chiralization of Kostant's Theorem \ref{Thm:whittaker-model}
in the sense that 
we recover  Theorem \ref{Thm:whittaker-model} from Theorem \ref{Th:FFcenter}
by considering the induced map
between Zhu's algebras of both sides.
The statement of Theorem \ref{Th:FFcenter}
holds
for any simple Lie algebra $\fing$ (\cite{FeiFre92}).

 \begin{rmk}
 For a general simple Lie algebra $\fing$,
 the image of the Miura map  for a generic $k$
 is described in terms of {\em screening operators},
 see \cite[15.4]{FreBen04}.
 Theorem \ref{thm:image-of-Miura}
 for $\fing=\mf{gl}_n$
also follows 
from this description (the proof reduces to the case
$\fing=\mf{sl}_2$).
An important application of this realization is 
the {\em Feigin-Frenkel duality} which states
 \begin{align*}
\W^k(\fing)\cong \W^{{}^Lk}({}^L\fing),
\end{align*}
where ${}^L\fing$ is the Langlands dual Lie algebra of $\fing$,
$r^{\vee}(k+h^{\vee})({}^Lk+{}^Lh^{\vee})=1$.
Here $r^{\vee}$ is the maximal number of the edges of the Dynking diagram of $\fing$
and ${}^Lh^{\vee}$  is the dual Coxeter number of ${}^L\fing$.
In \cite{FeiFre92,FreBen04} this isomorphism was stated only  for a generic  $k$,
but it is not too difficult to see the isomorphism remains valid for an arbitrary $k$
using the injectivity of the Miura map.

 The Miura map is  defined   \cite{KacRoaWak03}
 for the $W$-algebra $\W^k(\fing,f)$  associated with an arbitrary $f$,
 which is injective as well since the proof of Theorem \ref{thm:Miura}
 applies.
 Recently Naoki Genra \cite{Genra}
 has obtained 
 the description of the image 
by the Miura map
 in terms of screening operators
for the $W$-algebra $\W^k(\fing,f)$  associated with an arbitrary nilpotent element $f$.
  \end{rmk}
 
 \subsection{Classical $W$-algebras}
 Since the Poisson structure of 
$\C[\mc{S}]$ is trivial,
we can give $\gr \W^\kappa(\fing)$ a Poisson vertex algebra structure by the formula
\eqref{eq:further}.
The Poisson structure of $R_{V^{k+n}(\finh)}=\C[\finh]$ is also trivial,
hence 
$\gr V^{\kappa_\finh}(\finh)=\C[J\finh^*]$
is equipped with the Poisson vertex algebra structure by the formula
\eqref{eq:further} as well.
Then the map $\gr \hat \Mi:\gr \W^\kappa(\fing)\hookrightarrow \gr V^{\kappa_{\finh}}(\finh)$
is a homomorphism of Poisson vertex algebras with respect to these structures.
Set
 $\kappa=k\kappa_0$, $k\in \C$,
 and
 consider its restriction 
$\gr \hat \Mi:\gr \W^k(\mf{sl}_n)\hookrightarrow \gr V^{\kappa_{\finh}}(\finh')$,
where $\finh'$ is the Cartan subalgebra of $\mf{sl}_n$.

In $\gr V^{\kappa_\finh}(\finh')$ 
we have
 \begin{align*}
\{h_{\lam}h'\}=\kappa_\finh(h,h')=(k+n)\kappa_0(h,h'),
\end{align*}
and this uniquely determines
the $\lam$-bracket of $\gr V^{\kappa_\finh}(\finh')$.
Hence it is independent of $k$ provided that $k\ne -n$.
Since the image of $\gr \W^k(\mf{sl}_n)$
is
strongly generated by
elements of $\C[(\finh')^*]^W$,
it follows that
the Poisson vertex algebra structure of 
$\gr \W^k(\mf{sl}_n)$, $k\ne-n$,
is independent of $k$.
We denote this Poisson vertex algebra 
by
$\W^{cl}(\mf{sl}_n)$.

The Poisson vertex algebra $\W^{cl}(\mf{sl}_n)$ is called the {\em classical $W$-algebra} associated with $\mf{sl}_n$,
which appeared in the works of Adler \cite{Adl78},
Gelfand-Dickey \cite{GelDik78}
and Drinfeld-Sokolov \cite{DriSok84}.
Thus,
the $W$-algebra $\W^k(\mf{sl}_n)$, $k\ne -n$,
is a deformation of $\W^{cl}(\mf{sl}_n)$.

On the other hand the $\W$-algebra $\W^{-n}(\mf{sl}_n)$ at the critical level 
can be identified with the space of the {\em $\mf{sl}_n$-opers} \cite{BeiDri05} on the disk $D$.
We refer to \cite{FreBen04,Fre07} for more on this subject.

\section{Representations of $W$-algebras}
\label{sec:rep-theory}
From now on we set $\fing=\mf{sl}_n$
and study the representations of $\W^k(\fing)$ (see \eqref{eq:sln-W}).
\subsection{Poisson modules}
Let $R$ be a Poisson algebra.
Recall that a {\em Poisson
$R$-module} is
a $R$-module  $M$ in the usual associative sense equipped with
a bilinear map
\begin{align*}
R\times M\ra M,\quad (r,m)\mapsto \ad r(m)=\{r,m\},
\end{align*}
which makes $M$ a Lie algebra module over $R$
satisfying 
\begin{align*}
 \{r_1,r_2 m\}=\{r_1,r_2\}m+r_2\{r_1,m\},\quad
\{r_1 r_2,m\}=r_1\{r_2,m\}+r_2\{r_1,m\}
\end{align*}
for $r_1,r_2\in R$, $m\in M$.
Let $R\on{-PMod}$ be the category of 
 Poisson modules
over $R$.

 \begin{lem}\label{lem:C[g]-modules}
A Poisson module over $\C[\mf{g}^*]$ is the same as 
a $\C[\mf{g}^*]$-module $M$ in the usual associative sense equipped with
a Lie algebra module 
structure $\mf{g}\ra \End M$,
$x\mapsto \ad(x)$,
 such that
\begin{align*}
\ad(x)(fm)=\{x,f\}.m+ f.\ad(x)(m)
\end{align*}
for $x\in \mf{g}$,
$f\in \C[\mf{g}^*]$,
$m\in M$.
 \end{lem}

\subsection{Poisson vertex modules}
 \begin{Def}
A {\em Poisson vertex module} over 
a Poisson vertex algebra $V$ is a 
$V$-module $M$ as a vertex algebra equipped with 
 a linear map
\begin{align*}
 V\mapsto (\End M)[[z\inv]]z\inv,\quad a\mapsto Y_-^M(a,z)=\sum_{n\geq 0}a^M_{(n)}z^{-n-1},
\end{align*}
satisfying
\begin{align}
& a_{(n)}^Mm=0\quad\text{for }n\gg 0,   
\\
&(T a)_{(n)}^M=-na^M_{(n-1)},\\
 &a_{(n)}^M(b v)=(a_{(n)}b)  v+b (a_{(n)}^Mv),\\
&[a^M_{(m)},b^M_{(n)}]=\sum_{i\geq 0}\begin{pmatrix}
				  m\\ i
				 \end{pmatrix}(a_{(i)}b)^M_{(m+n-i)},
\label{eq:commutatoer-pv}
\\
&(ab)^M_{(n)}=\sum_{i=0}^{\infty}(a_{(-i-1)}^Mb_{(n+i)}^M+b_{(-i-1)}^Ma_{(n+i)}^M)
\label{eq:associativity-VPA}
\end{align}
for all $a, b\in V$, $m, n\geq 0$, $v\in M$.

\end{Def}

A  Poisson vertex algebra $R$ is naturally a Poisson vertex module over
itself.

 \begin{exm}
Let  $M$ be a Poisson vertex module over 
 $\C[J\fing^*]$.
Then
by \eqref{eq:commutatoer-pv},
the assignment 
\begin{align*}
xt^n\mapsto x^M_{(n)}\quad x\in \fing\subset \C[\fing^*]\subset \C[J\fing^*],\
n\geq 0,
\end{align*}
defines a $J\fing=\fing[[t]]$-module structure on $M$.
In fact, a Poisson vertex module over $\C[J\fing^*]$
is the same as a 
$\C[J\fing^*]$-module $M$
in the usual associative sense
equipped with an action of the Lie algebra $J\fing$
such that
$(xt^n) m=0$ for $n\gg 0$,
$x\in \fing$,
$m\in M$,
and 
\begin{align*}
(xt^n)\cdot (a m)=(x_{(n)}a) m+a (x t^n)\cdot m
\end{align*}
for $x\in \fing$, $n\geq 0$,
$a\in \C[J\fing^*]$,
$m\in M$.

 \end{exm}

Below we often write $a_{(n)}$ for $a^{M}_{(n)}$.

The proofs of
the following assertions are straightforward.
 \begin{Lem}\label{Lem:VPA-induced}
Let $R$ be a Poisson algebra,
$E$ a Poisson module over $R$.
There is a unique
 Poisson vertex $J R$-module structure 
on $
J R\*_R E
$
such that
\begin{align*}
 a_{(n)}(b\* m)=(a_{(n)}b)\*
 m+\delta_{n,0}b\* \{a,m\}
\end{align*}
for
$n\geq 0$, 
$a\in R\subset J R$, $b\in J R$.
$m\in E$
(Recall that $JR=\C[J\Spec R]$.)
 \end{Lem}

 \begin{Lem}\label{Lem:universality-of-induced-VPA-module}
Let $R$ be a Poisson algebra,
$M$  a Poisson vertex module over $JR$.
Suppose that
there exists a
 $R$-submodule $E$ of $M$
(in the usual commutative sense)
such that
$a_{(n)}E=0$  for $n>0$, $a\in R$,
and $M$ is generated by $E$
(in the usual commutative sense).
Then there exists a surjective homomorphism
\begin{align*}
 JR\*_{R}E\twoheadrightarrow M
\end{align*}
of Poisson vertex modules.
 \end{Lem}

\subsection{Canonical filtration of modules over vertex algebras}
Let $V$ be a vertex algebra graded by a Hamiltonian $H$.
A {\em compatible filtration} 
of a $V$-module $M$ is a decreasing filtration
\begin{align*}
M=\Gamma^0M\supset \Gamma^1M\supset \cdots
\end{align*}
such that
\begin{align*}
&
 a_{(n)}\Gamma^q M\subset \Gamma^{p+q-n-1}M\quad \text{for }
a\in F^p V,\ \forall n\in \Z,
\nonumber \\
&
 a_{(n)}\Gamma^qM\subset  \Gamma^{p+q-n}M\quad \text{for }
 a\in
F^p V,
\ n\geq 0,
\\
& H. \Gamma^p M\subset \Gamma^p M\quad \text{for all }p\geq 0,\\
&\bigcap_p \Gamma^p M=0.
\end{align*}
For a compatible filtration
$\Gamma^{\bullet}M$
the associated graded space
\begin{align*}
\gr^\Gamma M=\bigoplus_{p\geq 0} \Gamma^p M/\Gamma^{p+1}M
\end{align*}
is
naturally a graded vertex Poison module over the graded vertex Poisson algebra
$\gr^F V$,
and hence,
it is a 
graded vertex Poison module over $ JR_V=\C[\tilde{X}_V]$
by Theorem \ref{thm:surj-poisson}.

The vertex Poisson
$JR_V$-module
structure 
of $\gr^\Gamma M$
restricts to 
the Poisson $R_V$-module structure
of $M/\Gamma^1 M=\Gamma^0M /\Gamma^1M$,
and
$a_{(n)}(M/\Gamma^1 M)=0$
 for $a\in R_V\subset JR_V$,
$n>0$.
It follows that
there is a  homomorphism
\begin{align*}
 JR_V\*_{R_V}(M/\Gamma^1 M)\ra \gr^{\Gamma}M,
\quad a\* \bar m\mapsto a\bar m,
\end{align*}
of vertex Poisson modules
by Lemma \ref{Lem:universality-of-induced-VPA-module}.

Suppose that
$V$ is positively graded and so
is a $V$-module $M$.
We denote by $F^\bullet M$  
the Li filtration \cite{Li05}  
of $M$,
which 
is 
 defined 
 by
 \begin{align*}
F^pM=\haru_{\C}\{a^1_{(-n_1-1)}\dots a^r_{(-n_r-1)}m\mid a^i\in V,\ m\in M,\ n_1+\dots +n_r\geq p  \}.
\end{align*}
It is a compatible filtration of $M$,
and in fact is
 the finest compatible filtration of $M$,
that is,
$F^p M\subset \Gamma^p M$ for all $p$
for any compatible filtration 
$\Gamma^\bullet M$ of $M$.
The subspace $F^1 M$ is spanned by the
vectors $a_{(-2)}m$ with $a\in V$, $m\in M$,
which is often denoted by
$C_2(M)$ in the literature.
Set
\begin{align}
 \bar M=M/F^1 M(=M/C_2(M)),
\end{align}
which 
 is a 
Poisson module over $R_V=\bar V$.
By \cite[Proposition 4.12]{Li05},
the vertex Poisson module homomorphism
\begin{align*}
 JR_V\*_{R_V}\bar M\ra \gr^F M
\end{align*}
is surjective.

Let $\{ a^i; i\in I\}$ be elements 
of $V$ such that their images generate  $R_V$
in usual commutative sense,
and let $U$ be a subspace of
$M$ such that $M=U+F^1 M$.
The surjectivity of the above map is equivalent to that
\begin{align}
 &F^p M\label{eq:span-FpM}
\\
=\haru_{\C}&\{a_{(-n_1-1)}^{i_1}
\dots a_{(-n_r-1)}^{i_r}m\mid m\in U,\ n_i\geq 0,
n_1+\dots +n_r\geq p,\
 i_1,\dots,i_r\in I\}.
\nonumber
\end{align}

 \begin{lem}
Let $V$ be a vertex algebra, $M$ a $V$-module.
The Poisson vertex algebra module structure of $\gr^F M$
restricts to the Poisson module structure
of $\bar M:=M/F^1 M$ over $R_V$, that is,
$\bar M$ is a Poisson $R_V$-module by
\begin{align*}
\bar  a\cdot \bar m=\overline{a_{(-1)}m},\quad \ad(\bar a)(\bar m)=\overline{a_{(0)}m}.
\end{align*}
 \end{lem}

A $V$-module $M$ is called {\em finitely strongly generated}
if $\bar M$ is finitely generated as a $R_V$-module in the usual
associative sense.

\subsection{Associated varieties of modules over affine vertex algebras}
 A $\affg$-module $M$ of level $k$ is called {\rm smooth} if 
$x(z)$ is a field on $M$ for $x\in \fing$,
that is, $xt^nm=0$ for $n\gg 0$, $x\in \fing$, $m\in M$.
 Any $V^{k}(\fing)$-module $M$ is naturally a {smooth} $\affg$-module of level $k$.
Conversely, any smooth $\affg$-module of level $k$ can be regarded as  a $V^k(\fing)$-module.
It follows that a $V^k(\fing)$-module is the same as a smooth $\affg$-module of level $k$.

For a  $V=V^k(\fing)$-module $M$,
or equivalently, a smooth $\affg$-module of level $k$,
we have
\begin{align*}
 \bar M=M/\fing[t^{-1}]t^{-2}M,
\end{align*}
and the Poisson $\C[\fing^*]$-module structure is given by
\begin{align*}
x\cdot \bar m=\overline{xt^{-1}m},\quad \ad(x)\bar m=\overline{xm}.
\end{align*}
For a $\fing$-module $E$ let
\begin{align*}
 V_{E}^k:=U(\affg)\*_{U(\fing[t]\+ \C K)}E,
\end{align*}
where $E$ is considered as a $\fing[t]\+ \C K$-module 
on which $\fing[t]$ acts  via the projection
$\fing[t]\ra \fing$ and $K$ acts as multiplication by $k$.
Then
\begin{align*}
 \overline{V_{E}^k}\cong \C[\fing]\* E,
\end{align*}
where the Poisson $\C[\fing^*]$-module structure is given by
\begin{align*}
f\cdot g\* v=(fg)\* v,\quad 
\ad x (f\*v)=\{x,f\}\*v+f\* xv,
\end{align*}
for $f,g\in \C[\fing^*]$, $v\in V$.

Let $\mc{O}_k$ be the category $\mc{O}$ of $\affg$ of level $k$ (\cite{Kac74}),
$\KL_k$ the full subcateogory of $\mc{O}_k$ consisting of modules $M$ which are
integrable over $\fing$.
Note that
$V_{E}^k$ is a object of $\KL_k$ for a finite-dimensional representation $E$ of $\fing$.
Thus, $V^k(\fing)=V_{\C}^k$ and its simple quotient $V_k(\fing)$ are also objects of $\KL_k$.

Both $\mc{O}_k$  and $\KL_k$ can be regarded as full subcategories of the category of $V^k(\fing)$-modules.
 \begin{lem}
For $M\in \KL_k$ the following conditions are equivalent.
\begin{enumerate}
 \item  $M$ is finitely strongly generated as  a $V^k(\fing)$-module,
\item $M$ is finitely generated as a $\fing[t^{-1}]t^{-1}$-module,
\item $M$ is finitely generated as a $\affg$-module.
\end{enumerate}

 \end{lem}

For a finitely strongly generated $V^k(\fing)$-module $M$ define its {\em associated variety}
$X_M$ by
\begin{align*}
 X_M=\on{supp}_{R_V}(\bar M)\subset X_V,
\end{align*}
equipped with a reduced scheme structure.

 \begin{exm}
$X_{V_{E}^k}=\fing^*$ for a  finite-dimensional representation $E$ of $\fing$.
 \end{exm}
\subsection{Ginzburg's correspondence}
Let $\overline{\HC}$ be the full subcategory of the category of Poisson
$\C[\fing^*]$-modules on which the Lie algebra $\fing$-action (see Lemma
\ref{lem:C[g]-modules}) is
integrable.
 \begin{lem}\label{lem:barHC}
For $M\in \KL_k$,
the Poisson $\C[\fing^*]$-module $\bar M$ belongs to $\overline{\HC}$.
 \end{lem}
By Lemma \ref{lem:barHC}
we have a right exact functor
\begin{align*}
 \KL_k\ra \overline{\HC},\quad M\mapsto \bar M.
\end{align*}

For $M\in \overline{\HC}$,
$M\* \overline{Cl}$
 is naturally a Poisson module over $\bar
C(\fing)=\C[\fing^*]\* \overline{Cl}$.
(The notation of Poisson modules natural extends to the Poisson supralgebras.)
Thus,
$(M\* \overline{Cl}, \ad \bar Q)$
is a differential graded Poisson module over 
the differential graded Poisson module
$(\bar C(\fing),\ad \bar Q)$.
In particular
it cohomology
$H^{\bullet}(\bar M\* \overline{Cl},\ad \bar Q)$ is a Poisson module
over
$H^{\bullet}(\bar C(\fing),\ad \bar Q)=\C[\mc{S}]$.
So we get a functor
\begin{align*}
 \overline{\HC}\ra \C[\mc{S}]\on{-Mod},\quad M\mapsto
 H^0(M):=H^0(M\*\overline{Cl},\ ad \bar Q).
\end{align*}

The following assertion is a restatement of a result of Ginzburg \cite{Gin08}
(see \cite[Theorem 2.3]{A2012Dec}).
 \begin{thm}\label{Th:classocial-Harish-Chandra1}
Let $M\in \overline{\HC}$.
Then
 $H^i(M)=0$ for $i\ne 0$,
and we have an isomorphism
\begin{align*}
 H^0(M)\cong (M/\sum_{i}\C[\fing^*](x_i-\chi(x_i))M
 )^{N}.
\end{align*}
  In particular
  if $M$ is finitely generated $H^0(M)$ is finitely generated over $\C[\mc{S}]$ and 
  \begin{align*}
   \on{supp}_{\C[\mc{S}]}H^0(M)=(\on{supp}_{\C[\fing^*]}M)\cap \mc{S}.
  \end{align*}
 \end{thm}

 \begin{cor}
  The functor $\overline{\HC}\ra \C[\mc{S}]\on{-Mod}$,
  $M\mapsto H^0(M)$, is exact.
 \end{cor}

 Denote by
 $\mc{N}$ the set of nilpotent elements of $\fing$,
 which equals to 
 the zero locus of the augmentation ideal $\C[\fing^*]^G_+$ of $\C[\fing^*]^G$
 under the identification $\fing=\fing^*$ via $(~|~)$.
 Since the element $f$ (defined in \eqref{eq:f}).
) is regular (or principal),
 the orbit 
 \begin{align*}
\mb{O}_{prin}:=G.f\subset \fing=\fing^*
 \end{align*}
 is dense in $\mc{N}$:
 \begin{align*}
  \mc{N}=\overline{\mb{O}_{prin}}.
 \end{align*}
The   transversality of $\mc{S}$ implies that
\begin{align*}
 \mc{S}\cap \mc{N}=\{f\}.
\end{align*}

 \begin{thm}[\cite{Gin08}]\label{Th:classocial-Harish-Chandra2}
  Let $M$ be a finitely generated object in $ \overline{\HC}$.
  \begin{enumerate}
   \item  $H^0(M)\ne 0$ if and only if
	  $\mc{N}\subset \on{supp}_{\C[\fing^*]}M$.
\item $H^0(M)$ is nonzero and finite-dimensional if  $\on{supp}_{\C[\fing^*]}M=\mc{N}$.
  \end{enumerate}
 \end{thm}
 \begin{proof}
(1)  Note that $\on{supp}_{\C[\mc{S}]}H^0(M)$
  is invariant under the $\C^*$-action
  \eqref{eq:C-star-action} on $\mc{S}$,
  which contracts the point $\{f\}$,
 Hence
  $   \on{supp}_{\C[\mc{S}]}H^0(M)=(\on{supp}_{\C[\fing^*]}M)\cap \mc{S}$ is nonempty if and only if
  $f\in \on{supp}_{\C[\mc{S}]}H^0(M)$.
  The assertion follows since
  $\on{supp}_{\C[\mc{S}]}H^0(M)$ is $G$-invariant and closed.
  (2) Obvious since the assumption implies that 
  $\on{supp}_{\C[\mc{S}]}H^0(M)=\{f\}$.
 \end{proof}

  \subsection{Losev's correspondence}
  Let $\HC$ be the 
  category of {\em Harish-Chandra bimodules}, that is,
  the
  full subcategory of the category of $U(\fing)$-bimodules on which the adjoint action of $\fing$ is integrable.

  \begin{lem}
   Every finitely generated object $M$ of $\HC$   admits a good filtration,
   that is, an increasing filtration
   $0=F_0M\subset F_1M\subset \dots$
   such that
   $M=\bigcup F_p M$,
\begin{align*}
U_p(\fing)\cdot F_qM\cdot U_r(\fing)\subset F_{p+q+r}M,\quad[U_p(\fing),F_pM]\subset F_{p+q-1}M,
\end{align*}
   and $\gr^F M=\bigoplus_p F_pM/F_{p-1}M$ is finitely generated over $\C[\fing^*]$.
  \end{lem}

  If $M\in \HC$  and $F_{\bullet}M$ is a good filtration,
  then $\gr^F M$ is naturally a Poisson module over $\C[\fing^*]$.
  Therefore, it is an object of $\overline{\HC}$.

Let $M$ be a finitely generated object in $\HC$.  It is known since Bernstein that
  \begin{align*}
   \on{Var}(M):=\on{supp}_{\C[\fing^*]}(\gr^F M)\subset \fing^*
  \end{align*}
  in independent of the choice of a good filtration $F_{\bullet}M$ of $M$.

  For $M\in \HC$,
  $M\* Cl$ is naturally a bimodule over  $C(\fing)=U(\fing)\* Cl$.
  Thus,
  $(M\* Cl,\ad Q)$ is a differential graded bimodule over $C(\fing)$,
  and its cohomology
  \begin{align*}
   H^{\bullet}(M):=H^{\bullet}(M\* Cl,\ad Q) 
  \end{align*}
  is naturally a module over $H^0(C(\fing),\ad Q)$ that is identified with $Z(\fing)$  by
  Theorem \ref{Thm:whittaker-model}.
  Thus,
  we have a functor
  \begin{align}
   \HC\ra Z(\fing)\on{-Mod},\quad M\mapsto H^0(M).
\label{eq:functor-classical}
  \end{align}

  Let $M\in \HC$ be finitely generated,
  $F_{\bullet}M$ a good filtration.
  Then $F_p (M\* Cl):=\sum_{i+j=p}F_i M\* Cl_j$ defines a good filtration of $M\* Cl$,
  and the associated graded space
$\gr_F (M\* Cl)=\sum_{i}F_p (M\* Cl)/F_{p-1}(M\* Cl)=(\gr_F M)\*
  \overline{ Cl}$ is a Poisson module over $\gr C(\fing)=\bar C(\fing)$.

The filtration $F_{\bullet}(M\* Cl)$ induces 
a filtration $F_{\bullet}H^\bullet(M)$ on $H^\bullet(M)$,
and $\gr_F H^\bullet(M)=\bigoplus_p F_p H^{\bullet}(M)/F_{p-1}H^0(M)$ is
  a module over $\gr Z(\fing)=\C[\mc{S}]$.
  
  For a finitely generated $\mc{Z}(\fing)$-module $M$,
  set $\on{Var}(M)=\on{supp}_{\C[\mc{S}]}(\gr M)$,
  $\gr M$ is the associated graded $M$ with respect to a good filtration of $M$.
  
  The following assertion follows from Theorems \ref{Th:classocial-Harish-Chandra1}
  and \ref{Th:classocial-Harish-Chandra2}.
   \begin{thm}[\cite{Gin08,Los11}]
\begin{enumerate}
 \item    We have $H^i(M)=0$ for all $i\ne 0$, $M\in \HC$.
Therefore the functor \eqref{eq:functor-classical}
is exact.
\item 
 Let $M$ be a finitely generated object of $\HC$,
$F_{\bullet} M$ a good filtration.
Then $\gr _FH^0(M)\cong H^0(\gr_F M)$.
In particular
 $H^0(M)$ is finitely generated, $F_{\bullet}H^0(M)$ is a good filtration
      of $H^0(M)$.
      \item  For a finitely generated object $M$ of $\HC$,
     $ \on{Var}(H^0(M))= \on{Var}(M)\cap \mc{S}$.
\end{enumerate}
   \end{thm}

   \subsection{Frenkel-Zhu's bimodules}
   Recall that for a graded vertex algebra $V$,
   Zhu's algebra $\Zhu(V)=V/V\circ V$ is defined.
   There is a similar construction for modules due to Frenkel and Zhu \cite{FreZhu92}.
   For a $V$-module $M$ set 
\begin{align*}
 \Zhu(M)=M/V\circ M,
\end{align*}
where 
$V\circ M$ is the subspace of $M$ spanned by the vectors
\begin{align*}
 a\circ m=\sum_{i\geq 0}\begin{pmatrix}
			 \Delta_a\\ i
			\end{pmatrix}a_{(i-2)}m
\end{align*}
for  $a\in V_{\Delta_a}$, $\Delta_a\in \Z$, and $m\in M$.
 \begin{prp}[\cite{FreZhu92}]
$\Zhu (M)$ is a bimodule over $\Zhu (V)$ by the multiplications
\begin{align*}
 a* m=\sum_{i\geq 0}\begin{pmatrix}
		     \Delta_a\\i
		    \end{pmatrix}a_{(i-1)}b,\quad 
m* a=\sum_{i\geq 0}\begin{pmatrix}
		     \Delta_a-1\\i
\end{pmatrix}a_{(i-1)}m
\end{align*}
for $a\in V_{\Delta_a}$, $\Delta_a\in \Z$, and $m\in M$.
 \end{prp}
Thus, we have a right exact functor
\begin{align*}
 V\on{-Mod}\ra \Zhu(V)\on{-biMod},\quad M\mapsto \Zhu(M).
\end{align*}

 \begin{lem}
Let $M=\bigoplus_{d\in h+\Z_+}M_d$ be a positive energy representation
  of a  $\Z_+$-graded vertex algebra $V$.
Define an increasing filtration $\{\Zhu_p(M)\}$ on $\Zhu(V)$ by 
\begin{align*}
 \Zhu_p(M)=\im (\bigoplus_{d=h}^{h+p}M_p\ra \Zhu(M)).
\end{align*}
\begin{enumerate}
 \item 
We have
\begin{align*}
&\Zhu_p (V)\cdot \Zhu_{q}(M)\cdot \Zhu_r(V)\subset \Zhu_{p+q+r}(M),\\
&
[\Zhu_p(V),\Zhu_q(M)]\subset \Zhu_{p+q-1}(M).
\end{align*}
Therefore $\gr \Zhu(M)=\bigoplus_{p}\Zhu_p(M)/\Zhu_{p-1}(M)$
is a Poisson $\gr \Zhu(V)$-module,
and hence is a Poisson $R_V$-module through the homomorphism
$\eta_V: R_V\twoheadrightarrow \gr \Zhu(V)$.
\item
There is a natural surjective homomorphism
\begin{align*}
 \eta_M:\bar M(=M/F^1 M)\ra \gr \Zhu(M)
\end{align*}
of Poisson $R_V$-modules.
This is an isomorphism
if $V$ admits a PBW basis and $\gr M$ is free over $\gr V$.
\end{enumerate}
 \end{lem}

 \begin{exm}
Let $M=V_{E}^k$.
Since $\gr V_{E}^k$ is free over $\C[J\fing^*]$,
we have the isomorphism
\begin{align*}
 \eta_{V_{E}^k}: \overline{V_{E}^k}=E\* \C[\fing^*]\isomap \gr \Zhu (V_{E}^k).
\end{align*}
On the other hand, 
there is a 
$U(\fing)$-bimodule homomorphism
\begin{align}
 \begin{split}
E\* U(\fing)&\ra \Zhu (V_{E}^k),\\ v\*x_1\dots x_r&\mapsto 
(1\*v)* (x_1t^{-1})* (x_1t^{-1})+ V^k(\fing)\circ V_{E}^k
 \end{split}\label{eq;weyl}
\end{align}
which respects the filtration.
Here the $U(\fing)$-bimodule structure of $U(\fing)\* E$ is  given by
\begin{align*}
 x (v\*u)=(xv)\* u+v\* xu,\quad (v\* u)x=v\* (ux),
\end{align*}
and the filtration
of 
$U(\fing)\* E$ is given by $\{U_i(\fing)\* E\}$.
Since the induced homomorphism between associated graded  spaces
\eqref{eq;weyl} coincides with $\eta_{V_{E}^k}$,
\eqref{eq;weyl}
is an isomorphism.
 \end{exm}
 \begin{lem}
For $M\in \KL_k$ we have $\Zhu(M)\in \HC$.
If $M$ is finitely generated, then so is $\Zhu(M)$.
 \end{lem}
\subsection{Zhu's two functors commute with BRST reduction}
For a smooth $\affg$-module $M$ over level $k$,
$C(M):=M\* \mc{F}$ is naturally a module over $C^k(\fing)=V^k(\fing)\* \mc{F}$.
Thus, $(C(M),Q_{(0)})$ is a cochain complex, 
and its cohomology $H^{\bullet}(M):=H^{\bullet}(C(M),Q_{(0)})$ is a module over 
$\W^k(\fing)=H^{\bullet}(C^k(\fing),Q_{(0)})$.
Thus 
we have a functor 
\begin{align*}
V^k(\fing)\on{-Mod}\ra \W^k(\fing)\on{-Mod},\quad
M\mapsto H^0(M).
\end{align*}
Here $V\on{-Mod}$ denotes the category of modules over a vertex algebra $V$.

\begin{Th}\label{Th:vanising}
\begin{enumerate}
\item (\cite{FreGai07,Ara09b}) We have $H^i(M)=0$ for $i\ne 0$, $M\in \KL_k$.
In particular 
the functor 
\begin{align*}
\KL_k\ra \W^k(\fing)\on{-Mod},\quad M\mapsto  H^0(M),
\end{align*}
is exact.
\item  (\cite{Ara09b})
For a finitely generated object $M$ of $\KL$,
$$\overline{H^0(M)}\cong H^0(\bar M)$$ as Poisson modules over $R_{W^k(\fing)}=\C[\mc{S}]$.
In particular $H^0(M)$ is finitely strongly generated and 
 $$X_{H^0(M)}=X_{M}\cap \mc{S}.$$
 \item (\cite{A2012Dec}) For a finitely generated object $M$ of $\KL$,
$$\on{Zhu}(H^0(M))\cong H^0(\on{Zhu}(M))$$
 as bimodules over 
 $\on{Zhu}(\W^k(\fing))=\mc{Z}(\fing)$.
\end{enumerate}
\end{Th}

Let $\W_k(\fing)$ denote the unique simple graded quotient of $\W^k(\fing)$.
Then $X_{\W_k(\fing)}$ is a $\C^*$-invariant subvariety of $\mc{S}$.
Therefore $X_{\W^k(\fing)}$ is lisse if and only if $X_{\W_k(\fing)}=\{f\}$
since the $\C^*$-action on $ \mc{S}$ contracts to the point $f$.
\begin{Co}\label{Co:lisse}
\begin{enumerate}
\item $H^0(V_k(\fing))$ is a quotient of $\W^k(\fing)=H^0(V^k(\fing))$.
In particular $\W_k(\fing)$ is a quotient of $H^0(V_k(\fing))$
if $H^0(V_k(\fing))$ is nonzero.
\item $H^0(V_k(\fing))$ is nonzero if and only if $X_{V_k(\fing)}\supset \overline{G.f}=\mc{N}$.
\item The simple $W$-algebra $\W_k(\fing)$ is lisse if $X_{V_k(\fing)}={\overline{G.f}}=\mc{N}$.
\end{enumerate}
\end{Co}
\begin{proof}
(1) follows from the exactness statement of Theorem \ref{Th:vanising}.
(2) $H^0(V_k(\fing))$ is nonzero if and only of  $X_{H^0(M)}=X_{M}\cap \mc{S}$ is non-empty.
This happens if and only if $f\in X_{M}$
since $ X_{H^0(M)}$ is $\C^*$-stable.
The assertion follows since $X_{M}$ is $G$-invariant and closed.
(3) If $X_{V_k(\fing)}={\overline{G.f}}$,
$X_{H^0(V_k(\fing))}=X_{M}\cap \mc{S}=\{f\}$, and thus,
$H^0(V_k(\fing))$ is lisse,
and thus, so its quotient $\W_k(\fing)$.
\end{proof}

\begin{Rem}\label{Rem:minimal}
\begin{enumerate}
\item The above results hold for $W$-algebras associated with any $\fing$ and any $f\in \mc{N}$
without any restriction on the level $k$
(\cite{Ara09b,A2012Dec}).
In particular 
we have the vanishing result
\begin{align}
H^i_{f}(M)=0\quad \text{for }i\ne 0,\ M\in \KL_k,
\label{eq;vanishing-general}
\end{align}
for the BRST cohomology $H^i_{f}(M)$
of
 the quantized Drinfeld-Sokolov reduction functor
associated with $f$ in the coefficient in an object $M$ of $\KL_k$.
Thus the functor
\begin{align*}
\KL_k\ra \W^k(\fing,f)\on{-Mod},\quad M\mapsto H^0_{f}(M),
\end{align*}
is exact,
and  moreover, 
\begin{align*}
X_{H^0_f(V_k(\fing))}=X_{V_k(\fing)}\cap \mc{S}_f,
\end{align*}
where $S_f$ is the  Slodowy slice at $f$ (see \S \ref{subsection:Generalization to finite $W$-algebras}).
In particular
\begin{align}
H^0_f(V_k(\fing))\ne 0 \iff X_{V_k(\fing)}\supset \overline{G.f}.
\label{eq:cirterion}
\end{align}
\item 
In the case that $f=f_{\theta}$,
a 
 minimal nilpotent element of $\fing$,
 then we also have the following 
 result \cite{Ara05}:
 \begin{align*}
H_{f_{\theta}}(V_k(\fing))=\begin{cases}
\text{$\W_k(\fing,f_{\theta})$}&\text{if }k\not \in \Z_+,\\
0&\text{if }k\in \Z_+.
\end{cases}
\end{align*}
Here $\W_k(\fing,f_{\theta})$
is the simple quotient of $\W^k(\fing,f_{\theta})$.
Together with  \eqref{eq:cirterion},
this proves 
 the ``only if'' part of Theorem \ref{t:6}.
 Indeed, 
 if $V_{k}(\fing)$ is  lisse, then
  $H_{f_{\theta}}(V_k(\fing))= 0$ by  \eqref{eq:cirterion},
 and hence, $k\in \Z_+$.
\end{enumerate}
\end{Rem}

\section{Irreducible representations of $W$-algebras}
\label{sec:irrep}
In this section we quickly review results obtained in \cite{Ara07}.

Since 
$\on{Zhu}(\W^k(\fing))\cong \mc{Z}(\fing)$,
by Zhu's theorem irreducible positive energy representations of $\W^k(\fing)$ are parametrized by
central characters of $\mc{Z}(\fing)$.
For a central character $\gamma:\mc{Z}(\fing)\ra \C$,
let $\mathbb{L}(\gamma)$ be the corresponding irreducible positive energy representations of $\W^k(\fing)$.
This is a simple quotient of the {\em Verma module} $\mathbb{M}(\gamma)$ 
of $\W^k(\fing)$ with highest weight $\gamma$,
which has the character 
\begin{align*}
\ch \mathbb{M}(\gamma):=\on{tr}_{\mathbb{M}(\gamma)}q^{L_0}=\frac{q^{\frac{\gamma(\Omega)}{2(k+h^{\vee})}}}{\prod_{j\geq 1}(1-q^j)^{\on{rk}\fing}}
\end{align*}
in the case that $k$ is non-critical,
where $\Omega$ is the Casimir element of $U(\fing)$.

In Theorem \ref{Th:vanising}
we showed that the functor 
$\KL_k\ra \W^k(\fing)\on{-Mod}$,
$M\mapsto  H^0(M)$, is exact.
However in order to 
 obtain all the irreducible positive energy representation 
 we need to extend this functor to the whole category $\mc{O}_k$.
However 
 the functor 
$\mc{O}_k\ra \W^k(\fing)\on{-Mod}$,
$M\mapsto  H^0(M)$, is {\em not} exact in general except for the case $\fing=\mf{sl}_2$ (\cite{Ara05}).
Nevertheless, we can \cite{FKW92}  modify the functor to obtain the following result.
\begin{Th}[\cite{Ara07}]\label{Th:invent}
There exists an exact functor 
\begin{align*}
\mc{O}_k\ra \W^k(\fing)\on{-Mod},\quad M\mapsto  H^0_-(M)
\end{align*}
(called the {^^ ^^ $-$"}-reduction functor in \cite{FKW92}), 
which enjoys the following properties.
\begin{enumerate}
\item $H^0_-(M(\lam))\cong \mathbb{M}(\gamma_{\bar \lam})$,
where $M(\lam)$ is the Verma module of $\affg$ with highest weight $\lam$,
and $\gamma_{\bar \lam}$ is the evaluation of $\mc{Z}(\fing)$ at the Verma module
$M_{\fing}(\bar \lam)$ of $\fing$ with highest weight $\bar \lam$.
\item $H^0_-(L(\lam))\cong \begin{cases}\mathbb{L}(\gamma_{\bar \lam})&\text{if }\bar \lam\text{ is anti-dominant
(that is, $M_{\fing}(\bar \lam)$ is simple)},\\
0&\text{otherwise.}\end{cases}$
\end{enumerate}
\end{Th}

\begin{Co}\label{Co:character}
Write $\ch L(\lam)=\sum_{\mu}c_{\lam,\mu}\ch M(\mu)$ with $c_{\lam,\mu}\in \Z$.
If  $\bar \lam$ is anti-dominant,
we have
\begin{align*}
\ch \mathbb{L}(\gamma_{\bar \lam})=\sum_{\mu}c_{\lam,\mu}\ch \mathbb
{M}(\gamma_{\bar \mu}).
\end{align*}
\end{Co}

In the case that $k$ is non-critical,  then 
it is known by Kashiwara and Tanisaki  \cite{KasTan00} that
the coefficient $c_{\lam,\mu}$ is expressed in terms of Kazhdan-Lusztig polynomials.
Since any central character of $\mc{Z}(\fing)$ can be written as $\gamma_{\bar \lam}$ with anti-dominant $\bar \lam$,
Corollary \ref{Co:character} determines the character of {\em all} the irreducible positive energy representations
of $\W^k(\fing)$ for all non-critical $k$.

On the other hand,
in the case that $k$ is critical, 
all $\mathbb{L}(\gamma_{\bar \lam})$ are one-dimensional since $\W^{-n}(\fing)$ is commutative.
This fact with Theorem \ref{Th:invent}
 can be used in the study of the critical level representations of $\affg$, see \cite{AraFie08}.
 
 The results in this section hold for arbitrary simple Lie algebra $\fing$.

\begin{Rem}
The condition $\bar \lam\in \finh^*$ is anti-dominant does not imply that
$\lam\in \widehat{\finh}^*$ is anti-dominant.
In fact this condition is satisfied by all {\em non-degenerate admissible weights} $\lam$ (see below) which are regular dominant.
\end{Rem}

\begin{Rem}
Theorem \ref{Th:invent} has been generalized in \cite{Ara08-a}.
In particular 
the character of all the simple ordinary representations (=simple positive energy representations with finite-dimensional
homogeneous spaces) has been determined  for $W$-algebras associated with  all nilpotent elements $f$ in type $A$.
\end{Rem}
\section{Kac-Wakimoto admissible representations and Frenkel-Kac-Wakimoto  conjecture}
\label{sec:KW}
We continue to assume that $\fing=\mf{sl}_n$,
but the results in this section holds for arbitrary simple Lie algebra $\fing$ as well
with appropriate modification unless otherwise stated.

\subsection{Admissible affine vertex algebras}

  Let $\affh$ be the Cartan subalgebra $\finh\+ \C K$ of $\affg$,
  $\tilde{\finh}=\finh\+ \C K \+ \C D$ the extended Cartan subalgebra,
  $\widehat{\Delta}$ the set of roots of $\affg$ in 
  $\tilde{\finh}^*=\finh^*\+ \C \Lam_0\+ \C \delta$,
  where $\Lam_0(K)=1=\delta(d)$,
  $\Lam_0(\finh+ \C D)=\delta(\finh\+ \C K)=0$,
  $\widehat{\Delta}_+$ the set of positive roots.
  $\widehat{\Delta}^{re}\subset\widehat{\Delta} $ the set of real roots,
$\widehat{\Delta}^{re}_+=\widehat{\Delta}^{re}\cap \widehat{\Delta}_+$.  
Let $\widehat{W}$ be the affine Weyl group of $\affg$.
  
 \begin{Def}[\cite{KacWak89}]
A weight $\lam\in \affh^*$ is called {\em admissible} if 
\begin{enumerate}
 \item $\lam$ is regular dominant, that is,
\begin{align*}
 \bra \lam+\rho,\alpha^{\vee}\ket \not\in -\Z_+\quad\text{for all
 }\alpha\in \widehat{\Delta}^{re}_+,
\end{align*}
 \item $\Q \widehat{\Delta}(\lam)=\Q \widehat{\Delta}^{re}$,
       where $\widehat{\Delta}(\lam)=\{\alpha\in \widehat{\Delta}^{re}\mid
       \bra \lam+\rho,\alpha^{\vee}\ket\in \Z\}$.
\end{enumerate}
 \end{Def}
The irreducible highest weight representation $L(\lam)$ of $\affg$ with highest weight $\lam\in \affh^*$ is called {\em admissible}
 if $\lam$ is admissible.
 Note that an irreducible integrable representations of $\affg$ is admissible.

Clearly, integrable representations of $\affg$ are admissible.

 For an admissible representation $L(\lam)$
 we have \cite{KacWak88}
 \begin{align}
  \on{ch} L(\lam)=\sum_{w\in \widehat{W}(\lam)}(-1)^{\ell_{\lam}(w)} \ch M(w\circ \lam)
  \label{ch:Weyl-Kac}
 \end{align}
 since $\lam$ is regular dominant,
 where $\widehat{W}(\lam)$ is the {\em integral Weyl group} (\cite{KasTan98,MooPia95}) of $\lam$,
 that is,
 the subgroup of $\widehat{W}$ generated by the reflections $s_{\alpha}$
 associated with $\alpha\in \widehat{\Delta}$
 and $w\circ \lam=w(\lam+\rho)-\rho$.
 Further  the condition (2) implies that
 $\ch L(\lam)$ is written in terms of certain theta functions.
Kac and Wakimoto  \cite{KacWak89} showed that 
admissible representations are {\em modular invariant}, that is,
the characters of admissible representations form an $SL_2(\Z)$ invariant subspace.

 Let $\lam$,
 $\mu$  be distinct admissible weights.
 Then the condition (1) implies that
 \begin{align*}
  \on{Ext}_{\affg}^1(L(\lam),L(\mu))=0.
 \end{align*}
 Further,  the following fact is known by Gorelik and Kac \cite{GorKac0905}.
 \begin{Th}[\cite{GorKac0905}]
  Let $\lam$ be admissible.
  Then $\on{Ext}_{\affg}(L(\lam),L(\lam))=0$.
 \end{Th}
 Therefore admissible representations  form a semisimple fullsubcategory
 of the category of $\affg$-modules.

 Recall that
 the simple affine vertex algebra
$V_k(\fing)$ is isomorphic 
to $L(k\Lam_0)$ as an $\affg$-module.
  \begin{lem}The following conditions are equivalent.
   \begin{enumerate}
    \item    $k\Lam_0$ is admissible.
    \item $k\Lam_0$ is regular dominant and $k\in \Q$.
	  \item $k+h^{\vee}=p/q$, $p,q\in \N$, $(p,q)=1$, $p\geq h^{\vee}=n$.
   \end{enumerate}
If this is the case, the level $k$ is called admissible for $\affg$,
and $V_k(\fing)$ is called an admissible affine vertex algebra.
  \end{lem}

  For an admissible number $k$
  let $Pr_k$ be the set of admissible weights of $\affg$ of level $k$.
  (For $\fing=\mf{sl}_n$, $Pr_k$ is  the same as the set of {\em principal admissible weights} of level $k$.)
  
\subsection{Feigin-Frenkel Conjecture and Adamovi\'{c}-Milas Conjecture}
  The following fact was conjectured by Feigin and Frenkel
  and proved for the case that $\fing=\mf{sl}_2$ by  Feigin and Malikov \cite{FeiMal97}.
      \begin{thm}[\cite{Ara09b})]\label{thm:FFconf}
    The associated variety
   $X_{V_k(\fing)}$ is contained in $\mc{N}$ if  $k$ is admissible.
     \end{thm}

     In fact the following holds.
 \begin{thm}[\cite{Ara09b}] \label{Th:admissible-orbits}
  Let $k$ be admissible, and let $q\in N$ be the denominator of $k$,
  that is, $k+h^{\vee}=p/q$, $p\in N$, $(p,q)=1$.
  Then
  \begin{align*}
   X_{V_k(\fing)}=\{x\in \fing\mid (\ad x)^{2q}=0\}=\overline{\mathbb{O}_{q}},
  \end{align*}
  where $\mathbb{O}_q$ is the nilpotent orbit corresponding to the partition
\begin{align*}
 \begin{cases}
     (n)&\text{if }q\geq n,\\ (q,q,\dots,q,s)\quad (0\leq s\leq n-1)&\text{if }q<n.
     \end{cases}
\end{align*}  
 \end{thm}
     
     The following fact was conjectured by Adamovi{\'c} and Milas \cite{AdaMil95}.
      \begin{thm}[\cite{A12-2}]\label{thm:AMconj}
       Let $k$ be admissible.
       Then an irreducible highest weight representation $L(\lam)$ is a $V_k(\fing)$-module if and only if $k\in Pr_k$.
       Hence if $M$ is a finitely generated $V_k(\fing)$-module on which
       $\affn_+$ acts locally nilpotently and $\affh$ acts locally finitely then
       $M$ is a direct sum of $L(\lam)$ with $\lam\in Pr_k$.
      \end{thm}

  \subsection{Outline of proofs of
  Theorems \ref{thm:FFconf}, \ref{Th:admissible-orbits} and \ref{thm:AMconj} }
  The idea of the proofs of Theorem  \ref{thm:FFconf} and Theorem \ref{thm:AMconj}
  is to reduce to the $\widehat{\mf{sl}_2}$-cases.

  Let $\mf{sl}_{2,i}\subset \fing$ be the copy of
  $\mf{sl}_2$ spanned by $e_i:=e_{i,i+1}$, $h_i:=e_{i,i}-e_{i+1,i+1}$, $f_i:=e_{i+1,i}$,
  and let $\mf{p}_i=\mf{sl}_{2,i}+\finb\subset \fing$, the associated minimal parabolic subalgebra.
  Then
  \begin{align*}
   \mf{p}_i=\mf{l}_i\+ \mf{m}_i,
  \end{align*}
  where $\mf{l}_i$ is the Levi subalgebra $\mf{sl}_{2,i}+\finh$,
  and $\mf{m}_i$ is the nilradical $\bigoplus\limits_{1\leq p<q\leq n \atop (p,q)\ne (i,i+1) }\C e_{p,q}$.
  
  Consider the semi-infinite cohomology $H^{\frac{\infty}{2}+0}(\mf{m}_i[t,t^{-1}],M)$.
  It is defined as a cohomology of Feigin's complex $(C(\mf{m}_i[t,t^{-1}],M),d)$ (\cite{Feu84}).
There is a natural vertex algebra homomorphism
 \begin{align}
V^{k_i}(\mf{sl}_2)\ra H^{\frac{\infty}{2}+0}(\mf{m}_i[t,t^{-1}],M),
\label{eq:va-hom}
\end{align}
where $k_i=k+n-2$, see, e.g. \cite{HosTsu91}.
Note that if $k$ is an admissible number for $\affg$ then 
$k_i$ is an admissible number for $\widehat{\mf{sl}_2}$.
 \begin{Th}[\cite{A-BGG}]\label{Th:semi-infinite-restriction}
 Let $k$ be an admissible number.
 The map \eqref{eq:va-hom}
 factors through the vertex algebra embedding
 $$V_{k_i}(\mf{sl}_2)\hookrightarrow H^{\frac{\infty}{2}+0}(\mf{m}_i[t,t^{-1}], V_k(\fing)).$$
 \end{Th}
\begin{proof}[Outline of proof of  Theorem \ref{thm:FFconf}]
First, consider the case that $\fing=\mf{sl}_2$. 
Let $N_k$ be the maximal submodule of $V^k(\fing)$,
and let $I_k$ be the image of $N_k$ in $R_{V^k(\fing)}=\C[\fing^*]$,
so that $R_{V_k(\fing)}=\C[\fing^*]/I_k$.
It is known by Kac and Wakimoto \cite{KacWak88}
that $N_k$ is generated by a singular vector, say $v_k$.
The projection formula \cite{MalFeuFuk86}
implies that the image $[v_k]$ of $v_k$ in $I_k$ is nonzero.
Since $[v_k]$ is a singular vector of $\C[\fing^*]$ with respect to the adjoint action of $\fing$,
Kostant's Separation Theorem implies that
$$[v_k]=e^m \Omega^n$$
for some $m,n\in \N$ up to constant multiplication,
where $\Omega=ef+fe+\frac{1}{2}h^2$.
Now suppose that $X_{V_k(\fing)}\not\subset \mc{N}$
and let
  $\lam\in X_{V_k(\fing)}\backslash \mc{N}$,
  so that
$\Omega(\lam)\ne 0$.
Then
$e(\lam)=0$.
Since $X_{V_k(\fing)}$ is $G$-invariant this implies that $x(\lam)=0$ for any nilpotent element $x$ of $\fing$.
Because any element of $\fing$ can be written as a sum of nilpotent elements we get that $\lam=0$.
Contradiction.

Next, consider the case that $\fing$ is general.
Note that since $X_{V_k(\fing)}$ is $G$-invariant and closed,
the condition $X_{V_k(\fing)}\subset \mc{N}$ is equivalent to that
$X_{V_k(\fing)}\cap \finh^*=\{0\}$.
Now the  complex 
structure of $C(\mf{m}_i[t,t^{-1}],V_k(\fing))$
induces the complex structure on
Zhu's $C_2$-algebra 
$R_{C(\mf{m}_i[t,t^{-1}],V_k(\fing))}$.
The embedding in Theorem \ref{Th:semi-infinite-restriction}
induces a homomorphism
\begin{align*}
R_{V_{k_i}(\mf{sl}_2)}\ra H^0(R_{C(\mf{m}_i[t,t^{-1}],V_k(\fing))},d)
\end{align*}
of Poisson algebra.
Since $\Omega$ is nilpotent in 
$R_{V_{k_i}(\mf{sl}_2)}$,
so is its image $\Omega_i=e_if_i+f_ie_i+\frac{1}{2}h_i^2$
in $H^0(R_{C(\mf{m}_i[t,t^{-1}],V_k(\fing))},d)$.
It follows that $h_i^N\equiv 0\pmod{\finn_+ R_{V_k(\fing)}+\finn_- R_{V_k(\fing)}}$
in $R_{V_k(\fing)}$ for all $i=1,\dots, n-q$,
and we get that $X_{V_k(\fing)}\cap \finh^*=\{0\}$
as required.
\end{proof}
 \begin{proof}[Outline of proof of  Theorem \ref {Th:admissible-orbits}]
 The proof is done by determining
 the variety $X_{V_k(\fing)}$.
 By Theorem \ref{thm:FFconf},
  $X_{V_k(\fing)}$ is a finite union of nilpotent orbits.
  Thus it is enough to know which nilpotent element 
  orbits is contained in $X_{V_k(\fing)}$.
  On the other hand,  \eqref{eq:cirterion} says
  $ X_{V_k(\fing)}\supset \overline{G.f}$
  if and only $H^0_f(V_k(\fing))\ne 0$.
  Thus,
  it is sufficient to compute the  character of 
   $H^0_f(V_k(\fing))$.
   This is in fact possible since we know the explicit formula \eqref{ch:Weyl-Kac} of the character of 
   $V_k(\fing)$,
   and thanks of the vanishing theorem \eqref{eq;vanishing-general}
and the Euler-Poincar\'{e} principle.
 \end{proof}
 \begin{proof}[Outline of proof of
Theorem \ref{thm:AMconj}]
Let $L(\lam)$ be a $V_k(\fing)$-module.
Then,  
the space 
$H^{\frac{\infty}{2}+i}(\mf{m}_i[t,t^{-1}],L(\lam))$,
$i\in \Z$, is naturally a  $H^{\frac{\infty}{2}+i}(\mf{m}_i[t,t^{-1}],V_k(\fing))$-module.
By Theorem \ref{Th:semi-infinite-restriction},
this means that 
$H^{\frac{\infty}{2}+i}(\mf{m}_i[t,t^{-1}],L(\lam))$ is in particular a module over 
the admissible affine vertex algebra $V_{k_i}(\mf{sl}_2)$.
Therefore 
Theorem \ref{thm:AMconj} for $\fing=\mf{sl}_2$ that was established by  Adamovi{\'c} and Milas \cite{AdaMil95} implies 
that $H^{\frac{\infty}{2}+i}(\mf{m}_i[t,t^{-1}],L(\lam))$  must be a direct sum of 
admissible representations of $\widehat{\mf{sl}}_2$. This information is sufficient to 
conclude that $L(\lam)$ is admissible.

Conversely,
suppose that   $L(\lam)$ is an admissible representation of level $k$.
If $L(\lam)$ is integrable over $\fing$,
then 
it has been already proved by
Frenkel and Malikov \cite{FreMal97}
that $L(\lam)$ is a $V_k(\fing)$-module.
But then
 an affine analogue of Duflo-Joseph Lemma \cite[Lemma 2.6]{A12-2}
 implies that
 this is true for a general admissible representation as well.
\end{proof}
 \subsection{Lisse property of $W$-algebras}
 An admissible number $k$ is called   {\em non-degenerate} if 
 $X_{V_k(\fing)}=\mc{N}$.
 By Theorem \ref{Th:admissible-orbits},
this condition is equivalent to that 
\begin{align*}
k+n=\frac{p}{q},\ p,q\in \N,\ (p,q)=1,\ p\geq n,\ q\geq n.
\end{align*}

The following assertion follows immediately from Corollary \ref{Co:lisse}.
 \begin{thm}[\cite{Ara09b}]\label{Th:lisse}
 Let $k$ be a non-degenerate admissible number.
 Then the $W$-algebra $\W_k(\fing)$ is lisse.
 \end{thm}
 
 \subsection{Minimal models of $W$-algebras}
 A vertex algebra $V$ is called {\em rational} if any $V$-module is completely reducible.
 To a lisse and rational conformal vertex algebra $V$ one can associate {\em rational $2d$ conformal field theory},
 and
in particular,
 the category
 $V\on{-Mod}$ of $V$-modules  forms \cite{Hua08rigidity} a {\em modular tensor category} \cite{Bakalov:2001kq},
 as in the case of the category of integrable representation of $\affg$ at a  positive level
and the category of {\em minimal series representations} \cite{BPZ84} of the Virasoro algebra.
 
 An admissible weight $\lam$ is called {\em non-degenerate} if
 $\bar \lam$ is anti-dominant.
 Let $Pr_{k}^{non-deg}$
 be the set of non-degenerate admissible weights of level $k$ of $\affg$.
  It is known \cite{FKW92} that $Pr_{k}^{non-deg}$ is non-empty if and only if $k$ is non-degenerate.
 
By Theorem \ref{Th:invent},
for $\lam\in Pr^k$,
$H^0_-(L(\lam))$ is a (non-zero) simple $\W^k(\fing)$-module if and only of $\lam\in Pr_{k}^{non-deg}$,
and $H^0_-(L(\lam))\cong H^0_-(L(\mu))$ if and only if $\mu\in W\circ \lam$ for $\lam,\mu\in Pr_{k}^{non-deg}$.
 
 Let $[Pr_{k}^{non-deg}]=Pr_{k}^{non-deg}/\sim$ ,
 where $\lam\sim \mu\iff \mu\in W\circ \lam$.
 It is known \cite{FKW92} that
 we have a bijection
 \begin{align*}
(\widehat{P}_+^{p-n}\times \widehat{P}_+^{q-n})/\Z_n\isomap [Pr_{k}^{non-deg}],
\quad [(\lam,\mu)]\mapsto [\bar \lam-\frac{p}{q}(\bar \mu+\rho)+k\Lam_0].
\end{align*}
Here $k+n=p/q$ as before, 
$\widehat{P}_+^{k}$ is the set of integral dominant weights of level $k$ of $\affg$,
 the cyclic group $\Z_n$ acts diagonally on $\widehat{P}_+^{p-n}\times \widehat{P}_+^{q-n}$ as the Dynkin automorphism,
 and $\rho=\frac{1}{2}\sum_{\alpha\in \Delta_+}\alpha$.

 The following assertion was conjectured by Frenkel, Kac and Wakimoto \cite{FKW92}.
 \begin{thm}[\cite{A2012Dec}]\label{Th:rational}
  Let $k$ be a non-degenerate admissible number.
   Then the simple $W$-algebra $\W_k(\fing)$ is rational,
   and $\{\mathbb{L}(\gamma_{\bar \lam})=H^0_-(L(\lam))\mid \lam\in [Pr_{k}^{non-deg}]\}$
   forms the complete set of isomorphism classes of simple $\W_k(\fing)$-modules.
  \end{thm}
   In the case that $\fing=\mf{sl}_2$,
   Theorems \ref{Th:lisse} and \ref{Th:rational} have been proved in \cite{BeiFeiMaz,Wan93},
and
  the above representations are exactly the minimal series representations of the Virasoro algebra.
 
 The representations  $$\{\mathbb{L}(\gamma_{\bar \lam})\mid  \lam\in [Pr_{k}^{non-deg}]\}$$
are called the {\em minimal series representations} of $\W^k(\fing)$,
and if $k+n=p/q$, $p,q\in \N$, $(p,q)=1$, $p,q,\geq n$,
then the rational $W$-algebra $\W_k(\fing)$ is called the {\em $(p,q)$-minimal model} of  $\W^k(\fing)$.
Note that the $(p,q)$-minimal model and the $(q,p)$-minimal model are isomorphic due to
the duality, see Corollary \ref{Co:duality}.

\begin{proof}[Outline of the proof of Theorem \ref{Th:rational}]
Let $k$ be a non-degenerate admissible number.
We have
 $$H^0(V_k(\fing))\cong \W_k(\fing)$$
 by \cite{Ara07}.
 Hence by Theorem \ref{Th:vanising} (3)
 $$\on{Zhu}(\W_k(\fing))=\on{Zhu}(H^0(V_k(\fing)))=H^0(\on{Zhu}(V_k(\fing)).$$
From this together with Theorem \ref{thm:AMconj},
it is not too difficult to obtain the classification is the simple $\W_k(\fing)$-modules
as stated in Theorem  \ref{Th:rational}.
One sees that the extensions between simple modules are trivial using
the linkage principle that follows from Theorem \ref{Th:invent}.
\end{proof}

  \begin{Rem}
  \begin{enumerate}
\item  We have $\W_k(\fing)=\mathbb{L}(\gamma_{-(k+n)\rho})$ for a non-degenerate admissible number  $k$.
(Note that $k\Lam_0\not\in Pr_k^{non-deg}$.)
\item Let $\lam\in Pr_k$.
From Corollary \ref{Co:character}
and \eqref{ch:Weyl-Kac},
we get
\begin{align}
\on{ch}\mathbb{L}(\gamma_{\bar \lam})=\sum_{w\in \widehat{W}(\lam)} \epsilon(w)\ch \mathbb{M}(\gamma_{\overline{w\circ \lam}}).
\label{eq:FKW-ch-formula}
\end{align}
This  was conjectured by \cite{FKW92}.
\item  When it is trivial (that is, equals to $\C$),
$\W_k(\fing)$ is obviously lisse
and rational.
    This happens 
    if and only if 
    $\W_k(\fing)$ is the $(n,n+1)$-minimal model
   (=the $(n+1,n)$-minimal model). 
In this case the character formula \eqref{eq:FKW-ch-formula}
for  $\W_k(\fing)=\mathbb{L}(\gamma_{\bar \lam})$,
$\lam=-(k+n)\rho+k\Lam_0$, gives the following {\em denominator formula}:
  \begin{align*}
\sum_{w\in \widehat{W}(\lam)}\epsilon(w)q^{\frac{(\overline{w\circ \lam},\overline{w\circ \lam}+2\rho)}{2(k+n)}}
=\prod_{j=1}^{n-1}(1-q^j)^{n-1}.
\end{align*}
In the case that $\fing=\mf{sl}_2$,
we get the denominator formula for the Virasoro algebra, which is  identical to  {\em Euler's pentagonal identity}.
\item 
As a generalization of the GKO construction \cite{GodKenOli85}
it has been conjectured \cite{KacWak90} that 
the $(p,q)$-minimal model of $\W^k(\fing)$, with $p>q$, is isomorphic to
the commutant of $V_{l+1}(\fing)$ inside $V_{l}(\fing)\* V_1(\fing)$,
where $l+n=q/(p-q)$.
(Note that $V_{l}(\fing)$ and $V_{l+1}(\fing)$ are admissible.)
This conjecture has been proved in \cite{ALY16} for the special case that $(p,q)=(n+1,n)$.

A similar conjecture exists in the case that $\fing$ is simply laced.
\item  The existence of rational and lisse $W$-algebras
 has been conjectured for general $W$-algebras $\W^k(\fing,f)$ by Kac and Wakimoto \cite{KacWak08}.
This has been proved in \cite{ArationalII} in part including all the cases in type $A$.
See \cite{Kawa15,AM15}
for a recent development in the classification problem of rational and lisse $W$-algebras.
 
\end{enumerate}
  \end{Rem}


\end{document}